\documentclass{birkjour}

\usepackage{graphicx}
\usepackage{amsmath}
\usepackage{amscd}
\usepackage{amsfonts}
\usepackage{amssymb}
\usepackage{color}
\usepackage[mathscr]{eucal}
\usepackage{mathtools}
\usepackage{stmaryrd}
\usepackage[pagebackref,hypertexnames=false, colorlinks, citecolor=red,linkcolor=blue, urlcolor=red]{hyperref}

\numberwithin{equation}{section}

\newtheorem{theorem}{Theorem}[section]
\newtheorem{lemma}[theorem]{Lemma}
\newtheorem{proposition}[theorem]{Proposition}

\newtheorem{corollary}[theorem]{Corollary}

\theoremstyle{definition}
\newtheorem{example}[theorem]{Example}
\newtheorem{remark}[theorem]{Remark}

\newtheorem{problem}[theorem]{Problem}

\renewcommand{\proofname}{\em\textbf{Proof.}}

\newcommand{\be}{\begin{equation}}
\newcommand{\ee}{\end{equation}}
\newcommand{\bes}{\begin{equation*}}
\newcommand{\ees}{\end{equation*}}

\newcommand{\cA}{\mathcal{A}}
\newcommand{\cB}{\mathcal{B}}
\newcommand{\cC}{\mathcal{C}}
\newcommand{\cD}{\mathcal{D}}
\newcommand{\cE}{\mathcal{E}}
\newcommand{\cF}{\mathcal{F}}
\newcommand{\cG}{\mathcal{G}}
\newcommand{\cH}{\mathcal{H}}
\newcommand{\cI}{\mathcal{I}}
\newcommand{\cJ}{\mathcal{J}}
\newcommand{\cK}{\mathcal{K}}
\newcommand{\cL}{\mathcal{L}}
\newcommand{\cM}{\mathcal{M}}
\newcommand{\cN}{\mathcal{N}}
\newcommand{\cO}{\mathcal{O}}

\newcommand{\cS}{\mathcal{S}}
\newcommand{\cT}{\mathcal{T}}

\newcommand{\cW}{\mathcal{W}}

\newcommand{\bB}{\mathbb{B}}
\newcommand{\bC}{\mathbb{C}}
\newcommand{\bD}{\mathbb{D}}
\newcommand{\bM}{\mathbb{M}}
\newcommand{\bN}{\mathbb{N}}

\newcommand{\bR}{\mathbb{R}}

\newcommand{\bT}{\mathbb{T}}
\newcommand{\bZ}{\mathbb{Z}}

\newcommand{\diag}{\operatorname{diag}}

\newcommand{\mlt}{\operatorname{Mult}}

\newcommand{\ol}{\overline}

\newcommand{\spn}{\operatorname{span}}

\newcommand{\UCP}{\operatorname{UCP}}

\newcommand{\Wmin}[1]{\cW^{\textup{min}}_{#1}}
\newcommand{\Wmax}[1]{\cW^{\textup{max}}_{#1}}

\newcommand{\re}{\operatorname{Re}}
\newcommand{\im}{\operatorname{Im}}

\begin{document}

\title[Dilation theory]{Dilation theory: a guided tour}

\author{Orr Moshe Shalit}
\address{Faculty of Mathematics\\
Technion - Israel Institute of Technology\\
Haifa\; 3200003\\
Israel}
\email{oshalit@technion.ac.il}

\thanks{Partially supported by ISF Grant no. 195/16}
\begin{abstract} 
Dilation theory is a paradigm for studying operators by way of exhibiting an operator as a compression of another operator which is in some sense well behaved. For example, every contraction can be dilated to (i.e., is a compression of) a unitary operator, and on this simple fact a penetrating theory of non-normal operators has been developed. 
In the first part of this survey, I will leisurely review key classical results on dilation theory for a single operator or for several commuting operators, and sample applications of dilation theory in operator theory and in function theory.  
Then, in the second part, I will give a rapid account of a plethora of variants of dilation theory and their applications. 
In particular, I will discuss dilation theory of completely positive maps and semigroups, as well as the operator algebraic approach to dilation theory.  
In the last part, I will present relatively new dilation problems in the noncommutative setting which are related to the study of matrix convex sets and operator systems, and are motivated by applications in control theory. 
These problems include dilating tuples of noncommuting operators to tuples of commuting normal operators with a specified joint spectrum. 
I will also describe the recently studied problem of determining the optimal constant $c = c_{\theta,\theta'}$, such that every pair of unitaries $U,V$ satisfying $VU = e^{i\theta} UV$ can be dilated to a pair of $cU', cV'$, where $U',V'$ are unitaries that satisfy the commutation relation $V'U' = e^{i\theta'} U'V'$. 
The solution of this problem gives rise to a new and surprising application of dilation theory to the continuity of the spectrum of the almost Mathieu operator from mathematical physics. 
\end{abstract}

\maketitle

\newpage
\tableofcontents 

{\bf Dilation theory} is a collection of results, tools, techniques, tricks, and points of view in operator theory and operator algebras, that fall under the unifying idea that one can learn a lot about an operator (or family of operators, or a map, etc.) by viewing it as ``a part of" another, well understood operator. 
This survey on dilation theory consists of three parts. 
The first part is a stand-alone exposition aimed at giving an idea of what dilation theory is about by describing several representative results and applications that are, in my opinion, particularly interesting. 
The climax of the first part is in Section \ref{sec:Pick}, where as an application of the material in the first three sections, we see how to prove the Pick interpolation theorem using the commutant lifting theorem. 
Anyone who took a course in operator theory can read Part 1. 

Out of the theory described in the first part, several different research directions have developed. 
The second part of this survey is an attempt to give a quick account of some of these directions. 
In particular, we will cover Stinespring's dilation theorem, and the operator algebraic approach to dilation theory that was invented by Arveson. 
This survey up to Section \ref{sec:opalg} contains what everyone working in dilation theory and/or nonselfadjoint operator algebras should know. 
I will also cover a part of the dilation theory of CP-semigroups, and take the opportunity to report on my work with Michael Skeide, which provides our current general outlook on the subject.

In the third and last part I will survey some recent dilation results in the noncommutative setting, in particular those that have been motivated by the study of matrix convex sets. 
Then I will focus on my recent joint work with Malte Gerhold, where we study the problem of dilating $q$-commuting unitaries. 
Experts on dilation theory can read the last three sections in this survey independently. 

I made an effort to include in this survey many applications of dilation theory. 
The theory is interesting and elegant in itself, but the applications give it its vitality. 
I believe that anyone, including experts in dilation theory, will be able to find in this survey an interesting application which they have not seen before.  

Some results are proved and others are not. 
For some results, only an idea of the proof is given. 
The guiding principle is to include proofs that somehow together convey the essence or philosphy of the field, so that the reader will be able to get the core of the theory from this survey, and then be able to follow the references for more. 

As for giving references: this issue has given me a lot of headaches. 
On the one hand, I would like to give a historically precise picture, and to give credit where credit is due. 
On the other hand, making too big of a fuss about this might result in an unreadable report, that looks more like a legal document than the inviting survey that I want this to be. 
Some results  have been rediscovered and refined several times before reaching their final form. 
Who should I cite? 
My solution was to always prefer the benefit of the reader. 
For ``classical" results, I am very happy to point the reader to an excellent textbook or monograph, that contains a proof, as well as detailed references and sometimes also historical remarks. 
I attach a specific paper to a theorem only when it is clear-cut and useful to do so. 
In the case of recent results, I sometimes give all relevant references and an account of the historical development, since this appears nowhere else. 

There are other ways to present dilation theory, and by the end of the first section the reader will find references to several alternative sources. 
Either because of my ignorance, or because I had to make choices, some things were left out. 
I have not been able to cover all topics that could fall under the title, nor did I do full justice to the topics covered. 
After all this is just a survey, and that is the inevitable nature of the genre.

\subsubsection*{Acknowledgements.} 
This survey paper grew out of the talk that I gave at the International Workshop on Operator Theory and its Applications (IWOTA) that took place in the Instituto Superior T\'{e}cnico, Lisbon, Portugal, in July 2019. 
I am grateful to the organizers of IWOTA 2019 for inviting me to speak in this incredibly successful workshop, and especially to Am\'{e}lia Bastos, for inviting me to contribute to these proceedings. 
I used a preliminary version of this survey as lecture notes for a mini-course that I gave in the workshop Noncommutative Geometry and its Applications, which took place in January 2020, in NISER, Bhubaneswar, India.  
I am grateful to the organizers Bata Krishna Das, Sutanu Roy and Jaydeb Sarkar, for the wonderful hospitality and the opportunity to speak and organize my thoughts on dilation theory. 
I also owe thanks to Michael Skeide and to Fanciszek Szafraniec, for helpful feedback on preliminary versions. 
Finally, I wish to thank an anonymous referee for several useful comments and corrections.

\part{An exposition of classical dilation theory}
\section{The concept of dilations}\label{sec:concept}
The purpose of this introductory section is to present the notion of {\em dilation}, and to give a first indication that this notion is interesting and can be useful. 

Let $\cH$ be a Hilbert space, and $T \in B(\cH)$ be a {\bf{contraction}}, that is, $T$ is an operator such that $\|T\|\leq 1$. 
Then $I - T^*T \geq 0$, and so we can define $D_T = \sqrt{I - T^*T}$. 
Halmos \cite{Hal50} observed that the simple construction 
\[
U = \begin{pmatrix} T & D_{T^*} \\ D_T & -T^* \end{pmatrix}
\]
gives rise to a unitary operator on $\cH \oplus \cH$. 
Thus, every contraction $T$ is a {\bf{compression}} of a unitary $U$, meaning that 
\be\label{eq:1dil}
T = P_\cH U \big|_\cH, 
\ee
where $P_\cH$ denotes the orthogonal projection of $\cH \oplus \cH$ onto $\cH \oplus \{0\} \cong \cH$. 
In this situation we say that $U$ is a {\bf{dilation}} of $T$, and below we shall write $T \prec U$ to abbreviate that $U$ is a dilation of $T$. 

This idea can be pushed further. 
Let $\cK =\cH^{N+1} = \cH \oplus \cdots \oplus \cH$ be the $(N+1)$th-fold direct sum of $\cH$ with itself, and consider the following $(N+1) \times (N+1)$ operator matrix
\be\label{eq:U}
U = \begin{pmatrix} {T} & 0 & 0 & \cdots & 0 &  {D_{T^*}} \\ {D_T} & 0 & 0 & \cdots & 0 & {-T^*} \\  0 & {I} & 0 & \cdots & 0 & 0 \\  0 & 0 & {I} & & & 0  \\ \vdots &  \vdots &  & \ddots & & \vdots  \\ 0 & 0 & \cdots & 0 & {I} & 0 \end{pmatrix} .
\ee
Egerv\'{a}ry \cite{Ege54} observed that $U$ is unitary on $B(\cK)$, and moreover, that 
\be\label{eq:NdilBlock}
U^k = \begin{pmatrix} T^k & * \\ * & * \end{pmatrix} \,\, , \,\, k = 1,2,\ldots, N, 
\ee
in other words, if we identify $\cH$ with the first summand of $\cK$, then 
\be\label{eq:Ndil}
p(T) = P_\cH p(U)\big|_\cH
\ee
for every polynomial $p$ of degree at most $N$. Such a dilation was called an {\bf{$N$-dilation}} in \cite{LS14}. 
Thus, an operator $U$ satisfying \eqref{eq:1dil} might be referred to as a {\bf{$1$-dilation}}, however, the recent ubiquity of $1$-dilations has led me to refer to it simply as a {\bf{dilation}}.

We see that, in a sense, every contraction is a ``part of" a unitary operator. 
Unitaries are a very well understood class of operators, and contractions are as general a class as one can hope to study. Can we learn anything interesting from the dilation picture?

\begin{theorem}[von Neumann's inequality \cite{vN51}]\label{thm:vNinequality}
Let $T$ be a contraction on some Hilbert space $\cH$. 
Then, for every polynomial $p \in \bC[z]$, 
\be\label{eq:vNineq}
\|p(T)\| \leq \sup_{|z| = 1}|p(z)| . 
\ee
\end{theorem}
\begin{proof}
Suppose that the degree of $p$ is $N$. 
Construct $U$ on $\cK = \cH \oplus \cdots \oplus \cH$ (direct sum $N+1$ times) as in \eqref{eq:U}. 
Using \eqref{eq:Ndil}, we find that 
\[
\|p(T)\| = \|P_\cH p(U) \big|_\cH\| \leq \|p(U)\| = \sup_{z \in \sigma(U)} |p(z)| ,
\]
by the spectral theorem, where $\sigma(U)$ denotes the spectrum of $U$. 
Since for every unitary $U$, the spectrum $\sigma(U)$ is contained in the unit circle $\bT=\{z\in \bC : |z|=1\}$, the proof is complete. 
\end{proof}

\begin{remark}
The above proof is a minor simplification of the proof of von Neumann's inequality due to Sz.-Nagy \cite{SzN53}, which uses the existence of a unitary {\bf{power dilation}} (see the next section). 
The simplification becomes significant when $\dim \cH < \infty$, because then $U$ is a unitary which acts on a finite dimensional space, and the spectral theorem is then a truly elementary matter. 
Note that even the case $\dim \cH = 1$ is not entirely trivial: in this case von Neumann's inequality is basically the maximum modulus principle in the unit disc; as was observed in \cite{ShalitMonthly}, this fundamental results can be proved using linear algebra! 
\end{remark}

\begin{remark}\label{rem:mat_val_pol}
A {\bf{matrix valued polynomial}} is a function $z \mapsto p(z) \in M_n$ where $p(z) = \sum_{k=0}^N A_k z^k$, and $A_1, \ldots, A_k \in M_n = M_n(\bC)$ (the $n \times n$ matrices over $\bC$). 
Equivalently, a matrix valued polynomial is an $n \times n$ matrix of polynomials $ p = (p_{ij})$, where the $ij$th entry is given by $p_{ij}(z) = \sum_{k=0}^N (A_k)_{ij} z^k$. 
If $T \in B(\cH)$, then we may evaluate a matrix valued polynomial $p$ at $T$ by setting $p(T) = \sum_{k=0}^N A_k \otimes T^k \in M_n \otimes B(\cH)$, or, equivalently, $p(T)$ is the $n\times n$ matrix over $B(\cH)$ with $ij$th entry equal to $p_{ij}(T)$ (the operator $p(T)$ acts on the direct sum of $\cH$ with itself $n$ times). 
It is not hard to see that if $p$ is a matrix valued polynomial with values in $M_n$ and $T$ is a contraction, then the inequality \eqref{eq:vNineq} still holds but with $|p(z)|$ replaced by $\|p(z)\|_{M_n}$: first one notes that it holds for unitary operators, and then one obtains it for a general contraction by the dilation argument that we gave. 
\end{remark}

The construction \eqref{eq:U} together with Theorem \ref{thm:vNinequality} illustrate what {{\em dilation theory}} is about and how it can be used: every object in a general class of objects (here, contractions) is shown to be ``a part of" an object in a smaller, better behaved class (here, unitaries); the objects in the better behaved class are well understood (here, by the spectral theorem), and thus the objects in the general class inherit some good properties. 
The example we have just seen is an excellent one, since proving von Neumann's inequality for non-normal contractions is not trivial. 
The simple construction \eqref{eq:U} and its multivariable generalizations have other applications, for example they lead to concrete cubature formulas and operator valued cubature formulas \cite[Section 4.3]{LS14}. 
In Section \ref{sec:Pick} we will examine in detail a deeper application of dilation theory -- an operator theoretic proof of Pick's interpolation theorem. 
Additional applications are scattered throughout this survey. 

Before continuing I wish to emphasize that ``dilation theory" and even the word ``dilation" itself mean different things to different people. 
As we shall see further down the survey, the definition changes, as do the goals and the applications. 
Besides the expository essay \cite{LS14}, the subject is presented nicely in the surveys \cite{AMDil} and \cite{ArvDil}, and certain aspects are covered in books, e.g., \cite{AMBook}, \cite{PauBook}, and \cite{Pis96} (the forthcoming book \cite{BhatBhattacharyya} will surely be valuable when it appears). 
Finally, the monograph \cite{nagy} is an indispensable reference for anyone who is seriously interested in dilation theory of contractions.

\section{Classical dilation theory of a single contraction}\label{sec:single}

\subsection{Dilations of a single contraction}\label{subsec:single}

It is quite natural to ask whether one can modify the construction \eqref{eq:U} so that \eqref{eq:NdilBlock} holds for all $k \in \bN$, and not just up to some power $N$. 
We will soon see that the answer is affirmative. 
Let us say that an operator $U \in B(\cK)$ is a {\bf{power dilation}} of $T \in B(\cH)$, if $\cH$ is a subspace of $\cK$ and if $T^k = P_\cH U^k \big|_\cH$ for all $k \in \bN = \{0,1,2,\ldots\}$. 
The reader should be warned that this is not the standard terminology used, in the older literature one usually finds the word {\bf{dilation}} used to describe what we just called a power dilation, whereas the concept of $N$-dilation does not appear much (while in the older-older literature one can again find {\em power dilation}). 

In this and the next sections I will present what I and many others refer to as {\em classical dilation theory}. 
By and large, this means the theory that has been pushed and organized by Sz.-Nagy and Foias (though there are many other contributors), and appears in the first chapter of \cite{nagy}. 
The book \cite{nagy} is the chief reference for classical dilation theory. 
The proofs of most of the results in this and the next section, as well as references, further comments and historical remarks can be found there.

\begin{theorem}[Sz.-Nagy's unitary dilation theorem \cite{SzN53}]\label{thm:unidil}
Let $T$ be a contraction on a Hilbert space $\cH$. 
Then there exists a Hilbert space $\cK$ containing $\cH$ and a unitary $U$ on $\cK$, such that 
\be\label{eq:power_dil}
T^k = P_\cH U^k \big|_\cH \,\,\,\, , \,\,\,\, \textup{ for all } k = 0,1,2,\ldots 
\ee
Moreover, $\cK$ can be chosen to be {\bf{minimal}} in the sense that the smallest reducing subspace for $U$ that contains $\cH$ is $\cK$. 
If $U_i \in B(\cK_i)$ ($i=1,2$) are two minimal unitary dilations, then there exists a unitary $W : \cK_1 \to \cK_2$ acting as the identity on $\cH$ such that $U_2W = WU_1$. 
\end{theorem} 
One can give a direct proof of existence of the unitary dilation by just writing down an infinite operator matrix $U$ acting on $\ell^2(\bZ, \cH) = \oplus_{n \in \bZ} \cH$, similarly to \eqref{eq:U}. 
A minimal dilation is then obtained by restricting $U$ to the reducing subspace $\bigvee_{n \in \bZ} U^n \cH$ (the notation means the closure of the span of the subspaces $U^n \cH$). 
The uniqueness of the minimal unitary dilation is then a routine matter. 
We will follow a different path, that requires us to introduce another very important notion: {\bf{the minimal isometric dilation}}. 

Before presenting the isometric dilation theorem, it is natural to ask whether it is expected to be of any use. 
A unitary dilation can be useful because unitaries are ``completely understood" thanks to the spectral theorem. 
Are isometries well understood? The following theorem shows that, in a sense, isometries are indeed very well understood. 

For a Hilbert space $\cG$, we write $\ell^2(\bN, \cG)$ for the direct sum $\oplus_{n \in \bN} \cG$. 
The {\bf{unilateral shift of multiplicity $\dim \cG$}} (or simply {the shift}) is the operator $S : \ell^2(\bN,\cG) \to \ell^2(\bN,\cG)$ given by 
\[
S(g_0, g_1, g_2, \ldots) = (0, g_0, g_1, \ldots). 
\]
The space $\cG$ is called the {\bf{multiplicity space}}. 
Clearly, the shift is an isometry. 
Similarly, the {\bf{bilateral shift}} on $\ell^2(\bZ,\cG)$ is defined to be the operator 
\[
U(\ldots, g_{-2}, g_{-1}, \boxed{g_0}, g_1, g_2, \ldots) = (\ldots g_{-3}, g_{-2}, \boxed{g_{-1}}, g_0, g_1, \ldots)
\]
where we indicate with a box the element at the $0$th summand of $\ell^2(\bZ,\cG)$. 

\begin{theorem}[Wold decomposition]\label{thm:wold}
Let $V$ be an isometry on a Hilbert space $\cH$. 
Then there exists a (unique) direct sum decomposition $\cH = \cH_s \oplus \cH_u$ such that $\cH_s$ and $\cH_u$ are reducing for $V$, and such that $V\big|_{\cH_s}$ is unitarily equivalent to a unilateral shift and $V\big|_{\cH_u}$ is unitary. 
\end{theorem} 

For the proof, the reader has no choice but to define $\cH_u = \cap_{n\geq 0}V^n \cH$. 
Once one shows that $\cH_u$ is reducing, it remains to show that $V\big|_{\cH_u^\perp}$ is a unilateral shift. Hint: the multiplicity space is $\cH \ominus V\cH$, (this suggestive notation is commonly used in the theory; it means ``the orthogonal complement of $V\cH$ inside $\cH$", which is in this case just $(V\cH)^\perp$). 

\begin{theorem}[Sz.-Nagy's isometric dilation theorem]\label{thm:isodil}
Let $T$ be a contraction on a Hilbert space $\cH$. 
Then there exists a Hilbert space $\cK$ containing $\cH$ and an isometry $V$ on $\cK$, such that 
\be\label{eq:coext}
T^* = V^* \big|_\cH 
\ee
and in particular, $V$ is a power dilation of $T$. 
Moreover, $\cK$ can be chosen to be {\bf{minimal}} in the sense that the minimal invariant subspace for $V$ that contains $\cH$ is $\cK$. 
If $V_i \in B(\cK_i)$ ($i=1,2$) are two minimal isometric dilations, then there exists a unitary $W : \cK_1 \to \cK_2$ acting as the identity on $\cH$ such that $V_2W = WV_1$. 
\end{theorem} 
\begin{proof}
Set $D_T = (I-T^*T)^{1/2}$ --- this is the so called {\bf{defect operator}}, which measures how far $T$ is from being an isometry --- and let $\cD = \ol{D_T(\cH)}$. 
Construct $\cK = \cH \oplus \cD \oplus \cD \oplus \ldots$, in which $\cH$ is identified as the first summand. 
Now we define, with respect to the above decomposition of $\cK$, the block operator matrix
\[
V = \begin{pmatrix}
T & 0 & 0 & 0 & 0 &\cdots \\
D_T & 0 & 0 & 0 & 0 &\cdots\\
0 & I_\cD & 0 & 0 & 0 &\cdots \\ 
0 & 0 & I_\cD & 0 & 0 & \\
0 & 0 & 0 & I_\cD & 0 & \\
\vdots & \vdots & & & \ddots &\ddots \\
\end{pmatrix}
\]
which is readily seen to satisfy $V^*\big|_\cH = T^*$. 
That $V$ is an isometry, and the fact that $\cK = \bigvee_{n \in \bN}V^n\cH$, can be proved directly and without any pain. 

The uniqueness is routine, but let's walk through it for once. If $V_i \in B(\cK_i)$ are two minimal isometric dilations, then we can define a map $W$ on $\spn\{V_1^n h : n \in \bN, h \in \cH\} \subseteq \cK_1$ by first prescribing
\[
W V_1^n h =  V_2^n h \in \cK_2. 
\]
This map preserves inner products: assuming that $m \leq n$, and using the fact that $V_i$ is an isometric dilation of $T$, we see
\[
\langle W V_1^m g, W V_1^n h \rangle = \langle V_2^m g, V_2^n h \rangle = \langle  g, V_2^{n-m} h \rangle = \langle g, T^{n-m}h \rangle = \langle V_1^m g, V_1^n h \rangle. 
\]
The map $W$ therefore well defines an isometry from the dense subspace $\spn\{V_1^n h : n \in \bN, h \in \cH\} \subseteq \cK_1$ onto the dense subspace $\spn\{V_2^n h : n \in \bN, h \in \cH\} \subseteq \cK_2$, and therefore extends to a unitary which, by definition, intertwines $V_1$ and $V_2$. 
\end{proof}

Note that the minimal isometric dilation is actually a {\bf{coextension}}: $T^* = V\big|_\cH^*$. 
A coextension is always a power dilation, so $T^n = P_\cH V^n \big|_\cH$ for all $n \in \bN$, but, of course, the converse is not true (see Theorem \ref{thm:sarason} and \eqref{eq:sarason} below for the general form of a power dilation). 
Note also that the minimality requirement is more stringent than the minimality required from the minimal unitary dilation. 
The above proof of existence together with the uniqueness assertion actually show that every isometry which is a power dilation of $T$ and is minimal in the sense of the theorem, is in fact a coextension. 

Because the minimal isometric dilation $V$ is actually a coextension of $T$, the adjoint $V^*$ is a coisometric extension of $T^*$. Some people prefer to speak about {\bf{coisometric extensions}} instead of isometric coextensions. 

Once the existence of an isometric dilation is known, the existence of a unitary dilation follows immediately, by the Wold decomposition. Indeed, given a contraction $T$, we can dilate it to an isometry $V$. 
Since $V \cong S \oplus V_u$, where $S$ is a unilateral shift and $V_u$ is a unitary, we can define a unitary dilation of $T$ by $U \oplus V_u$, where $U$ is the bilateral shift of the same multiplicity as $S$. 
This proves the existence part of Theorem \ref{thm:unidil}; the uniqueness of the minimal unitary dilation is proved as for the minimal isometric one.

\subsection{A glimpse at some applications of single operator dilation theory}\label{subsec:glimpse}

The minimal unitary dilation of a single contraction can serve as the basis of the development of operator theory for non-normal operators. 
This idea was developed to a high degree by Sz.-Nagy and Foias and others; see the monograph \cite{nagy} (\cite{nagy} also contains references to alternative approaches to non-normal operators, in particular the theories of de Branges-Rovnyak, Lax-Phillips, and Livsi\v{c} and his school). 
The minimal unitary dilation can be used to define a refined functional calculus on contractions, it can be employed to analyze one-parameter semigroups of operators, it provides a ``functional model" by which to analyze contractions and by which they can be classified, and it has led to considerable progress in the study of invariant subspaces. 

To sketch just one of the above applications, let us briefly consider the functional calculus (the following discussion might be a bit difficult for readers with little background in function theory and measure theory; they may skip to the beginning of Section \ref{sec:commuting} without much loss). 
We let $H^\infty = H^\infty(\bD)$ denote the algebra of bounded analytic functions on the open unit disc $\bD$. 
Given an operator $T \in B(\cH)$, we wish to define a functional calculus $f \mapsto f(T)$ for all $f \in H^\infty$. 
If the spectrum of $T$ is contained in $\bD$, then we can apply the holomorphic functional calculus to $T$ to define a homomorphism $f \mapsto f(T)$ from the algebra $\cO(\bD)$ of analytic functions on $\bD$ into $B(\cH)$. 
In fact, if $f \in \cO(\bD)$ and $\sigma(T) \subset \bD$, then we can simply plug $T$ into the power series of $f$. 
Thus, in this case we know how to define $f(T)$ for all $f \in H^\infty$. 

Now, suppose that $T$ is a contraction, but that $\sigma(T)$ is not contained in the {\em open} disc $\bD$. Given a bounded analytic function $f \in H^\infty$, how can we define $f(T)$? 
Note that the holomorphic functional calculus cannot be used, because $\sigma(T)$ contains points on the circle $\bT = \partial \bD$, while not every $f\in H^\infty$ extends to a holomorphic function on a neighborhood of the closed disc. 

The first case that we can treat easily is the case when $f$ belongs to the {\bf disc algebra} $A(\bD)\subset H^\infty$, which is the algebra of all bounded analytic functions on the open unit disc that extend continuously to the closure $\ol{\bD}$. 
This case can be handled using von Neumann's inequality (Theorem \ref{eq:vNineq}), which is an immediate consequence of Theorem \ref{thm:unidil}. 
Indeed, it is not very hard to show that $A(\bD)$ is the closure of the polynomials with respect to the supremum norm $\|f\|_\infty = \sup_{|z|\leq 1}|f(z)|$. 
If $p_n$ is a sequence of polynomials that converges uniformly on $\ol{\bD}$ to $f$, and $T$ is a contraction, then von Neumann's inequality implies that $p_n(T)$ is a Cauchy sequence, so we can define $f(T)$ to be $\lim_n p_n(T)$. 
It is not hard to show that the functional calculus $A(\bD) \ni f\mapsto f(T)$ has all the properties one wishes for: it is a homomorphism extending the polynomial functional calculus, it is continuous, and it agrees with the continuous functional calculus if $T$ is normal. 

Defining a functional calculus for $H^\infty$ is a more delicate matter, but here again the unitary dilation leads to a resolution. 
The rough idea is that we can look at the minimal unitary dilation $U$ of $T$, and use spectral theory to analyze what can be done for $U$.  
If the spectral measure of $U$ is absolutely continuous with respect to Lebesgue measure on the unit circle, then it turns out that we can define $f(U)$ for all $f \in H^\infty$, and then we can simply define $f(T)$ to be the compression of $f(U)$ to $\cH$. 
In this case the functional calculus $f \mapsto f(T)$ is a homomorphism that extends the polynomial and holomorphic functional calculi, it is continuous in the appropriate sense, and it agrees with the Borel functional calculus when $T$ is normal. 
(Of course, this only becomes useful if one can find conditions that guarantee that the minimal unitary dilation of $T$ has absolutely continuous spectral measure. A contraction $T$ is said to be {\bf{completely nonunitary}} (c.n.u.) if is has no reducing subspace $\cM$ such that the restriction  $T\big|_\cM$ is unitary. Every contraction splits as a direct sum $T = T_0 \oplus T_1$, where $T_0$ is unitary and $T_1$ is c.n.u.  Sz.-Nagy and Foias have shown that if $T$ is c.n.u., then the spectral measure of its minimal unitary dilation is absolutely continuous.)

If the spectral measure of $U$ is not absolutely continuous with respect to Lebesgue measure, then there exits a subalgebra $H_U^\infty$ of $H^\infty$ for which there exists a functional calculus $f \mapsto f(U)$, and then one can compress to get $f(T)$; it was shown that $H_U^\infty$ is precisely the subalgebra of functions in $H^\infty$ on which $f \mapsto f(T)$ is a well defined homomorphism. 
See \cite[Chapter III]{nagy} for precise details. 

There are two interesting aspects to note. 
First, an intrinsic property of the minimal dilation --- the absolute continuity of its spectral measure --- provides us with nontrivial information about $T$ (whether or not it has an $H^\infty$-functional calculus). 
The second interesting aspect is that there do exist interesting functions $f \in H^\infty$ that are not holomorphic on a neighborhood of the closed disc (and are not even continuous up to the boundary) for which we would like to evaluate $f(T)$.
This technical tool has real applications. 
See \cite[Section III.8]{nagy}, for example. 

\section{Classical dilation theory of commuting contractions}\label{sec:commuting}

\subsection{Dilations of several commuting contractions}

The manifold applications of the unitary dilation of a contraction on a Hilbert space motivated the question (which is appealing and natural in itself, we must admit) whether the theory can be extended in a sensible manner to families of operators. 
The basic problem is: given commuting contractions $T_1, \ldots, T_d \in B(\cH)$, to determine whether there exist commuting isometries/unitaries $U_1, \ldots, U_d$ on a larger Hilbert space $\cK \supseteq \cH$ such that 
\be\label{eq:dild}
T_1^{n_1} \cdots T_d^{n_d} = P_\cH U_1^{n_1} \cdots U_d^{n_d} \big|_\cH 
\ee
for all $(n_1, \ldots, n_d) \in \bN^d$. Such a family $U_1, \ldots, U_d$ is said to be an {\bf{isometric/unitary dilation}} (I warned you that the word is used differently in different situations!). 

Clearly, it would be nice to have a unitary dilation, because commuting unitaries are completely understood by spectral theory. 
On the other hand, isometric dilations might be easier to construct. 
Luckily, we can have the best of both worlds, according to the following theorem of It\^{o} and Brehmer (see \cite[Proposition I.6.2]{nagy}). 

\begin{theorem}\label{thm:iso2unidil}
Every family of commuting isometries has a commuting unitary extension.
\end{theorem} 
\begin{proof}
The main idea of the proof is that given commuting isometries $V_1$, \ldots, $V_d$ on a Hilbert space $\cH$, one may extend them to commuting isometries $W_1, \ldots, W_d$ such that (a) $W_1$ is a unitary, and (b) if $V_i$ is unitary then $W_i$ is a unitary. 
Given that this is possible, one my repeat the above process $d$ times to obtain a unitary extension. 
The details are left to the reader. 
\end{proof}
In particular, every family of commuting isometries has a commuting unitary dilation. 
Thus, a family of commuting contractions $T_1, \ldots, T_d \in B(\cH)$ has an isometric dilation if and only if it has a unitary dilation. 

\begin{theorem}[And\^{o}'s isometric dilation theorem \cite{Ando}]\label{thm:Andoisodil}
Let $T_1, T_2$ be two commuting contractions on a Hilbert space $\cH$. 
Then there exists a Hilbert space $\cK \supseteq \cH$ and two commuting isometries $V_1, V_2$ on $\cK$ such that 
\be\label{eq:dil2}
T_1^{n_1} T_2^{n_2} = P_\cH V_1^{n_2} V_2^{n_2} \big|_\cH \quad \textup{ for all } n_1, n_2 \in \bN. 
\ee
In fact, $V_1, V_2$ can be chosen such that $V_i^*\big|_\cH = T_i^*$ for $i=1,2$. 
\end{theorem}
In other words, every pair of contractions has an isometric dilation, and in fact, an isometric coextension. 
\begin{proof}
The proof begins similarly to the proof of Theorem \ref{thm:isodil}: we define the Hilbert space $\cK = \oplus_{n \in \bN} \cH = \cH \oplus \cH \oplus \cdots$, and we define two isometries $W_1, W_2$ by 
\[
W_i(h_0, h_1, h_2, \ldots) = (T_i h_0, D_{T_i}h_0, 0, h_1, h_2, \ldots)
\]
for $i=1,2$. 
These are clearly isometric coextensions, but they do not commute: 
\begin{align*}
W_j W_i (h_0, h_1, h_2, \ldots) &= W_j (T_i h_0, D_{T_i}h_0, 0, h_1, h_2, \ldots) \\
&= (T_j T_i h_0, D_{T_j} T_i h_0, 0, D_{T_i}h_0, 0, h_1, h_2, \ldots).
\end{align*}
Of course, in the zeroth entry we have equality $T_1 T_2 h_0 = T_2 T_1 h_0$ and from the fifth entry on we have $(h_1, h_2, \ldots) = (h_1, h_2, \ldots)$. 
The problem is that usually $D_{T_1} T_2 h_0 \neq D_{T_2} T_1 h_0$ and $D_{T_1}h_0 \neq D_{T_2}h_0$. 
However , 
\begin{align*}
\|D_{T_1} T_2 h_0\|^2 + \|D_{T_2} h_0\|^2 &= \langle T_2^* (I-T_1^*T_1)T_2 h_0, h_0 \rangle + \langle (I - T_2^*T_2)h_0, h_0 \rangle \\
&= \langle (I - T_2^*T_1^*T_1 T_2) h_0, h_0 \rangle \\
&= \|D_{T_2} T_1 h_0\|^2 + \|D_{T_1} h_0\|^2, 
\end{align*}
and this allows us to define a unitary operator $U_0 : \cG:=\cH \oplus \cH \oplus \cH \oplus \cH \to \cG$ that satisfies 
\[
U_0(D_{T_1} T_2 h_0, 0, D_{T_2}h_0, 0) = (D_{T_2} T_1 h_0, 0, D_{T_1}h_0, 0).
\]
Regrouping $\cK = \cH \oplus \cG \oplus \cG \oplus \cdots$, we put $U = I_\cH \oplus U_0 \oplus U_0 \oplus \cdots$, and now we define $V_1 = UW_1$ and $V_2 = W_2 U^{-1}$. 
The isometries $V_1$ and $V_2$ are isometric coextensions --- multiplying by $U$ and $U^{-1}$ did not spoil this property of $W_1, W_2$. 
The upshot is that $V_1$ and $V_2$ commute; we leave this for the reader to check. 
\end{proof}

As a consequence (by Theorem \ref{thm:iso2unidil}), 
\begin{theorem}[And\^{o}'s unitary dilation theorem \cite{Ando}]\label{thm:Andounitarydil}
Every pair of contractions has a unitary dilation. 
\end{theorem} 
One can also get minimal dilations, but it turns out that in the multivariable setting minimal dilations are not unique, so they are not canonical and don't play a prominent role. 
Once the existence of the unitary dilation is known, the following two-variable version of von Neumann's inequality follows just as in the proof of Theorem \ref{thm:vNinequality}. 

\begin{theorem}\label{thm:AndoVN}
Let $T_1, T_2$ be two commuting contractions on a Hilbert space $\cH$. 
Then for every complex two-variable polynomial $p$, 
\[
\|p(T) \| \leq \sup_{z\in\bT^2}|p(z)|. 
\]
\end{theorem}
Here and below we use the shorthand notation $p(T) = p(T_1, \ldots, T_d)$ whenever $p$ is a polynomial in $d$ variables and $T = (T_1, \ldots, T_d)$ is a $d$-tuple of operators. 
The proof of the above theorem (which is implicit in the lines preceding it) gives rise to an interesting principle: whenever we have a unitary or a normal dilation then we have a von Neumann type inequality. 
This principle can be used in reverse, to show that for three or more commuting contractions there might be no unitary dilation, in general. 

\begin{example}\label{ex:noDil}
There exist three contractions $T_1, T_2, T_3$ on a Hilbert space $\cH$ and a complex polynomial $p$ such that 
\[
\|p(T)\| > \|p\|_\infty := \sup_{z\in \bT^3}|p(z)| .
\] 
Consequently, $T_1, T_2, T_3$ have no unitary, and hence also no isometric, dilation. 
There are several concrete examples. 
The easiest to explain, in my opinion, is the one presented by Crabb and Davie \cite{CD75}. 
One takes a Hilbert space $\cH$ of dimension $8$ with orthonormal basis $e, f_1, f_2, f_3, g_1, g_2, g_3, h$, and defines partial isometries $T_1, T_2, T_3$ by 
\begin{align*}
T_i e &= f_i \\
T_i f_i &= - g_i \\
T_i f_j &= g_k \,\, , \,\,  k \neq i,j \\
T_i g_j &= \delta_{ij} h \\
T_ih &= 0
\end{align*}
for $i,j,k = 1,2,3$. 
Obviously, these are contractions, and checking that $T_i T_j = T_j T_i$ is probably easier than you guess. 
Now let $p(z_1, z_2, z_3) = z_1 z_2 z_3 - z_1^3 - z_2^3 - z_3^3$. 
Directly evaluating we see that $p(T_1, T_2, T_3)e = 4h$, so that $\|p(T_1, T_2, T_3)\| \geq 4$. 
On the other hand, it is elementary to show that $|p(z)|<4$ for all $z\in \bT^3$, so by compactness of $\bT^3$ we get $\|p\|_\infty < 4$, as required. 
\end{example}

\begin{remark}
At more or less the same time that the above example appeared, Kaijser and Varopoulos discovered three $5 \times 5$ commuting contractive matrices that do not satisfy von Neumann's inequality \cite{Var74}. 
On the other hand, it was known (see \cite[p. 21]{Dru83}) that von Neumann's inequality holds for any $d$-tuple of $2 \times 2$ matrices, in fact, every such $d$-tuple has a commuting unitary dilation. 
It was therefore begged of operator theorists to decide whether or not von Neumann's inequality holds for $3$-tuples of $3 \times 3$ and $4 \times 4$ commuting contractive matrices. 
Holbrook found a $4 \times 4$ counter example in 2001 \cite{Hol01}, and the question of whether von Neumann's inequality holds for three $3 \times 3$ contractions remained outrageously open until finally, only very recently, Knese \cite{Kne16} showed how results of Kosi\'{n}ski on the three point Pick interpolation problem in the polydisc \cite{Kos15} imply that in the $3 \times 3$ case the inequality holds (it is still an open problem whether or not every three commuting $3\times 3$ contractions have a commuting unitary dilation; the case of four $3\times 3$ contractions was settled negatively in \cite{CD13}).  
\end{remark}

It is interesting to note that the first example of three contractions that do not admit a unitary dilation did not involve a violation of a von Neumann type inequality. 
Parrott \cite{Par70} showed that if $U$ and $V$ are two noncommuting unitaries, then the operators
\be\label{eq:Parrott}
T_1 = \begin{pmatrix} 0 & 0 \\ I & 0 \end{pmatrix} \,\, , \,\,
T_2 = \begin{pmatrix} 0 & 0 \\ U & 0 \end{pmatrix} \,\, , \,\,
T_3 = \begin{pmatrix} 0 & 0 \\ V & 0 \end{pmatrix}
\ee
are three commuting contractions that have no isometric dilation. 
However, these operators can be shown to satisfy von Neumann's inequality. 

What is it exactly that lies behind this dramatic difference between $d=2$ and $d=3$? Some people consider this to be an intriguing mystery, and there has been effort made into trying to understand which $d$-tuples are the ones that admit a unitary dilation (see, e.g., \cite{StoSz01} and the references therein), or at least finding sufficient conditions for the existence of a nice dilation. 

A particularly nice notion of dilation is that of {{\em regular dilation}}. 
For a $d$-tuple $T = (T_1, \ldots, T_d) \in B(\cH)^d$ and $n = (n_1, \ldots, n_d) \in \bN^d$, we write $T^n = T_1^{n_1} \cdots T^{n_d}_d$. 
If $n = (n_1, \ldots, n_d) \in \bZ^d$, then we define
\[
T(n) = (T^{n_-})^* T^{n_+} 
\]
where $n_+ = (\max\{n_1,0\}, \ldots, \max\{n_d,0\})$ and $n_- = n_+ -  n$. 
For a commuting unitary tuple $U$ and $n \in \bZ^d$, we have $U(n) = U^n:=U_1^{n_1} \cdots U_d^{n_d}$. 
Now, if $\cK$ contains $\cH$ and $U = (U_1, \ldots, U_d) \in B(\cK)^d$ a $d$-tuple of unitaries, say that $U$ is a {\bf{regular dilation}} of $T$ if
\be\label{eq:regdil}
T(n) = P_\cH U^n \big|_\cH  \,\, \textup{ for all } n \in \bZ^d. 
\ee
Note that a unitary (power) dilation of a single contraction is automatically a regular dilation, because applying the adjoint to \eqref{eq:Ndil} gives $T(k) = P_\cH U^k\big|_\cH$ for all $k \in \bZ$. 
However, a given unitary dilation of a pair of contractions need not satisfy \eqref{eq:regdil}, and in fact there are pairs of commuting contractions that have no regular unitary dilation. 

In contrast with the situation of unitary dilations, the tuples of contractions that admit a regular unitary dilation can be completely characterized. 

\begin{theorem}[Regular unitary dilation]\label{thm:regular}
A $d$-tuple $T = (T_1, \ldots, T_d)$ of commuting contractions on a Hilbert space has a regular unitary dilation if and only if, 
\be\label{eq:Brehmer}
\sum_{\{i_1, \ldots, i_m\} \subseteq S} (-1)^{m} T_{i_1}^* \cdots T_{i_m}^* T_{i_1} \cdots T_{i_m} \geq 0 \quad, \quad \textup{ for all } S \subseteq \{1, \ldots, d\}.
\ee
\end{theorem}
The conditions \eqref{eq:Brehmer} are sometimes called {{\em Brehmer's conditions}}. 
For the proof, one shows that the function $n \mapsto T(n)$ is a {\bf{positive definite function on the group $\bZ^d$}} (see Section I.9 in \cite{nagy}), and uses the fact that every positive definite function on a group has a unitary dilation \cite[Section I.7]{nagy}.

\begin{corollary}\label{cor:suff_reg}
The following are sufficient conditions for a $d$-tuple $T = (T_1, \ldots, T_d)$ of commuting contractions on a Hilbert space to have a regular unitary dilation:
\begin{enumerate}
\item $\sum_{i=1}^d \|T_i\|^2 \leq 1$. 
\item $T_1, \ldots, T_d$ are all isometries. 
\item $T_1, \ldots, T_d$ doubly commute, in the sense that $T_iT_j^* = T_j^*T_i$ for all $i\neq j$ (in addition to $T_i T_j = T_j T_i$ for all $i,j$). 
\end{enumerate}
\end{corollary}
\begin{proof}
It is not hard to show that the conditions listed in the corollary are sufficient for Brehmer's conditions \eqref{eq:Brehmer} to hold. 
\end{proof}

\subsection{Commutant lifting}

We return to the case of two commuting contractions. 
The following innocuous looking theorem, called the {\em commutant lifting theorem}, has deep applications (as we shall see in Section \ref{sec:Pick}) and is the prototype for numerous generalizations. 
It originated in the work of Sarason \cite{Sar67}, was refined by Sz.-Nagy and Foias (see \cite{nagy}), and has become a really big deal (see \cite{FFBook}). 

\begin{theorem}[Commutant lifting theorem]\label{thm:commLift}
Let $A$ be a contraction on a Hilbert space $\cH$, and let $V \in B(\cK)$ be the minimal isometric dilation of $A$. For every contraction $B$ that commutes with $A$, there exists an operator $R \in B(\cK)$ such that 
\begin{enumerate}
\item $R$ commutes with $V$, 
\item $B = R^*\big|_\cH$, 
\item $\|R\| \leq 1$. 
\end{enumerate}
\end{theorem}
In other words, every operator commuting with $A$ can be ``lifted" to an operator commuting with its minimal dilation, without increasing its norm. 
\begin{proof}
Let $U,W \in B(\cL)$ be the commuting isometric coextension of $A,B$, where $\cL$ is a Hilbert space that contains $\cH$ (the coextension exists by And\^{o}'s isometric dilation theorem, Theorem \ref{thm:Andoisodil}). 
The restriction of the isometry $U$ to the subspace $\bigvee_{n \in \bN} U^n \cH$ is clearly
\begin{enumerate}
\item an isometry, 
\item a dilation of $A$, 
\item a minimal dilation, 
\end{enumerate} 
and therefore (by uniqueness of the minimal isometric dilation), the restriction of $U$ to $\bigvee_{n \in \bN} U^n \cH$ is unitarily equivalent to {\em the} minimal isometric dilation $V$ on $\cH$, so we identify $\cK = \bigvee_{n \in \bN} U^n \cH$ and $V = U\big|_\cH$. 
It follows (either from our knowledge on the minimal dilation, or simply from the fact that $U$ is a coextension) that $V$ is a coextension of $A$. 
With respect to the decomposition $\cL = \cK \oplus \cK^\perp$, 
\[
U = \begin{pmatrix} 
V & X \\0 & Z
\end{pmatrix} 
\quad , \quad 
W = \begin{pmatrix} 
R & Q \\ P & N
\end{pmatrix} 
\]
It is evident that $\|R\|\leq 1$ and that $R$ is a coextension of $B$. 
We wish to show that $RV = VR$. 

From $UW = WU$ we find that $VR + XP = RV$. 
Thus, the proof will be complete if we show that $X = 0$. 
Equivalently, we have to show that $\cK$ is invariant under $U^*$. 
To see  this, consider $U^* U^n h$ for some $h \in \cH$ and $n \in \bN$. 
If $n\geq 1$ we get $U^{n-1}h$ which is in $\cK$. 
If $n=0$ then we get $U^*h = A^*h \in \cH \subseteq \cK$, because $U$ is a coextension of $A$. 
That completes the proof. 
\end{proof}

\subsection{Dilations of semigroups and semi-invariant subspaces}\label{subsec:semi}

Above we treated the case of a single operator or a tuple of commuting operators. 
However, dilation theory can also be developed, or at least examined, in the context of operator semigroups.

Let $T = \{T_s\}_{s \in \cS} \subset B(\cH)$ be a family of operators parametrized by a semigroup $\cS$ with unit $e$. 
Then $T$ is a said to be a {\bf{semigroup of operators over $\cS$}} if 
\begin{enumerate}
\item $T_e = I$,
\item $T_{st} = T_s T_t$ for all $s,t \in \cS$. 
\end{enumerate}
If $\cS$ is a topological semigroup, then one usually requires the semigroup $T$ to be continuous in some sense. 
A semigroup $V = \{V_s\}_{s\in \cS} \subset B(\cK)$ is said to be a {\bf{dilation}} of $T$ if $\cK \supset \cH$ and if 
\[
T_s = P_\cH V_s \big|_\cH \quad, \quad \textup{ for all } s \in \cS. 
\]

Note that Sz.-Nagy's unitary dilation theorem can be rephrased by saying that every semigroup of contractions over $\cS = \bN$ has a unitary dilation, in the above sense. 
Similarly, there are notions of {\bf{extension}} and {\bf{coextension}} of a semigroup of operators. 
Some positive results have been obtained for various semigroups. 
Sz.-Nagy proved that every semigroup $T = \{T_t\}_{t \in \bR_+}$ of contractions that is point-strong continuous (in the sense that $t\mapsto T_t h$ is continuous for all $h \in \cH$) has isometric and unitary dilations, which are also point-strong continuous (see \cite[Section I.10]{nagy}). 
This result was extended to the two parameter case by S{\l}oci\'{n}ski \cite{Slo74} and Ptak \cite{Pta85}; the latter also obtained the existence of regular dilations for certain types of multi-parameter semigroups. 
Douglas proved that every commutative semigroup of isometries has a unitary extension \cite{Dou69}. 
Letting the commutative semigroup be $\cS = \bN^d$, we recover Theorem \ref{thm:iso2unidil}. 
Douglas's result was generalized by Laca to semigroups of isometries parametrized by an Ore semigroup \cite{Lac00}, and in fact to ``twisted" representations. 

A result that somewhat sheds light on the question, which tuples of operators have a unitary dilation and which don't, is due to Opela. 
If $T = \{T_i\}_{i\in I}\subset B(\cH)$ is a family of operators, we say that $T$ {\bf{commutes according to the graph}} $\cG$ with vertex set $I$, if $T_i T_j = T_j T_i$ whenever $\{i,j\}$ is an edge in the (undirected) graph $\cG$. 
We can consider $T$ as a semigroup parameterized by a certain quotient of the free semigroup over $I$. 
Opela proved the following compelling result: {given a graph $\cG$, every family $T = \{T_i\}_{i\in I}$ of contractions commuting according to $\cG$ has a unitary dilation that commutes according to $\cG$, if and only if $\cG$ contains no cycles \cite{Ope06}.}

It is interesting that in the general setting of semigroups of operators, one can say something about the structure of dilations. 
\begin{theorem}[Sarason's lemma \cite{Sar65}]\label{thm:sarason}
Let $V = \{V_s\}_{s \in \cS}\subset B(\cK)$ be a semigroup of operators over a semigroup with unit $\cS$, and let $\cH$ be a subspace of $\cK$. 
Then the family $T = \{T_s := P_\cH V_s \big|_\cH\}_{s\in\cS}$ is a semigroup over $\cS$ if and only if there exist two subspaces $\cM \subseteq \cN \subseteq \cK$, invariant under $V_s$ for all $s$, such that $\cH = \cN \ominus \cM := \cN \cap \cM^\perp$. 
\end{theorem} 
\begin{proof}
The sufficiency of the condition is easy to see, if one writes the elements of the semigroup $V$ as $3 \times 3$ block operator matrices with respect to the decomposition $\cK = \cM \oplus \cH \oplus \cN^\perp$. 

For the converse, one has no choice but to define $\cN = \vee_{s\in \cS} V_s \cH$ (clearly an invariant space containing $\cH$), and then it remains to prove that $\cM := \cN \ominus \cH$ is invariant for $V$, or --- what is the same --- that $P_\cH V_t \cM = 0$ for all $t \in \cS$. 
Fixing $t$, we know that for all $s$, 
\[
P_\cH V_t P_\cH V_s P_\cH = T_t T_s = T_{ts} = P_\cH V_{ts} P_\cH = P_\cH V_t V_s P_\cH. 
\]
It follows that $P_\cH V_t P_\cH = P_\cH V_t$ on $\vee_s V_s \cH = \cN$. 
In particular, $P_\cH V_t \cM = P_\cH V_t P_\cH \cM = 0$ (since $\cM \perp \cH$), as required. 
\end{proof}
The theorem describes how a general dilation looks like. 
A subspace $\cH$ as above, which is the difference of two invariant subspaces, is said to be {\bf{semi-invariant}} for the family $V$. 
In the extreme case where $\cM = \{0\}$, the space $\cH = \cN$ is just an invariant subspace for $V$, and $V$ is an extension of $T$. 
In the other extreme case when $\cN = \cK$, the space $\cH$ is a coinvariant subspace for $V$, and $V$ is a coextension of $T$. 
In general, the situation is more complicated, but still enjoys some structure. 
In the special $\cS = \bN$, case Sarason's lemma implies that $V$ is a (power) dilation of an operator $T$ if and only if it has the following block form:
\be\label{eq:sarason}
V = \begin{pmatrix}
\ast & \ast & \ast \\ 
0     & T     & \ast \\ 
0     &   0   & \ast
\end{pmatrix}.
\ee
Sarason's lemma is interesting and useful also in the case of dilations of a single contraction. 

\begin{remark}
Up to this point in the survey, rather than attempting to present a general framework that encapsulates as much of the theory as possible, I chose to sew the different parts together with a thin thread. 
There are, of course, also ``high level" approaches. 
In Section \ref{sec:opalg} we will see how the theory fits in the framework of operator algebras, which is one unifying viewpoint (see also \cite{DK11,PauBook,Pis96}). 
There are other viewpoints.
A notable one is due to Sz.-Nagy ---
very soon after he proved his unitary dilation theorem for a single contraction, he found a far-reaching generalization in terms of dilations of positive functions on {\em $*$-semigroups}; see \cite{SzN60}, which contains a theorem from which a multitude of dilation theorems can be deduced (see also \cite{Sza10} for a more recent discussion with some perspective). 
Another brief but high level look on dilation theory can be found in Arveson's survey \cite{ArvDil}. 
\end{remark}

\section{An application: Pick interpolation}\label{sec:Pick}

The purpose of this section is to illustrate how classical dilation theory can be applied in a nontrivial way to prove theorems in complex function theory. 
The example we choose is classical -- Pick's interpolation theorem -- and originates in the work of Sarason \cite{Sar67}. 
Sarason's idea to use commutant lifting to solve the Pick interpolation problem works for a variety of other interpolation problems as well, including Carath\'{e}odory interpolation, matrix valued interpolation, and mixed problems. 
It can also be applied in different function spaces and multivariable settings. 
Here we will focus on the simplest case. 
Good references for operator theoretic methods and interpolation are \cite{AMBook} and \cite{FFBook}, and the reader is referred to these sources for details and references. 

\subsection{The problem}\label{sec:setting}

Recall that $H^\infty$ denotes the algebra of bounded analytic functions on the unit disc $\bD = \{z \in \bC : |z| < 1\}$. 
For $f \in H^\infty$ we define 
\[
\|f\|_\infty = \sup_{z \in \bD}|f(z)| . 
\]
This norm turns $H^\infty$ into a Banach algebra. 

The {\em Pick interpolation problem} is the following: given $n$ points $z_1, \ldots, z_n$ in the unit disc and $n$ target points $w_1, \ldots, w_n \in \bC$, determine whether or not there exists a function $f \in H^\infty$ such that 
\be\label{eq:interpolate}
f(z_i) = w_i \quad, \quad \textup{ for all } i = 1, \ldots, n, 
\ee
and 
\be\label{eq:norm}
\|f\|_\infty \leq 1. 
\ee
It is common knowledge that one can always find a polynomial (unique, if we take it to be of degree less than or equal to $n-1$) that {\em interpolates the data}, in the sense that \eqref{eq:interpolate} holds. 
The whole point is that we require \eqref{eq:norm} to hold as well. 
Clearly, this problem is closely related to the problem of finding the $H^\infty$ function of minimal norm that interpolates the points. 

Recall that an $n \times n$ matrix $A = (a_{ij})_{i,j=1}^n$ is said to be {\bf{positive semidefinite}} if for every $v = (v_1, \ldots, v_n)^T \in \bC^n$ 
\[
\langle Av, v \rangle = \sum_{i,j=1}^n a_{ij}v_j \ol{v}_i \geq 0. 
\]
If $A$ is positive semidefinite, then we write $A \geq 0$. 

\begin{theorem}[Pick's interpolation theorem]\label{thm:Pick}
Given points $z_1, \ldots, z_n$ and $w_1, \ldots, w_n$ as above, there exists a function $f \in H^\infty$ satisfying \eqref{eq:interpolate}-\eqref{eq:norm}, if and only if the following matrix inequality holds: 
\be\label{eq:pick_cond}
\left(\frac{1 - w_i \ol{w_j}}{1 - z_i \ol{z_j}} \right)_{i,j=1}^n \geq 0 .
\ee
\end{theorem}

The $n \times n$ matrix on the left hand side of \eqref{eq:pick_cond} is called the {\em Pick matrix}. 
What is remarkable about this theorem is that it gives an explicit and practical necessary and sufficient condition for the solvability of the interpolation problem: that the Pick matrix be positive semidefinite.

At this point it is not entirely clear how this problem is related to operator theory on Hilbert spaces, since there is currently no Hilbert space in sight. 
To relate this problem to operator theory we will represent $H^\infty$ as an operator algebra. 
The space on which $H^\infty$ acts is an interesting object in itself, and to this space we devote the next subsection. 
Some important properties of $H^\infty$ functions as operators will be studied in Section \ref{subsec:shift}, and then, in Section \ref{subsec:sol}, we will prove Theorem \ref{thm:Pick}. 

\subsection{The Hilbert function space \texorpdfstring{$H^2$}{H2}}\label{sec:RKHS}

The {\em Hardy space} $H^2 = H^2(\bD)$ is the space of analytic functions $h(z) = \sum_{n=0}^\infty a_n z^n$ on the unit disc $\bD$ that satisfy $\sum |a_n|^2 < \infty$. 
It is not hard to see that $H^2$ is a linear subspace, and that 
\[
\Big\langle \sum a_n z^n , \sum b_n z^n \Big\rangle = \sum a_n \ol{b}_n 
\]
is an inner product which makes $H^2$ into a Hilbert space, with norm 
\[
\|h\|^2_{H^2} = \sum |a_n|^2 . 
\]
In fact, after noting that every $(a_n)_{n=0}^\infty \in \ell^2:=\ell^2(\bN,\bC)$ gives rise to a power series that converges (at least) in $\bD$, it is evident that the map 
\[
(a_n)_{n=0}^\infty \mapsto \sum_{n=0}^\infty a_n z^n 
\]
is a unitary isomorphism of $\ell^2$ onto $H^2(\bD)$, so the Hardy space is a Hilbert space, for free. 
The utility of representing a Hilbert space in this way will speak for itself soon. 

For $w \in \bD$, consider the element $k_w \in H^2$ given by 
\[
k_w(z) = \sum_{n=0}^\infty \ol{w}^nz^n = \frac{1}{1-z\ol{w}} .
\]
Then for $h(z) = \sum a_n z^n$, we calculate
\[
\langle h, k_w\rangle = \left\langle \sum a_n z^n , \sum \ol{w}^n z^n \right\rangle = \sum a_n w^n = h(w). 
\]
We learn that the linear functional $h \mapsto h(w)$ is a bounded functional, and that the element of $H^2$ that implements this functional is $k_w$. 
The functions $k_w$ are called {\bf{kernel functions}}, and the function $k : \bD \times \bD \to \bC$ given by $k(z,w) = k_w(z)$ is called the {\bf{reproducing kernel}} of $H^2$. 
The fact that point evaluation in $H^2$ is a bounded linear functional lies at the root of a deep connection between function theory on the one hand, and operator theory, on the other. 

The property of $H^2$ observed in the last paragraph is so useful and important that it is worth a general definition. 
A Hilbert space $\cH \subseteq \bC^X$ consisting of functions on a set $X$, in which point evaluation $h \mapsto h(x)$ is bounded for all $x \in X$, is said to be a {\bf{Hilbert function space}} or a {\bf{reproducing kernel Hilbert space}}. 
See \cite{PauRagBook16} for a general introduction to this subject, and \cite{AMBook} for an introduction geared towards Pick interpolation (for readers that are in a hurry, Chapter 6  in \cite{ShalitBook} contains an elementary introduction to $H^2$ as a Hilbert function space). 
If $\cH$ is a Hilbert function space on $X$, then by the Riesz representation theorem, for every $x \in X$ there is an element $k_x \in \cH$ such that $h(x) = \langle h, k_x \rangle$ for all $h \in \cH$, and one may define the {\bf reproducing kernel} of $\cH$ by $k(x,y) = k_y(z)$.

The {\bf{multiplier algebra}} of a Hilbert function space $\cH$ on a set $X$ is defined to be 
\[
\mlt(\cH) = \{f : X \to \bC : fh \in \cH \,\, \textup{ for all } h \in \cH\}. 
\]
Every $f \in \mlt(\cH)$ gives rise to a linear {\bf{multiplication operator}} $M_f : \cH \to \cH$ that acts as $M_f h = fh$, for all $h \in \cH$. 
By the closed graph theorem, multiplication operators are bounded. 
The following characterization of multiplication operators is key to some applications. 

\begin{proposition}\label{prop:mult_eigs}
Let $\cH$ be a Hilbert function space on a set $X$. 
If $f \in \mlt(\cH)$, then  $M_f^*k_x = \ol{f(x)}k_x$ for all $x \in X$. 
Conversely, if $T \in B(\cH)$ is such that for all $x \in X$ there is some $\lambda_x \in \bC$ such that $T k_x = \lambda_x k_x$, then there exists $f \in \mlt(\cH)$ such that $T = M_f^*$. 
\end{proposition}
\begin{proof}
For all $h \in H$ and $x \in X$, 
\[
\langle h,  M_f^* k_x\rangle = \langle fh, k_x \rangle = f(x)h(x) = f(x) \langle h, k_x \rangle = \langle h, \ol{f(x)}k_x \rangle, 
\]
so $M_f^* k_x = \ol{f(x)} k_x$. 
The converse is similar.
\end{proof}

\begin{corollary}\label{cor:mult_bdd}
Every $f \in \mlt(\cH)$ is a bounded function, and 
\[
\sup_{x\in X} |f(x)| \leq \|M_f\|. 
\]
\end{corollary}

\begin{proposition}\label{prop:Hinf}
$\mlt(H^2) = H^\infty$ and $\|M_f\| = \|f\|_\infty$ for every multiplier. 
\end{proposition}
\begin{proof}
Since $1 \in H^2$, every multiplier $f = M_f 1$ is in $H^2$. 
In particular, every multiplier is an analytic function. 
By the above corollary, $\mlt(H^2) \subseteq H^\infty$, and $\|f\|_\infty \leq \|M_f\|$ for every multiplier $f$. 

Conversely, if $p(z) = \sum_{n=0}^N a_n z^n$ is a polynomial, then it is straightforward to check that 
\[
\|p\|_{H^2}^2 = \frac{1}{2\pi}\int_0^{2\pi} |p(e^{it})|^2 dt. 
\]
An approximation argument then gives
\[
\|h\|_{H^2}^2 = \lim_{r\nearrow 1}\frac{1}{2\pi}\int_0^{2\pi} |h(re^{it})|^2 dt
\]
for all $h \in H^2$. 
This formula for the norm in $H^2$ implies that $H^\infty \subseteq \mlt(H^2)$, and that $\|M_f\| \leq \|f\|_\infty$.  
\end{proof}

We will henceforth identify $f$ with $M_f$, and we will think of $H^\infty$ as a subalgebra of $B(H^2)$.

\subsection{The shift \texorpdfstring{$M_z$}{Mz}}\label{subsec:shift}

We learned that every bounded analytic function $f \in H^\infty$ defines a bounded multiplication operator $M_f : H^2 \to H^2$, but there is one that stands out as the most important. 
If we abuse notation a bit and denote the identity function ${\bf id} : \bD \to \bD$ simply as $z$, then we obtain the multiplier $M_z$, defined by 
\[
(M_z h)(z) = zh(z) . 
\]
It is quite clear that $M_z$ is an isometry, and in fact it is unitarily equivalent to the unilateral shift of multiplicity one on $\ell^2$ (defined before Theorem \ref{thm:wold}). 
We will collect a couple of important results regarding this operator, before getting back to the proof of Pick's theorem. 

Recall that the {\bf{commutant}} of a set of operators $\cS \subset B(\cH)$ is the algebra
\[
\cS' = \{T \in B(\cH) : ST = TS \,\, \textup{ for all } S \in \cS\}. 
\]

\begin{proposition}\label{prop:Hinf_comtnt}
$\{M_z\}' = (H^\infty)' = H^\infty$. 
\end{proposition}
\begin{proof}
Clearly $H^\infty \subseteq (H^\infty)' \subseteq \{M_z\}'$. 
Now suppose that $T \in \{M_z\}'$. We claim that $T = M_f$ for $f = T1$. 
Indeed, if $p(z) = \sum_{n=0}^Na_n z^n$ is a polynomial, then 
\[
Tp = T\sum_{n=0}^Na_n M_z^n 1 = \sum_{n=0}^Na_n M_z^n T1 = M_p f = fp.
\]
An easy approximation argument would show that $T = M_f$, if we knew that $f \in H^\infty$; but we still don't. 
To finesse this subtlety, we find for an arbitrary $h \in H^2$ a sequence of polynomials $p_n$ converging in norm $h$, and evaluate at all points $w \in \bD$, to obtain: 
\[
f(w) p_n(w) = (Tp_n)(w) \xrightarrow{n\to \infty} (Th)(w),
\]
while $f(w) p_n(w) \to f(w) h(w)$, on the other hand. 
This means that $Th = fh$ for all $h$ and therefore $f \in \mlt(H^2) = H^\infty$, as required. 
\end{proof}

Let $z_1, \ldots, z_n \in \bD$. 
It is not hard to see that $k_{z_1}, \ldots, k_{z_n}$ are linearly independent. 
Let $\cG = \spn\{k_{z_1}, \ldots, k_{z_n}\}$, and let $A = P_\cG M_z \big|_\cG$. 
By Proposition \ref{prop:mult_eigs}, $\cG$ is coinvariant for $M_z$, i.e., $M_z^* \cG \subseteq \cG$, and $A^* = M_z^*\big|_\cG$ is the diagonal operator given by 
\be\label{eq:A}
A^*: k_{z_i} \mapsto \ol{z_i} k_{z_i} . 
\ee
We claim that $M_z$ is the minimal isometric dilation of $A$. 
Well, it's clearly an isometric dilation, we just need to show that it is minimal. 
But $k_{z_i}(z) = \frac{1}{1-z\ol{z_i}}$, so $k_{z_i} - \ol{z_i}M_z k_{z_i} = 1 \in \bigvee_{n \in \bN} M_z^n \cG$. 
It follows that that all the polynomials are in $\bigvee_{n \in \bN} M_z^n \cG$, whence $H^2 = \bigvee_{n \in \bN} M_z^n \cG$. 

More generally, if we have a multiplier $f$, and we define $B = P_\cG M_f \big|_\cG$, then we have that 
$B^* = M_f^*\big|_\cG$ and that 
\be\label{eq:B}
B^*: k_{z_i} \mapsto \ol{f(z_i)} k_{z_i} . 
\ee

\subsection{Proof of Pick's interpolation theorem}\label{subsec:sol}

We can now prove Theorem \ref{thm:Pick}. 
We first show that \eqref{eq:pick_cond} is a necessary condition. 
Suppose that $f \in H^\infty$ satisfies $f(z_i) = w_i$ for all $i = 1, \ldots, n$ and that $\|f\|_\infty \leq 1$. 
Define $B = P_\cG M_f \big|_\cG$, where $\cG = \spn\{k_{z_1}, \ldots, k_{z_n}\}$ as in the previous subsection. 
Then, by \eqref{eq:B} $B^*$ is the diagonal operator given by $B^* k_{z_i} = \ol{w_i} k_{z_i}$. 
Since $\|M_f\|\leq 1$, also $\|B^*\|\leq 1$, thus for all $\alpha_1, \ldots, \alpha_n \in \bC$, 
\begin{align*}
0 \leq \left\|\sum_{i=1}^n \alpha_i k_{z_i} \right\|^2 - \left\|B^* \sum_{i=1}^n \alpha_i k_{z_i} \right\|^2 &= \left\|\sum_{i=1}^n\alpha_i k_{z_i} \right\|^2 - \left\|\sum_{i=1}^n \ol{w_i} \alpha_i k_{z_i} \right\|^2 \\
&= \sum_{i,j=1}^n \alpha_j \left(1 - \ol{w_j} w_i \right) \ol{\alpha_i}  \left \langle k_{z_j} , k_{z_i} \right\rangle \\ 
&= \sum_{i,j=1}^n \alpha_j \left( \frac{1 - w_i \ol{w_j} }{1 - z_i\ol{z_j}} \right)\ol{\alpha_i}.
\end{align*}
That is, the Pick matrix is positive semidefinite, and \eqref{eq:pick_cond} holds. 

Conversely, suppose that \eqref{eq:pick_cond} holds. 
Define a diagonal operator $D : \cG \to \cG$ by $D k_{z_i} = \ol{w_i} k_{z_i}$ for $i=1, \ldots, n$, and let $B = D^*$. 
Then the above computation can be rearranged to show that $\|B\| = \|B^*\| \leq 1$. 
Now, the diagonal operator $B^*$ clearly commutes with the diagonal operator $A^* = M_z^*\big|\cG$, so $B$ commutes with $A$. 
Since $M_z$ is the minimal isometric dilation of $A$, the commutant lifting theorem (Theorem \ref{thm:commLift}) implies that $B$ has a coextension to an operator $T$ that commutes with $M_z$ and has $\|T\|\leq 1$. 
By Proposition \ref{prop:Hinf_comtnt}, $T = M_f$ for some $f \in H^\infty$, and by Proposition \ref{prop:mult_eigs}, $f(z_i) = w_i$ for all $i=1, \ldots, n$. 
Since $\|f\|_\infty = \|T\| \leq 1$, the proof is complete.

\section{Spectral sets, complete spectral sets, and normal dilations}\label{sec:spectral}
Classical dilation theory does not end with dilating commuting contractions to commuting unitaries. 
Let us say that a $d$-tuple $N = (N_1, \ldots, N_d)$ is a {\bf{normal}} tuple if $N_1, \ldots, N_d$ are all normal operators and, in addition, they all commute with one another. 
Recall that the {\bf{joint spectrum}} $\sigma(N)$ of a normal tuple is the set 
\[
\sigma(N) = \{(\rho(N_1), \ldots, \rho(N_d)) : \rho \in \cM(C^*(N)) \} \subset \bC^d, 
\]
where $\cM(C^*(N))$ is the space of all nonzero complex homomorphisms from the unital C*-algebra $C^*(N)$ generated by $N$ to $\bC$. 
If $N$ acts on a finite dimensional space, then the joint spectrum is the set of joint eigenvalues, belonging to an orthogonal set of joint eigenvectors that simultaneously diagonalize $N_1, \ldots, N_d$. 
A commuting tuple of unitaries $U = (U_1, \ldots, U_d)$ is the same thing as a normal tuple with joint spectrum contained in the torus $\bT^d$. 
Since normal tuples are in a sense ``completely understood", it is natural to ask which operator tuples $T$ have a normal dilation $N$ (where the definition of dilation is as in \eqref{eq:dild}) with the spectrum $\sigma(N)$ prescribed to be contained in some set $X \subset \bC^d$.

Suppose that $T = (T_1, \ldots, T_d)$ is a commuting tuple of operators and that $N = (N_1, \ldots, N_d)$ is a normal dilation with $\sigma(N) = X \subset \bC^d$. 
Then we immediately find that 
\[
\|p(T)\|\leq \|p(N)\| = \sup_{z \in X}|p(z)|
\]
for every polynomial $p$ in $d$ variables. 
In fact, it is not too hard to see that the above inequality persists when $p$ is taken to be a rational function that is regular on $X$. 
This motivates the following definition: a subset $X \subseteq \bC^d$ is said to be a {\bf{$K$-spectral set}} for $T$ if $X$ contains the joint spectrum $\sigma(T)$ of $T$, and if for every rational function $f$ that is regular on $X$, 
\be\label{eq:spectral_set}
\|f(T)\| \leq K \|f\|_{X,\infty}, 
\ee
where $\|f\|_{X,\infty} = \sup_{z\in X}|f(z)|$. 
If $X$ is a $K$-spectral set for $T$ with $K = 1$, then it is simply said to be a {\bf{spectral set}} for $T$. 

I do not wish to define the joint spectrum of a non-normal commuting tuple, nor to go into how to evaluate a rational function on a tuple of operators, so I will be somewhat sloppy in what follows (see \cite[Section 1.1]{Arv72}; for a textbook treatment, I recommend also \cite{PauBook}). 
Two simplifying comments are in order: 
\begin{enumerate}
\item In the case $d=1$, i.e., just one operator $T$, the spectrum $\sigma(T)$ is the usual spectrum, and the evaluation $f(T)$ of a rational function on $T$ can be done naturally, and this is the same as using the holomorphic functional calculus. 
\item One may also discuss {\bf{polynomial spectral sets}}, in which \eqref{eq:spectral_set} is required to hold only for polynomials \cite{Coh15}. 
If $X$ is polynomially convex (and in particular, if $X$ is convex), then considering polynomials instead of rational functions leads to the same notion. 
\end{enumerate}

Thus, with the terminology introduced above, we can rephrase Theorem \ref{thm:AndoVN} by saying that the bidisc $\ol{\bD}^2$ is a spectral set of every pair $T = (T_1, T_2)$ of commuting contractions, and Example \ref{ex:noDil} shows that there exists three commuting contractions for which the tridisc $\ol{\bD}^3$ is not a spectral set. 

The notion of a spectral set of a single operator is due to von Neumann \cite{vN51}. 
A nice presentation of von Neumann's theory can be found in Sections 153--155 of \cite{RSzN}. 
The reader is referred to \cite{BBSurv} for a rather recent survey with a certain emphasis on the single variable case. 
To give just a specimen of the kind of result that one can encounter, which is quite of a different nature than what I am covering in this survey, let me mention the result of Crouzeix \cite{Cro07}, which says that for every $T \in B(\cH)$, the numerical range $W(T) := \{\langle Th, h \rangle : \|h\|=1\}$ of $T$ is a $K$-spectral set for some $K \geq 2$ (it is easy to see that one cannot have a constant smaller than $2$; Crouzeix conjectured that $K = 2$, and this conjecture is still open at the time of me writing this survey). 

It is plain to see that if $T$ has a commuting normal dilation $N$ with spectrum $\sigma(N) \subseteq X$, then $X$ is a polynomial spectral set for $T$, and it is true that in fact $X$ is a spectral set. 
It is natural to ask whether the converse implication holds, that is, whether the assumption that a set $X$ is a spectral set for a tuple $T$ implies that there exists a normal dilation with spectrum constrained to $X$ (or even to the {\em Shilov boundary} $\partial X$). 
There are cases when this is true (see \cite{BBSurv}), but in general the answer is {\em no}. 
For example, we already mentioned that Parrott's example \cite{Par70} of three commuting contractions that have no unitary dilation (hence also no normal dilation with spectrum contained in $\ol{\bD}^3$) does not involve a violation of von Neumann's inequality, in other words the tuple $T$ from \eqref{eq:Parrott} has $\ol{\bD}^3$ as a spectral set but has no unitary dilation. 

The situation was clarified by Arveson's work \cite{Arv72}, where the notion of {\em complete spectral set} was introduced. 
To explain this notion, we need matrix valued polynomials and rational functions. 
Matrix valued polynomials in several commuting (or noncommuting) variables, and the prescriptions for evaluating them at $d$-tuples of commuting (or noncommuting) operators, are defined in a similar manner to their definition in the one variable case in Remark \ref{rem:mat_val_pol}.
Once one knows how to evaluate a rational function in several variables at a commuting tuple, the passage to matrix valued rational functions is done similarly. 

Given a tuple $T\in B(\cH)^d$ of commuting contractions, we say that a set $X \subset \bC^d$ is a {\bf complete $K$-spectral set} for $T$, if $\sigma(T) \subseteq X$ and if for every matrix valued rational function $f$ that is regular on $X$, \eqref{eq:spectral_set} holds, where now for an $n \times n$ matrix valued rational function $\|f\|_{X,\infty} = \sup_{z\in X}\|f(z)\|_{M_n}$. 
If $X$ is a complete $K$-spectral set for $T$ with $K = 1$, then it is simply said to be a {\bf{complete spectral set}} for $T$.

\begin{theorem}[Arveson's dilation theorem \cite{Arv72}]\label{thm:Arv}
Let $T = (T_1, \ldots, T_d)$ be a tuple of commuting operators on a Hilbert space $\cH$. 
Let $X \subset \bC^d$ be a compact set and let $\partial X$ be the Shilov boundary of $X$ with respect to the algebra $rat(X) \subseteq C(X)$ of rational functions that are regular on $X$. 
Then $X$ is a complete spectral set for $T$ if and only if only if there exists a normal tuple $N = (N_1, \ldots, N_d)$ acting on a Hilbert space $\cK \supseteq \cH$, such that $\sigma(N) \subseteq \partial X$ and for every matrix valued rational function $f$ that is regular on $X$,
\[
f(T) = P_\cH f(N) \big|_\cH. 
\]
\end{theorem} 
Putting Arveson's dilation theorem together with some comments made above, we see that $\ol{\bD}^3$ is a spectral set for the triple $T$ from \eqref{eq:Parrott}, but it is not a complete spectral set. 
On the other hand, we know that for a pair of commuting contractions $T = (T_1, T_2)$, the bidsic $\ol{\bD}^2$ is a complete spectral set. 
Agler and M\raise.45ex\hbox{c}Carthy proved a sharper result: if $T = (T_1, T_2)$ acts on a finite dimensional space, and $\|T_1,\|,\|T_2\|<1$, then there exists a one dimensional complex algebraic subvariety $V \subseteq \bD^2$ (in fact, a so-called {\em distinguished variety}, which means that $\ol{V} \cap \partial (\bD^2) = \ol{V} \cap \bT^2$), such that $V$ is a complete spectral set for $T$ \cite{AM05}. 

If $X \subset \bC$ is a spectral set for an operator $T$, one may ask whether or not it is a complete spectral set. 
We close this section by mentioning some notable results in this direction. 
It is known that if $X\subset \bC$ is a compact spectral set for $T$ such that $rat(X) + \ol{rat(X)}$ is dense in $C(\partial X)$, then $X$ is a complete spectral set, and $T$ has a normal dilation with spectrum in $\partial X$. 
The condition is satisfied, for example, when $X$ is the closure of a bounded and simply connected open set (this result is due to Berger, Foias and Lebow (independently); see \cite[Theorem 4.4]{PauBook}). 
The same is true if $X$ is an annulus (Agler \cite{Agl85}), but false if $X$ is triply connected (Agler-Harland-Raphael \cite{AHR08} and Dritschel-McCullough \cite{DM05b}). 

If a pair of commuting operators $T = (T_1, T_2)$ has the {\bf{symmetrized bidisc}} $\Gamma:= \{(z_1 + z_2, z_1 z_2) : z_1, z_2 \in \ol{\bD}\}$ as a spectral set, then in fact $\Gamma$ is a complete spectral set for $T$ (Agler and Young \cite{AY00}).  
Pairs of operators having $\Gamma$ as a spectral set have a well developed model theory (see, e.g., Sarkar \cite{Sar15} and the references therein). 
Building on earlier work of Bhattacharyya, Pal and Roy \cite{BPR12}, and inspired by Agler and M\raise.45ex\hbox{c}Carthy's distinguished varieties result mentioned above, Pal and Shalit showed that if $\Gamma$ is a spectral set for a pair $T = (T_1, T_2)$ of commuting operators acting on a finite dimensional space, then there exists a {\em distinguished} one dimensional algebraic variety $V \subseteq \Gamma$ which is a complete spectral set for $T$ \cite{PalSh14}.

\part{A rapid overview of dilation theories}

\section{Additional results and generalizations of dilation theory}\label{sec:further}

\subsection{Some further remarks on \texorpdfstring{$N$}{N}-dilations}

The notion of a $1$-dilation of a single operator, which is usually referred to simply as {\em dilation}, has appeared through the years and found applications in operator theory; see e.g. \cite{BT14,CL01,Hal50} (the reader should be warned that the terminology is not universally accepted; for example, as we already mentioned, a power dilation is usually simply referred to as {\em dilation}. 
Even more confusingly, in \cite{BT14}, a {\em unitary $N$-dilation of $T$} means what we call here a unitary $1$-dilation of $T$ that acts on $\cH \oplus \bC^N$). 

Egerv\'{a}ry's simple construction \eqref{eq:U} of an $N$-dilation, and with it the concept of $N$-dilations, have been largely forgotten until \cite{LS14} seemed to revive some interest in it (see also \cite{Nag13}). The motivation was that the well-known Sz.-Nagy unitary (power) dilation of a contraction $T$ (given by Theorem \ref{thm:unidil}) always acts on an infinite dimensional space whenever $T$ is nonunitary, even if $T$ acts on a finite dimensional space. Arguably, one cannot say that an infinite dimensional object is better understood than a matrix. 
That's what led to the rediscovery of \eqref{eq:U} and thence to the dilation-theoretic proof of von Neumann's inequality that we presented, which has the conceptual advantage of never leaving the realm of finite dimensional spaces, in the case where $T$ acts on a finite dimensional space to begin with.  

Let $T = (T_1, \ldots, T_d)$ be a $d$-tuple of commuting operators acting on a Hilbert space $\cH$, and let $U = (U_1, \ldots, U_d)$ be a $d$-tuple of commuting operators acting on a Hilbert space $\cK \supseteq \cH$. 
We say that $U$ is a an {\bf{$N$-dilation}} of $T$ if 
\[
p(T) = P_\cH p(U)\big|_\cH
\]
for every polynomial in $d$ complex variables of degree less than or equal to $N$. 
We say that this dilation is a {\bf{unitary/normal dilation}} if every $U_i$ ($i=1,\ldots,d$) is unitary/normal. 
The construction \eqref{eq:U} shows that every contraction has a unitary $N$-dilation acting on $\cH^{N+1}$. 
In particular, it shows that every contraction acting on a finite dimensional space has a unitary $N$-dilation acting on a finite dimensional space, for all $N$. 

Curiously, it appears that the proof of Theorem \ref{thm:Andoisodil} cannot be modified to show that every pair of commuting contractions on a finite dimensional space has a commuting unitary $N$-dilation on a finite dimensional space, for all $N$. 
It was shown by M\raise.45ex\hbox{c}Carthy and Shalit that indeed such a finitary version of And\^{o}'s dilation theorem holds \cite{MS13}. 
Interestingly, the proof made use of And\^{o}'s dilation theorem. 
So, if one uses this finitary dilation theorem to prove von Neumann's inequality for pairs of matrices, one does not truly avoid infinite dimensional spaces. 
It is an open problem to come up with an explicit construction of a unitary $N$-dilation for commuting matrices. 

In fact, in \cite{MS13} it was also proved that a $d$-tuple of contractions acting on a finite dimensional space  has a unitary dilation if and only if for all $N$ it has a unitary $N$-dilation acting on a finite dimensional space. 
Likewise, it was shown that for such a tuple, the existence of a regular unitary dilation is equivalent to the existence of a {\em regular unitary $N$-dilation} (you can guess what that means) acting on a finite dimensional space, for all $N$. 
Additional finitary dilation results appeared, first in the setting of normal dilations of commuting tuples \cite{Coh15}, and then in the setting of $1$-dilations of noncommuting operators \cite[Section 7.1]{DDSS17}. 
A similar phenomenon was also observed in \cite{GS19}. 
At last, Hartz and Lupini found a finite dimensional version of Stinespring's dilation theorem (see Section \ref{subsec:CPSTI}), which provides a general principle by which one can deduce finite dimensional dilation theorems from their infinite dimensional counterparts \cite{HL+}.  

It is interesting to note that $N$-dilations found an application in simulating open quantum systems on a quantum computer \cite{HXK+}, and they also appeared in the context of quantum information theory \cite{LM18}. 
The notion of $N$-dilations also appeared in the dilation theory in general Banach spaces (about which will say a few words below), see \cite{FG19}.

\subsection{Models}\label{subsec:model}
Another direction in which dilation theory for commuting $d$-tuples has been developed is that of {\em operator models}. 
Roughly, the idea is that certain classes of $d$-tuples of operators can be exhibited as the compressions of a particular ``model" $d$-tuple of operators. 
We will demonstrate this with a representative example; for a broader point of view see \cite{MV93}, Chapter 14 in \cite{AMBook}, or the surveys \cite{AMDil} and \cite{SarSurvAppl}. 

Our example is the {\em $d$-shift} on the {\em Drury-Arveson space} $H^2_d$ \cite{Arv98,Drury} (see also the survey \cite{ShalitSurvey}). 
For a fixed $d$, we let $H^2_d$ denote the space of all analytic functions $f(z) = \sum_\alpha c_\alpha z^\alpha$ on the unit ball $\bB_d$ such that (with standard multi-index notation) 
\[
\|f\|^2_{H^2_d} := \sum_\alpha |c_\alpha|^2 \frac{\alpha!}{|\alpha|!} < \infty .
\]
This norm turns the space $H^2_d$ into a Hilbert space of analytic functions on $\bB_d$, such that point evaluation is bounded. 
In fact, $H^2_d$ is the reproducing kernel Hilbert space determined by the kernel $k(z,w) = \frac{1}{1-\langle z, w \rangle}$. 
Some readers might jump to their feet and object that this space is nothing but the good old {\em symmetric Fock space}, but it is fruitful and enlightening to consider it as a space of analytic functions (so please, sit down). 

For the record, let the reader know that the possibility $d = \infty$ is allowed, but we do not dwell upon this point. 

On $H^2_d$ there is a $d$-tuple of operators $S = (S_1, \ldots, S_d)$, called the {\em $d$-shift}, and defined by 
\[
S_i f(z) = z_i f(z) \quad , \quad i=1, \ldots, d,
\]
where $z=(z_1, \ldots, z_d)$ is the complex variable, and so $S_i$ is multiplication by the $i$th coordinate function $z_i$. 
The tuple $S$ is plainly a commuting tuple: $S_i S_j = S_j S_i$ (multiplication of functions is commutative). 
A short combinatorial exercise shows that $\sum S_i S_i^*$ is equal to the orthogonal projection onto the constant functions, and in particular $\sum S_i S_i^* \leq I$. 
Thus $S$ is a {\bf{row contraction}}, meaning that the row operator $[S_1\,\, S_2\,\, \cdots\,\, S_d] : H^2_d \oplus \cdots \oplus H^2_d \to H^2_d$ is a contraction. 
Another calculation reveals that $S$ is {\bf{pure}}, in the sense that $\sum_{|\alpha|=n}S^\alpha (S^\alpha)^* \xrightarrow{n \to \infty} 0$ in the strong operator topology. 

The remarkable fact is that $H^2_d$ is a {\em universal model} for pure commuting row contractions. 
I will now explain what these words mean. 
If $\cG$ is a Hilbert space, we can consider the space $H^2_d \otimes \cG$ (which can be considered as a Hilbert space of analytic $\cG$-valued functions), and the $d$-shift promotes to a shift $S \otimes I_\cG$ on $H^2_d \otimes \cG$, which is called a {\bf{multiple of the $d$-shift}}. 
A subspace $\cM \subseteq H^2_d \otimes \cG$ is said to be {\bf{coinvariant}} if it is invariant for $S_i^* \otimes I_\cG$ for all $i=1, \ldots, d$. 

\begin{theorem}[Universality of the $d$-shift]\label{thm:dshiftdil}
Let $T = (T_1, \ldots, T_d) \in B(\cH)^d$ be a pure, commuting row contraction. 
Then there exists a Hilbert space $\cG$ and a coinvariant subspace $\cM \subseteq H^2_d \otimes \cG$ such that $T$ is unitarily equivalent to the compression of $S \otimes I_\cG$ to $\cM$. 
\end{theorem} 
Thus, every row contraction $T$ is unitarily equivalent to the corestriction of a multiple of the $d$-shift to a coinvariant subspace. 
In particular, for every polynomial $p$ in $d$ variables, 
\be\label{eq:DruryIneq}
\|p(T)\| = \left\|P_\cM \left(p(S) \otimes I_\cG\right) \big|_\cM \right\| \leq \left\|p(S)\right\|, 
\ee
and this inequality replaces von Neumann's inequality in this setting (and this was Drury's motivation \cite{Drury}). 
It can be shown \cite{Drury} (see also \cite{Arv98,DavPittsPick}) that there exists no constant $C$ such that $\|P(S)\| \leq C\sup_{z\in \ol{\bB}_d}|p(z)|$, and in particular, commuting row contractions in general do not have normal dilations with spectrum contained in $\ol{\bB}_d$.

\subsection{Dilation theory for noncommutative operator tuples}\label{subsec:noncomm}

Dilation theory also plays a role in the analysis of tuples of noncommuting operators. 
Recall that a {\bf{row contraction}} is a tuple $T = (T_1, \ldots, T_d)$ such that $\sum T_i T_i^* \leq I$ (as in Section \ref{subsec:model}, we allow, but do not belabor, the case $d = \infty$, in which case the sum is understood in the strong-operator topology sense). 
A {\bf{row isometry}} is a tuple $V = (V_1, \ldots, V_d)$ such that $V_i^*V_j = \delta_{ij}I$, for all $i,j$. 
Thus, the operators $V_1, \ldots, V_d$ are all isometries which have mutually orthogonal ranges, and this is equivalent to the condition that the {\em row operator} $[V_1 \, V_2 \, \cdots \, V_d]$ is an isometry. 
The Sz.-Nagy isometric dilation theorem extends to the setting of (noncommuting) row contractions. 
The following theorem is due to Frazho \cite{Fra82} (the case $d=2$), Bunce \cite{Bun84} (the case $d \in \bN \cup \{\infty\}$) and Popescu \cite{Pop89} (who proved the existence of dilation in the case $d \in \bN \cup\{\infty\}$, and later developed a far reaching generalization of Sz-Nagy's and Foias's theory for noncommuting tuples and more). 

\begin{theorem}[Row isometric dilation of row contractions]\label{thm:rowiso_dil}
Let $T \in B(\cH)^d$ be a row contraction. 
Then there exists a Hilbert space $\cK$ containing $\cH$ and a row isometry $V = (V_1, \ldots, V_d) \in B(\cK)^d$ such that $V_i^* \big|\cH = T_i^*$ for all $i$. 
\end{theorem} 

There is also a very closely related dilation result, that shows that the shift $L = (L_1, \ldots, L_d)$ on the full Fock space is a universal model for {\em pure} row contractions, which reads similarly to Theorem \ref{thm:dshiftdil}, with the free shift $L$ replacing the commutative shift $S$. 
Correspondingly, there is a von Neumann type inequality $\|p(T)\| \leq \|p(L)\|$ which holds for every row contraction $T$ and every polynomial $p$ in noncommuting variables \cite{Pop89,Pop91}. 

Popescu has a large body of work in which this dilation/model theory is developed, applied, and generalized. 
In particular, the theory can be modified to accommodate tuples satisfying certain polynomial relations \cite{Pop06a} (see also \cite[Section 8]{ShalitSolel}) or tuples in certain {\em noncommutative polydomains} \cite{Pop16}. 

The isometric dilation of a row contractions lies at the heart of the free functional calculus for row contractions (see, e.g., \cite{Pop06}), and is important for understanding the algebraic structure of {\em noncommutative Hardy algebras} (also called {\em analytic Toeplitz algebras}, see \cite{DavPitts2}), as well as for the study and classification of algebras of bounded {\em nc analytic functions} on the {\em nc unit ball} and its subvarieties \cite{SSS18,SSS+}.

\subsection{Dilations in Banach spaces}\label{subsec:banach}
Until now, we have only considered operators on Hilbert spaces. 
But there are other kinds of interesting spaces, and the concept of dilations has appeared and been used in various settings. 
In the setting of Banach spaces, one may hope to dilate a contraction to an invertible isometry (that is, a surjective isometry); more generally one may wish to dilate a semigroup of operators to a group representation. 
Results along these lines, including a direct analogoue of Sz.-Nagy's unitary dilation theorem, were obtained by Stroescu; see \cite{Str73}. 

However, Banach spaces form a huge class of spaces, and the dilation theory in the context of general Banach spaces contains the additional aspect that one might like to ensure that the dilating space shares some properties with the original space. 
For example, if $T$ is a contraction on an $L^p$-space, one might wish to dilate to an invertible isometry acting on an $L^p$-space. 
Moreover, if $T$ is {\bf positive}, in the sense that $Tf \geq 0$ (almost everywhere) whenever $f \geq 0$ (almost everywhere), then one might hope to dilate to a positive invertible isometry. 
The following theorem is an example of the kind of result one can look for. 

\begin{theorem}[Akcoglu-Sucheston \cite{AS77}]\label{thm:AkcSuc}
Let $T : X \to X$ be a positive contraction on an $L^p$-space $X = L^p(\mu)$ ($1\leq p < \infty$). 
Then there exists another $L^p$-space $Y = L^p(\nu)$, a positive invertible isometry $U : Y \to Y$, a positive isometry $J : X \to Y$, and a positive projection $Q : Y \to Y$ such that 
\[
J T^n = Q U^n J \quad , \quad \textup{ for all } n\in \bN. 
\]
\end{theorem}
Note that even in the case $p=2$, this is not exactly Sz.-Nagy's dilation theorem: the assumptions are stronger, but so is the conclusion. 
For a modern approach to dilations in Banach spaces, generalizations, and also an overview of the history of the theory and its applications, see \cite{FG19}. 
Operator algebras are another class of spaces in which dilation theory was developed and applied; we will discuss this setting in Sections \ref{sec:opalg} and \ref{sec:CPsemigroups} below. 

\subsection{Dilations of representations of C*-correspondences}\label{subsec:correspondences}

A {\bf{Hilbert C*-module}} is a complex linear space $E$ which is a right module over a C*-algebra $\cA$, which carries an ``$\cA$-valued inner product" $\langle \cdot, \cdot \rangle : E \times E \to \cA$, that satisfies the following conditions:
\begin{enumerate}
\item $\langle x, x\rangle \geq 0$ for all $x \in E$, 
\item $\langle x, ya \rangle = \langle x, y \rangle a$ for all $x,y \in E$ and $a \in \cA$, 
\item $\langle x, y \rangle = \langle y, x \rangle^*$ for all $x,y \in E$, 
\item $\langle x, \alpha y  + \beta z  \rangle = \alpha \langle x,y \rangle + \beta \langle x,z \rangle$ for all $x,y,z \in E$ and $\alpha, \beta \in \bC$, 
\item $\|x\|:=\|\langle x,x \rangle\|^{1/2}$ is a norm on $E$ which makes $E$ into a Banach space. 
\end{enumerate}
The notion was introduced by Kaplansky \cite{Kap53} for the case where the C*-algebra $\cA$ is commutative, and then developed further by Paschke \cite{Pas73} and Rieffel \cite{Rie74} for general C*-algebras. 
It is now a standard tool in some fields in operator algebras; see \cite{LanBook} or Part I of \cite{SkeHab} for an introduction. 

Hilbert modules evolved into a more refined notion, called {\em Hilbert correspondences}, that involves a left action. 
A linear operator $T : E \to E$ is said to be {\bf{adjointable}} if there exists a linear operator $S : E \to E$ so that $\langle Tx, y \rangle = \langle x, Sy \rangle$ for all $x,y \in E$. 
One can show that every adjointable operator is a bounded right module map, but the converse is not true. 
The set of all adjointable operators on a Hilbert C*-correspondence $E$ is denoted $\mathscr{B}^a(E)$ or $\cL(E)$; it is a C*-algebra. 
A Hilbert {\bf{C*-correspondence} from $\cA$ to $\cB$} is a Hilbert $\cB$-module $E$ which also carries a left action of $\cA$ by adjointable operators. 
If $\cA = \cB$ then we say {\bf C*-correspondence over $\cA$}. 

Given a Hilbert C*-correspondence $E$ over the C*-algebra $\cA$, a {\bf{covariant representation}} of $E$ on a Hilbert space $\cH$ is a pair $(T, \sigma)$ where $T$ is linear map $T : E \to B(\cH)$ and $\sigma : \cA \to B(\cH)$ is a nondegenerate $*$-representation such that $T(a\cdot x \cdot b) = \sigma(a)T(x)\sigma(b)$ for all $a,b \in \cA$ and $x \in E$. 
A covariant representation is said to be {\bf{contractive}}/{\bf{completely contractive}}/{\bf{bounded}}, etc., if $T$ is contractive/completely contractive/bounded, etc; it is said to be {\bf{isometric}} if $T(x)^*T(y) = \sigma(\langle x,y \rangle)$ for all $x,y \in E$. 

Muhly and Solel proved that every completely contractive covariant representation $(T,\sigma)$ of $E$ on $\cH$ has an {\bf{isometric dilation}} $(V,\pi)$ of $E$ on $\cK \supseteq \cH$, \cite[Theorem 3.3]{MS98}. 
By this, we mean an isometric covariant representation $(V,\pi)$ on a Hilbert space $\cK$ that contains $\cH$, such that 
\begin{enumerate}
\item $\cH$ is reducing for $\pi$, and $P_\cH\pi(a)\big|_\cH = \sigma(a)$ for all $a \in \cA$, 
\item $P_\cH V(x) \big|_\cH = T(x)$ for all $x \in E$, 
\item $P_\cH V(x) \big|_{\cH^\perp} = 0$ for all $x \in E$.  
\end{enumerate}
Moreover, they proved that such an isometric dilation can be chosen to be minimal in a certain sense, and that the minimal isometric dilation is unique up to unitary equivalence (the third condition that an isometric dilation is required to satisfy looks more like something that should be called a {\em coextension}, it is actually a consequence of minimality; sometimes it is not required). 
The isometric dilation theorem was used in \cite{MS98} to analyze the representation theory of the {\bf{tensor algebra}} $\cT_+(E)$, which is a particular nonselfadjoint operator algebra, formed from the C*-correspondence in a way which we shall not go into. 
This has shed light on problems regarding an enormous class of operator algebras. 
Remarkably, Muhly and Solel's minimal isometric dilation enjoys also a commutant lifting theorem, and this, in turn, can lead to a Nevanlinna-Pick type interpolation theorem for so-called {\em noncommutative Hardy algebras}, with a proof reminiscent to the one we gave in Section \ref{sec:Pick} (see \cite{MS04}). 

This dilation theorem of Muhly and Solel is a far reaching generalization of Sz.-Nagy's isometric dilation theorem (Theorem \ref{thm:isodil}). 
In fact, the latter is obtained from the simplest case $E = \cA = \bC$ of Muhly and Solel's theorem. 
The row-isometric dilation of a row contraction (Theorem \ref{thm:rowiso_dil}) is obtained as the ``second simplest" case $E = \bC^d$ and $\cA = \bC$. 
Muhly and Solel's isometric dilation theorem also reduces to dilation results in the context of crossed product and semi-crossed product operator algebras, as well in graph C*-algebras.

On the other hand, And\^{o}'s theorem, for example, is not a special case of Muhly and Solel's isometric dilation theorem -- a single C*-correspondence is not sufficient to encode a pair of commuting contractions. 
The missing ingredient is the notion of {\em product systems}. 
A {\bf{product system}} over a monoid (i.e., a semigroup with unit $e$) $\cS$ is a family $E = \{E_s\}_{s \in \cS}$ of C*-correspondences over a C*-algebra $\cA$, such that for every $s,t\in\cS$ there exists an isomorphism of correspondences (i.e., an adjointable surjective isometry which is a bimodule map) $u_{s,t} : E_s \odot E_t \to E_{st}$ such that the multiplication $x_s y_t := u_{s,t}(x_s \odot y_t)$ is associative 
\[
(w_r x_s)y_t = w_r (x_s y_t). 
\]
(Here $E_s \odot E_t$, denotes the {\em internal} (or {\em interior}) {\em tensor product} of $E_s$ and $E_t$, sometimes also denoted $E_s \otimes E_t$; see \cite[Chapter 4]{LanBook}.) A {\bf{covariant representation}} of a product system $E = \{E_s\}_{s \in \cS}$ on $\cH$ is a family $T = \{T_s\}_{s\in \cS}$ such that for all $s \in \cS$, the pair $(T_s, T_e)$ is a covariant representation of $E_s$ on $\cH$, which satisfies in addition
\[
T_{st}(x_s \odot y_t) = T_s(x_s) T_t(y_t)
\]
for all $s,t \in \cS$ and all $x_s \in E_s$ and $y_t \in E_t$. 
An {\bf{isometric}} representation of $E$ on $\cK$ is a covariant representation $V = \{V_s\}_{s\in \cS}$ of $E$ such that for all $s \in \cS$, the pair $(V_s, V_e)$ is an isometric representation. 
One then says that $V$ is an {\bf{isometric dilation}} of $T$ if 
\begin{enumerate}
\item $\cH$ is reducing for $\pi$, and $P_\cH T_e(a)\big|_\cH = V_e(a)$ for all $a \in \cA$, 
\item $P_\cH V(x) \big|_\cH = T(x)$ for all $x \in E_s$. 
\end{enumerate}
The theory of isometric dilations of completely contractive representations of product systems, is analogous to the theory of isometric dilations of semigroups of contractions. 
Moreover, some of the proofs rely on the same ideas and approaches, albeit at a technical sophistication level that is one order of magnitude higher. 
In fact, several results (but not all) can be {\em reduced} to the case of operator semigroups (see \cite{ShalitReprep}). 
We will see in Section \ref{sec:CPsemigroups} that the dilation theory of covariant representations is important for the dilation theory of CP-semigroups. 

Here are some sample results. 
Solel proved a version of And\^{o}'s theorem in this setting: every completely contractive covariant representation of a product system over $\bN^2$ has an isometric dilation \cite[Theorem 4.4]{Sol06}. 
Solel also proved an analogue of Theorem \ref{thm:regular} (regular dilations) for product systems over $\bN^d$ using a direct proof \cite{Sol08}  (see also \cite{Ska09,SkaZac08}). 
Shalit later found another proof by reducing to the case of operator semigroups \cite{ShalitReprep}. 
The method of \cite{ShalitReprep} was later used in \cite[Section 5]{ShalitE0Dil} to prove a counterpart to Theorem \ref{thm:iso2unidil} (see also \cite{Ful11}). 
Vernik generalized Opela's result on dilations of contractions commuting according to a graph (see Section \ref{subsec:semi}) to the setting of product system representations \cite{Ver16}. 
All of the above results reduce to their counterparts that we discussed in earlier sections, when one considers the special case $\cA = E_s = \bC$ for all $s \in \cS$, where $\cS$ is the appropriate monoid.

\section{The operator algebraic perspective}\label{sec:opalg}

The operator algebraic outlook on dilation theory began with Arveson's visionary papers \cite{Arv69,Arv72}. 
Arveson sought to develop a systematic study of nonselfadjoint operator algebras, which is based on studying the relations between an operator algebra and the C*-algebras that it generates. 
From the outset, the approach was general and powerful enough to cover also certain operator spaces. 
On the one hand, this approach opened the door by which operator algebraic techniques entered into operator theory: these techniques have shed light on classical dilation theory, and they also created a powerful framework by which new dilation results could be obtained. 
On the other hand, the general philosophy of dilation theory found its way into operator algebras, and has led to remarkable developments. 

In this section I will present Stinespring's dilation theorem, and how Arveson's extension theorem and his notion of {\em C*-dilation} has made Stinespring's theorem into a ``dilation machine" that produces and explains dilation results in operator theory. 
Then I will briefly discuss how dilation theory is related to the notions of boundary representations and the C*-envelope, which lie at the heart of the above mentioned analysis of the relationship between on operator algebra/space and the C*-algebras that it generates. 

I will not attempt to cover all the manifold ways in which dilation theory appears in the theory of operator algebras, and I'll just mention a few (of my favorite) recent examples: \cite{DOS18,KakSha19,KR19,KenSha16}. 
The reader is referred to the survey \cite{DFK18} or the paper \cite{DK11} in order to get an idea of the role it plays, in particular in operator algebras related to dynamical systems and semicrossed products.

\subsection{Completely positive maps and Stinespring's theorem}\label{subsec:CPSTI}

An {\bf{operator space}} is a subspace $\cM \subseteq B(\cH)$ of the bounded operators on some Hilbert space $\cH$. 
We say that $\cM$ is {\bf{unital}} if $1 = I_\cH \in \cM$. 
If $\cM$ is a subalgebra of $B(\cH)$, then it is called an {\bf{operator algebra}} (note that operator algebras are not assumed to be closed under the adjoint). 
A unital operator space $\cM$ is said to be an {\bf{operator system}} if it is closed under the adjoint operation. 
Since every C*-algebra can be represented faithfully on a Hilbert space, we can consider a subspace of a C*-algebra as an operator space (and likewise for unital operator spaces, operator algebras and operator systems). 
C*-algebras are operator algebras, and unital C*-algebras are operator systems, of course. 

An operator space $\cM \subseteq B(\cH)$ inherits from $B(\cH)$ a norm and a notion of positivity: an element $a \in \cM$ is said to be {\bf{positive}}, if it is positive as an operator on $\cH$, i.e., $\langle ah,h \rangle \geq 0$ for all $h \in \cH$. 
Operator systems are spanned by their positive elements, indeed, if $a \in \cM$ then its real and imaginary parts are also in $\cM$, and if $a$ is selfadjoint then $\frac{1}{2}(\|a\|\cdot 1\pm a) \geq 0$ and the difference of these two positive elements is $a$. 

As a consequence, it makes sense to speak of positive maps. 
If $\cM$ and $\cN$ are operator systems, a linear map $\phi : \cM \to \cN$ is said to be {\bf{positive}} if it takes positive elements to positive elements. 
The matrix spaces $M_n(\cM) \subseteq M_n(B(\cH))=B(\cH^n)$ and $M_n(\cN)$ are also operator systems, and then $\phi$ induces a linear map $\phi^{(n)} : M_n(\cM) \to M_n(\cN)$
\[
\phi^{(n)} = \phi \otimes {\bf id}_{M_n} : M_n(\cM) =  \cM \otimes M_n \to M_n(\cN) = \cN \otimes M_n
\]
acting elementwise as 
\[
\phi^{(n)} : (a_{ij})_{i,j=1}^n \mapsto (\phi(a_{ij}))_{i,j=1}^n \in M_n(\cN). 
\]
The map $\phi$ is said 
be {\bf{completely positive}} (or {\bf{CP}} for short) if $\phi^{(n)}$ is positive for all $n$. 
Likewise, $\phi$ is said to be {\bf{completely contractive}} (or {\bf{CC}} for short) if $\phi^{(n)}$ is contractive for all $n$. 
A map is {\bf{UCP}} if it is a unital CP map, and {\bf{UCC}} if it is a unital CC map. 

Completely positive maps were introduced by Stinespring \cite{Sti55}, but it was Arveson who observed how important they are and opened the door to their becoming an indispensable tool in operator theory and operator algebras \cite{Arv69}. 
There are several excellent sources to learn about operator spaces/systems and completely positive (and bounded) maps; see for example \cite{PauBook} and \cite{PisBook}. 

Completely positive maps arise also in mathematical physics in a natural way \cite{KrausBook}. 
The evolution of an open quantum system is described by a semigroup of completely positive maps \cite{DaviesBook}, and noisy channels in quantum information theory are modelled as trace preserving completely positive maps \cite{NCBook}. 
In quantum probability \cite{Par12}, semigroups of unit preserving completely positive maps play the role of Markov semigroups. 

The simplest examples of completely positive maps are $*$-homomorph\-isms between C*-algebras. 
Next, a map of the form $B(\cK) \ni T \mapsto V^* T V\in B(\cH)$, where $V$ is some fixed operator in $B(\cH,\cK)$, is readily seen to be completely positive. Since compositions of CP maps are evidently CP, we see that whenever $\cA$ is a C*-algebra, $\pi : \cA \to B(\cK)$ is a $*$-homomorphism, and $V \in B(\cH,\cK)$, then the map $a \mapsto V^*\pi(a) V$ is a CP map. 
The following fundamental theorem shows that essentially all CP maps on C*-algebras are of this form. 

\begin{theorem}[Stinespring's theorem \cite{Sti55}]\label{thm:Stinespring}
Let $\cA$ be a unital C*-algebra and let $\phi : \cA \to B(\cH)$ be a CP map. 
Then there exists a Hilbert space $\cK$, an operator $V \in B(\cH,\cK)$, and a $*$-representation $\pi: \cA \to B(\cK)$ such that 
\[
\phi(a) = V^*\pi(a) V \quad , \quad \textup{for all } a \in \cA. 
\]
The tuple $(\pi,\cK,V)$ can be chosen such that $\cK = [\pi(\cA)\cH]$ --- the smallest closed subspace containing $\pi(a)h$ for all $a \in \cA$ and $h \in \cH$ --- and in this case case the triple $(\pi,\cK,V)$ is unique up to unitary equivalence. 
\end{theorem} 
\begin{proof}
On the algebraic tensor product $\cA \otimes \cH$, we define a semi-inner product by setting $\langle a \otimes g, b\otimes h \rangle = \langle  g, \phi(a^*b) h \rangle_\cH$ and extending sesquilinearly (the complete positivity guarantees that this is a positive semidefinite form). 
Quotienting out the kernel and then completing gives rise to the Hilbert space $\cK$. 
The image of all the elementary tensors $b\otimes h \in \cA \otimes \cH$ in $\cK$ form a total set, and we continue to denote these images as $b \otimes h$. 
One needs to check that for every $a \in \cA$, the map $\pi(a) : b \otimes h  \mapsto ab \otimes h$ extends to a well defined, bounded linear operator on $\cK$. 
Once this is done, it is easy to verify that the map $a \mapsto \pi(a)$ is a $*$-homomorphism. 

To recover $\phi$, we define $V : \cH \to \cK$ by $V(h) = 1 \otimes h$, and then all that remains to do is to compute $\langle V^* (a \otimes h), g \rangle = \langle a \otimes h, V(g) \rangle = \langle a \otimes h, 1 \otimes g \rangle = \langle h, \phi(a^*) g \rangle$, so $V^*(a \otimes h) = \phi(a) h$, and thus
\[
V^* \pi(a) Vh = V^*(a\otimes h) = \phi(a) h, 
\]
as required. 
If $[\pi(\cA)\cH] \subsetneq \cK$, then we replace $\cK$ with $[\pi(\cA) \cH]$, and obtain a minimal representation. 
The uniqueness is a standard matter, and is left to the reader. 
\end{proof}
If $\cK = [\pi(\cA)\cH]$, then $(\pi,\cK,V)$ (or just $\pi$ sometimes) is called the {\bf{minimal Stinespring representation of $\phi$}}.  
\begin{remark}
If $\phi$ is unital, then $1 = \phi(1) = V^* \pi(1) V = V^*V$, so $V$ is an isometry. 
In this case it is convenient to identify $\cH$ with $V\cH \subseteq \cK$, and the Stinespring representation manifests itself as a dilation
\[
\phi(a) = P_\cH \pi(a) \big|_\cH. 
\]
In this situation, the minimal Stinespring representation is referred to as the {\bf minimal Stinespring dilation} of $\phi$. 
\end{remark}

\subsection{Arveson's extension theorem and C*-dilations}\label{subsec:dil_machine}

The utility of completely positive maps comes from the following extension theorem of Arveson. 
For a proof, see Arveson's paper or Paulsen's book \cite[Chapter 7]{PauBook}. 
\begin{theorem}[Arveson's extension theorem \cite{Arv69}]\label{thm:Arveson}
Let $\cM$ be an operator system in a C*-algebra $\cA$, and let $\phi : \cM \to B(\cH)$ be a CP map. 
Then there exists a CP map $\hat{\phi} : \cA \to B(\cH)$ such that $\|\hat{\phi}\| = \|\phi\|$ and which extends $\phi$, i.e., $\hat{\phi}(a) = \phi(a)$ for all $a \in \cM$. 
\end{theorem}

We will now see how the combination of Stinespring's dilation theorem and Arveson's extension theorem serve as kind of all purpose ``dilation machine", that produces dilation theorems in varied settings. 

Let $1 \in \cM \subseteq \cA$ be a unital operator space. 
A linear map $\phi : \cM \to B(\cH)$ is said to have a {\bf C*-dilation} to $\cA$ if there exists a $*$-representation $\pi : \cA \to B(\cK)$, $\cK \supseteq \cH$, such that 
\[
\phi(a) = P_\cH \pi(a) \big|_\cH \\,\, , \,\, \textup{ for all } a \in \cM. 
\]
Arveson showed that every UCP map is UCC, and that, conversely, every UCC map as above extends to a UCP map $\widetilde{\phi} : \widetilde{\cM} := \cM + \cM^* \to B(\cH)$ given by $\widetilde{\phi}(a+b^*) = \phi(a) + \phi(b)^*$. 
Combining this basic fact with Theorems \ref{thm:Stinespring} and \ref{thm:Arveson} we obtain the following versatile dilation theorem. 

\begin{theorem}\label{thm:Cstar_dil}
Every UCC or UCP map has a C*-dilation. 
\end{theorem}
Arveson's dilation theorem\footnote{The reader should be aware that Theorem \ref{thm:Cstar_dil} is sometimes referred to as {\em Arveson's dilation theorem}, whereas I used this name already for the more specific Theorem \ref{thm:Arv}.} (Theorem \ref{thm:Arv}) follows from the above theorem, once one carefully works through the delicate issues of joint spectrum, Shilov boundary and functional calculus (see \cite{Arv72}). 
We shall illustrate the use of the dilation machine by proving Arveson's dilation theorem for the simple but representative example of the polydisc $\ol{\bD}^d$.

\begin{theorem}[Arveson's dilation theorem for the polydisc]\label{thm:Arv_dil_poly}
A $d$-tuple of commuting contractions $T = (T_1, \ldots, T_d)$ has a unitary dilation if and only if the polydisc $\ol{\bD}^d$ is a complete spectral set for $T$. 
\end{theorem}
\begin{proof}
To relate the statement of the theorem to the language of Section \ref{sec:spectral}, we note that the Shilov boundary of $\ol{\bD}^d$ is just the torus $(\partial \ol{\bD})^d = \bT^d$, and therefore a unitary dilation is nothing but a normal dilation with joint spectrum contained in $\bT^d$. 
Recall that $\ol{\bD}^d$ being a spectral set is equivalent to that 
\be\label{eq:CSS}
\|p(T)\| \leq \|p\|_\infty := \sup_{z\in \ol{\bD}^d} \|p(z)\|
\ee
for every matrix valued polynomial $p$ (since $\ol{\bD}^d$ is convex it suffices to consider matrix valued polynomials, and there is no need to worry about matrix valued rational functions).   

If $U = (U_1, \ldots, U_d)$ is a tuple of commuting unitaries and $p$ is a matrix valued polynomial, then, using the spectral theorem, it is not hard to see that $\|p(U)\| = \sup_{z \in \sigma(U)}\|p(z)\| \leq \|p\|_\infty$. 
Now, if $U$ is a dilation of $T$ then $\|p(T)\| \leq \|p(U)\|$, and so the inequality \eqref{eq:CSS} holds. 
That was the easy direction. 

Conversely, suppose that $\ol{\bD}^d$ is a complete spectral set for a commuting tuple $T \in B(\cH)^d$, that is, suppose that \eqref{eq:CSS} holds for every matrix valued polynomial $p$. 
Let $\cM = \bC[z_1, \ldots, z_d]$ be the space of polynomials in $d$ variables, considered as a unital subspace of the C*-algebra $C(\bT^d)$, equipped with the usual supremum norm. 
It is useful to note that 
\[
\sup_{z \in \ol{\bD}^d}\|p(z)\| = \sup_{z \in \bT^d} \|p(z)\| ,
\]
by applying the maximum modulus principle in several variables. 
The fact that $\ol{\bD}^d$ is a complete spectral set for $T$ implies that the unital map $\phi : \cM \to B(\cH)$ given by $\phi(p) = p(T)$, is completely contractive. 
By Theorem \ref{thm:Cstar_dil}, $\phi$ has a C*-dilation $\pi : C(\bT^d) \to B(\cK)$, such that 
\[
p(T) = \phi(p) = P_\cH \pi(p) \big|_\cH \,\, , \,\, p \in \cM. 
\]
Now, $\pi$ is a $*$-representation, and the coordinate functions $z_1, \ldots, z_d \in C(\bT^d)$ are all unitary, so $U_i = \pi(z_i)$, $i = 1, \ldots, d$, are commuting unitaries. 
Since $\pi(p) = p(\pi(z_1), \ldots, \pi(z_d))$, we find that 
\[
p(T) = P_\cH p(U) \big|_\cH
\]
for all $p \in \cM$, that is, $U$ is a unitary dilation of $T$, as required. 
\end{proof}

Following Arveson, the above method has been used extensively for proving the existence of dilations in certain situations. 
The burden is then shifted from the task of {\em constructing} a dilation, to that of showing that certain naturally defined maps are UCP or UCC. 
In other words, by proving an inequality one obtains an existence proof --- a good bargain from which analysts have profited for a century. 
Of course, ``good bargain"  does not mean that we cheat, we still need to prove something. 
Let me give an example of how this works. 

\begin{example}
We now prove that for every contraction $T \in B(\cH)$, the map $\Phi : \bC[z] \to B(\cH)$ given by $\Phi(p) = p(T)$ is UCC. 
Combining this with Arveson's dilation theorem for the disc (the case $d = 1$ in Theorem \ref{thm:Arv_dil_poly}), we obtain a genuinely new proof of Sz.-Nagy's unitary dilation theorem (Theorem \ref{thm:unidil}). 
The reader should be able to adapt the details of this proof to get a proof that every row contraction has a row-isometric dilation, and that every pure commuting row contraction can be modelled by the $d$-shift (for hints, see the introduction of \cite{Pop99} or \cite{AMDil}, respectively). 

Let $S$ be the unilateral shift on $\ell^2$, and let $e_n$ denote the $n$th standard basis vector in $\ell^2$. 
If $T$ is a contraction and $r \in (0,1)$, we let $D_{rT} = (I - r^2 T T^*)^{1/2}$, and define $K_r(T) : \cH \to \ell^2 \otimes \cH$ by 
\[
K_r(T)h = \sum_n e_n \otimes \big( r^{n} D_{rT} T^{n*}h \big) , 
\]
for all $h \in \cH$. 
We compute: 
\begin{align*}
K_r(T)^* K_r(T) h &= \sum r^{2n} T^n D^2_{rT} T^{n*} h \\
&=  \sum r^{2n} T^n (I - r^2 T T^*) T^{n*} h \\
&= \sum r^{2n} T^n T^{n*} h- \sum r^{2(n+1)} T^{n+1} T^{(n+1)*}  h = h
\end{align*}
so that $K_r(T)$ is an isometry.  
On $C^*(S)$ we define a UCP map
\[
\phi_r(a) = K_r(T)^* (a \otimes I) K_r(T) .
\]
We compute that  
\begin{align*}
\phi_r(S) &= K_r(T)^* (S\otimes I) K_r(T) h & \\
&= K_r(T)^* \sum e_{n+1} \otimes \big( r^{n} D_{rT} T^{n*}h \big) \\
&= \sum r^{2n+1} T^{n+1}D_{rT}^2 T^{n*} h = rTh. 
\end{align*}
Likewise, $\phi_r(S^n) = r^{n} T^n$ for all $n \in \bN$. 

Now we define a UCP map $\Phi := \lim_{r\nearrow 1} \phi_r$. 
Then $\Phi(S^n) = T^n$ for all $n$. 
We see that the map $p(S) \mapsto p(T)$ is UCC. 
But $S$ is unitarily equivalent to the multiplication operator $M_z$ on $H^2$ (see Section \ref{subsec:shift}). 
Thus, for every matrix valued polynomial $p$
\[
\|p(T)\| \leq \|p(S)\| = \|p(M_z)\| = \sup_{|z|=1}\|p(z)\|, 
\]
so $\ol{\bD}$ is a complete spectral set for $T$. 
By Theorem \ref{thm:Arv_dil_poly}, $T$ has a unitary dilation. 
\end{example}

\subsection{Boundary representations and the C*-envelope}\label{subsec:boundary}

The ideas in this section are best motivated by the following classical example. 

\begin{example}\label{ex:discalg}
Consider the disc algebra $A(\bD)$, which is equal to the closure of the polynomials with respect to the norm $\|p\|_\infty = \sup_{z \in \ol{\bD}}|p(z)|$. 
The disc algebra is an operator algebra, being a subalgebra of the C*-algebra $C(\ol{\bD})$ of continuous functions on the disc $\ol{\bD}$. 
Moreover, $C^*(A(\ol{\bD})) = C(\ol{\bD})$, that is, the C*-subalgebra generated by $A(\bD) \subseteq C(\ol{\bD})$ is equal to $C(\ol{\bD})$. 
However, $C(\ol{\bD})$ is not determined uniquely by being ``the C*-algebra generated by the disc algebra". 
In fact, by the maximum modulus principle, $A(\ol{\bD})$ is also isometrically isomorphic to the closed subalgebra of $C(\bT)$ generated by all polynomials, and the C*-subalgebra of $C(\bT)$ generated by the polynomials is equal to $C(\bT)$. 

Now, $C(\bT)$ is the  quotient of $C(\ol{\bD})$ by the ideal of all continuous functions vanishing on the circle $\bT$. 
If $\pi : C(\ol{\bD}) \to C(\bT)$ denotes the quotient map, then $\pi(f) = f\big|_\bT$, and we note, using the maximum principle again, that $\pi$ is isometric on $A(\bD)$. 
It turns out that $\bT$ is the minimal subset $E \subseteq \ol{\bD}$ such that the map $f \mapsto f\big|_E$ is isometric on $A(\bD)$. 
\end{example}

The above phenomenon arises in all {\em uniform algebras}, that is, in all unital subalgebras $\cA$ of $C(X)$ that separate the points of $X$, where $X$ is some compact Hausdorff space. 
For every such algebra there exists a set $\partial_\cA \subseteq X$ --- called the {\bf{Shilov boundary}} of $\cA$ --- which is the unique minimal closed subset $E \subseteq X$ such that $f \mapsto f\big|_E$ is isometric (see \cite{Gamelin} for the theory of uniform algebras). 

In the spirit of noncommutative analysis, Arveson sought to generalize the Shilov boundary to the case where the commutative C*-algebra $C(X)$ is replaced by a noncommutative C*-algebra $\cB = C^*(\cA)$ generated by a unital operator algebra $\cA$. 
An ideal $\cI \triangleleft \cB$ is said to be a {\bf{boundary ideal}} for $\cA$ in $\cB$, if the restriction of the quotient map $\pi : \cB \to \cB / \cI$ to $\cA$ is completely isometric. 
The {\bf{Shilov ideal}} of $\cA$ in $\cB$ is the unique largest boundary ideal for $\cA$ in $\cB$. 
If $\cJ$ is the Shilov ideal of $\cA$ in $\cB$, then {\bf{C*-envelope}} of $\cA$ is defined to be the C*-algebra $C_e^*(\cA) = \cB / \cJ$. 
(The above notions were introduced in \cite{Arv69,Arv72}, but it took some time until the terminology settled down. 
A good place for the beginner to start learning this stuff is \cite{PauBook}.)

\begin{example}\label{ex:shilov}
If $\cB = C(X)$ is a commutative C*-algebra generated by the uniform algebra $\cA$, then the Shilov ideal of $\cA$ is just the ideal $\cI_{\partial_\cA}$ of functions vanishing on the Shilov boundary $\partial_\cA$. 
In this case $C^*_e(\cA) = C(\partial_\cA)$. 
\end{example}

The C*-envelope $C_e^*(\cA) = \cB / \cJ = \pi(\cB)$ has the following universal property: if $i : \cA \to \cB'$ is a completely isometric homomorphism such that $\cB' = C^*(i(\cA))$, then there exists a unique surjective $*$-homomorphism $\rho : \cB' \to C_e^*(\cA)$ such that $\pi(a) = \rho(i(a))$ for all $a \in \cA$. 
It follows that the C*-envelope depends only on the structure of $\cA$ as an operator algebra, not on the concrete realization $\cA \subseteq \cB$ with which we started. 
Thus, if $\cA_i \subseteq \cB_i = C^*(\cA_i)$ for $i=1,2$ have trivial Shilov ideals, then every completely isometric homomorphism $\phi : \cA_1 \to \cA_2$ extends to $*$-isomorphism $\rho :\cB_1 \to \cB_2$. 

In fact, the algebraic structure is not essential here, and the above notions also make sense for unital operator spaces. 
For some purposes, it is most convenient to work with operator systems, and focus is then shifted to this case (as in the next subsection). 
For example, the Shilov ideal of an operator system $\cS \subseteq \cB = C^*(\cS)$ is the largest ideal $\cI \triangleleft \cB$ such that the quotient map $\pi : \cB \mapsto \cB / \cI$ restricts to a complete isometry on $\cS$, etc. 

How does one find the Shilov ideal? 
Let us return to Example \ref{ex:discalg} (recalling also Example \ref{ex:shilov}). 
If $A(\ol{\bD})$ is given as a subalgebra of $C(\ol{\bD})$, how could one characterize its Shilov boundary? 
A little bit of function theory shows that a point $z \in \ol{\bD}$ is in the unit circle $\bT$ if and only if it has a unique representing measure. 
Recall that when we have a uniform algebra $\cA \subseteq C(X)$, a probability measure $\mu$ is said to be a {\bf{representing measure}} for $x$ if 
\[
f(x) = \int_X f d \mu
\]
for all $f \in \cA$. 
For example, the Lebesgue measure on the circle is a representing measure for the point $0$, because 
\[
f(0) = \frac{1}{2\pi} \int_0^{2\pi} f(e^{it}) dt
\] 
for every $f \in A(\ol{\bD})$, as an easy calculation shows. 
Every point can be represented by the delta measure $\delta_z$; the points on the circle are singled out by being those with a unique representing measure, that is, they can be represented {\em only} by the delta measure. 
In the general case of a uniform algebra $\cA \subseteq C(X)$, the points in $X$ that have a unique representing measure are referred to as the {\bf{Choquet boundary}} of $\cA$. 
It is not hard to show that the Choquet boundary of $A(\ol{\bD})$ is $\bT$. 
In general, the Choquet boundary of a uniform algebra is dense in the Shilov boundary.

Let return to the noncommutative case, so let $\cA \subseteq \cB = C^*(\cA)$ be again a unital operator algebra generating a C*-algebra. 
Point evaluations correspond to the irreducible representations of a commutative C*-algebra, and probability measures correspond to states, that is, positive maps into $\bC$. 
With this in mind, the reader will hopefully agree that the following generalization is potentially useful:
an irreducible representation $\pi : \cB \to B(\cH)$ is said to be a {\bf{boundary representation}} if the only UCP map $\Phi : \cB \to B(\cH)$ that extends $\pi\big|_\cA$ is $\pi$ itself. 

Arveson proved in \cite{Arv69} that if an operator algebra $\cA \subseteq \cB = C^*(\cA)$ has {\bf{sufficiently many boundary representations}}, in the sense that 
\[
\|A\| = \sup \{\|\pi^{(n)} (A)\| : \pi : \cB \to B(\cH_\pi) \textup{ is a boundary representation}\}
\]
for all $A \in M_n(\cA)$, then the Shilov ideal exists, and is equal to the intersection of all boundary ideals. 
For some important operator algebras, the existence of sufficiently many boundary representations was obtained (see also \cite{Arv72}), but the problem of existence of boundary representations in general remained open almost 45 years\footnote{The existence of the C*-envelope was obtained much earlier, without making use of boundary representations; see \cite{PauBook}.}. 
Following a sequence of important developments \cite{MS98a,DM05,Arv08}, Davidson and Kennedy proved that every operator system has sufficiently many boundary representations \cite{DK15}. 
Their proof implies that every unital operator space, and in particular every unital operator algebra, has sufficiently many boundary representations as well.

\subsection{Boundary representations and dilations}\label{subsec:bound_dil}

It is interesting that the solution to the existence problem of boundary representations was obtained through dilations. 
Davidson and Kennedy worked in the setting of operator systems, so let us follow them in this subsection. 

If $\cS \subseteq \cB = C^*(\cS)$ is an operator system inside the C*-algebra that it generates, an irreducible representation $\pi : \cB \to B(\cH)$ is a boundary representation if the only UCP map $\Phi : \cB \to B(\cH)$ that extends $\pi\big|_\cS$ is $\pi$ itself. 
This leads to the following definition: a UCP map $\phi : \cS \to B(\cH)$ is said to have the {\bf{unique extension property}} if there exists a unique UCP map $\Phi : \cB \to B(\cH)$ that extends $\phi$, and, moreover, {\em this $\Phi$ is a $*$-representation}. 
Thus, an irreducible representation $\pi$ is a boundary representation if and only if the restriction $\pi\big|_\cS$ has the unique extension property. 

Now, the unique extension property nicely captures the idea that $\cS$ is in some sense rigid in $\cB$, but it is hard to verify it in practice. 
The following notion is more wieldy. 
A UCP map $\psi : \cS \to B(\cK)$ is said to be a {\bf{dilation}} of a UCP map $\phi : \cS \to B(\cH)$ if $\phi(a) = P_\cH \psi(a)\big|_\cH$ for all $a \in \cS$. 
The dilation $\psi$ is said to be a {\bf{trivial}} dilation if $\cH$ is reducing for $\phi(\cS)$, that is, $\psi = \phi \oplus \rho$ for some UCP map $\rho$. 
A UCP map $\phi$ is said to be {\bf{maximal}} if it has only trivial dilations. 

Penetrating observations of Muhly and Solel \cite{MS98a}, and consequently Dritschel and McCullough \cite{DM05}, can be reformulated as the following theorem. 
The beauty is that the notion of maximality is intrinsic to the operator system $\cS$, and does not take the containing C*-algebra $\cB$ into account (similar reformulations exist for the categories of unital operator spaces and operator algebras). 

\begin{theorem}
A UCP map $\phi : \cS \to B(\cH)$ has the unique extension property if and only if it is maximal. 
\end{theorem}

Following Dritschel and McCullough's proof of the existence of the C*-envelope \cite{DM05} and Arveson's consequent work \cite{Arv08}, Davidson and Kennedy proved the following theorem (as above, similar reformulations exist for the categories of unital operator spaces and operator algebras). 

\begin{theorem}
Every UCP map can be dilated to a maximal UCP map, and every pure UCP map can be dilated to a pure UCP map. 
\end{theorem}

Davidson and Kennedy proved that {\em pureness} guarantees that the $*$-representation, which is the unique UCP extension of the maximal dilation, is in fact irreducible. 
Moreover, they showed that pure UCP maps completely norm an operator space. 
Thus, by dilating sufficiently many pure UCP maps, and making use of the above theorems, they concluded that there exist sufficiently many boundary representations \cite{DK15}. 

\begin{example}\label{ex:discalg2}
Let us see what are the maximal dilations in the case of the disc algebra $A(\ol{\bD}) \subseteq C(\ol{\bD})$ (we switch back from the category of operator systems to the category of unital operator algebras). 
A representation $\pi$ of $C(\ol{\bD})$ is determined uniquely by a normal operator $N$ with spectrum in $\ol{\bD}$ by the relation $N = \pi(z)$. 
A UCC representation $\phi : A(\ol{\bD}) \to B(\cH)$ is determined uniquely by the image of the coordinate function $z$, which is a contraction $T = \phi(z) \in B(\cH)$. 
Conversely, by the $A(\ol{\bD})$ functional calculus (see Section \ref{subsec:glimpse}), every contraction $T \in B(\cH)$ gives rise to a UCC homomorphism of $A(\ol{\bD})$ into $B(\cH)$. 
In this context, a dilation of a UCC map $\phi$ into $B(\cH)$ is simply a representation $\rho : A(\ol{\bD}) \to B(\cK)$ such that 
\[
f(T) = \phi(f) = P_\cH \rho(f) \big|_\cH = P_\cH f(V)\big|_\cH, 
\] 
for all $f \in A(\ol{\bD})$, where $V = \rho(z) \in B(\cK)$. 
So $\rho$ is a dilation of $\phi$ if and only if $V = \rho(z)$ is a (power) dilation of $T = \phi(z)$ in the sense of Section \ref{sec:single}. 

With the above notation, it is not hard to see the following two equivalent statements: {\em (i) a dilation $\rho : A(\ol{\bD}) \to B(\cK)$ is maximal if and only if $V$ is a unitary}, and {\em (ii) a representation $\pi : C(\ol{\bD}) \to B(\cK)$ is such that $\rho = \pi\big|_{A(\ol{\bD})}$ has the unique extension property if and only if $V$ is a unitary}. 
The fact that every UCC representation of $A(\ol{\bD})$ has a maximal dilation, is equivalent to the fact that every contraction has a unitary dilation. 

What are the boundary representations of the disc algebra? 
The irreducible representations of $C(\ol{\bD})$ are just point evaluations $\delta_z$ for $z \in \ol{\bD}$. 
The boundary representations are those point evaluations $\delta_z$ whose restriction to $A(\ol{\bD})$ have a unique extension to a UCP map $C(\ol{\bD}) \to \bC$. 
But UCP maps into the scalars are just states, and states of $C(\ol{\bD})$ are given by probability measures. 
Hence, boundary representations are point evaluations $\delta_z$ such that $z$ has a unique representing measure, so that $z \in \bT$. 

The Shilov ideal can be obtained as the intersection of the kernels of the boundary representations, and so it is the ideal of functions vanishing on $\bT$. 
The C*-envelope is the quotient of $C(\ol{\bD})$ by this ideal, thus it is $C(\bT)$, as we noted before. 
\end{example}

Note that to find the boundary representations of the disc algebra we did not need to invoke the machinery of maximal dilations.  
In the commutative case, the existence of sufficiently many boundary representations is no mystery: all of them are obtained as extensions of evaluation at a boundary point. 
The machinery of maximal dilations allows us to find the boundary representations in the noncommutative case, where there are no function theoretic tools at our disposal.

\section{Dilations of completely positive semigroups}\label{sec:CPsemigroups} 

The dilation theory of semigroups of completely positive maps can be considered as a kind of ``quantization" of classical isometric dilation theory of contractions on a Hilbert space.  
The original motivation comes from mathematical physics \cite{DaviesBook,EL76}. 
The theory is also very interesting and appealing in itself, having connections and analogies (and also surprising differences) with classical dilation theory. 
Studying dilations of CP-semigroups has led to the discovery of results and some structures that are interesting in themselves. 

In this section, I will only briefly review some results in dilation theory of CP-semigroups from the last two decades, of the kind that I am interested in. 
There are formidable subtleties and technicalities that I will either ignore, or only gently hint at. 
For a comprehensive and up-to-date account, including many references (also to other kinds of dilations), see \cite{ShalitSkeideBig}. 
All of the facts that we state without proof or reference below have either a proof or a reference in \cite{ShalitSkeideBig}. 
The reader is also referred to the monographs \cite{ArvBook} and \cite{Par12} for different takes on quantum dynamics and quantum probability. 

\subsection{CP-semigroups, E-semigroups and dilations}\label{subsec:CPsemigroups} 
Let $\cS$ be a commutative monoid. 
By a {\bf{CP-semigroup}} we mean family $\Theta = \{\Theta_s\}_{s \in \cS}$ of contractive CP-maps on a unital C*-algebra $\cB$ such that
\begin{enumerate}
\item $\Theta_0 = {\bf id}_\cB$,
\item $\Theta_s \circ \Theta_t = \Theta_{s+t}$ for all $s, t \in \cS$.  
\end{enumerate}
If $\cS$ carries some topology, then we usually require $s \mapsto \Theta_s$ to be continuous in some sense. 
A CP-semigroup $\Theta$ is said to be a {\bf{Markov semigroup}} (or a {\bf{CP$_0$-semigroup}}) if every $\Theta_s$ is a unital map. 
A CP-semigroup $\Theta$ is called an {\bf{E-semigroup}} if every element $\Theta_s$ is a $*$-endomorphism. 
Finally, an {\bf{E$_0$-semigroup}} is a Markov semigroup which is also an E-semigroup. 

One-parameter semigroups of $*$-automorphisms model the time evolution in a closed (or a reversible) quantum mechanical system, and one-parameter Markov semigroups model the time evolution in an open (or irreversible) quantum mechanical system \cite{DaviesBook}. 

The prototypical example of a CP-semigroup is given by 
\be\label{eq:proto}
\Theta_s(b) = T_s b T_s^* \,\, , \,\, b \in \cB
\ee
where $T = \{T_s\}_{s \in \cS}$ is a semigroup of contractions in $\cB$. 
We call such a semigroup {\bf{elementary}}. 
Of course, not all CP-semigroups are elementary. 

For us, a {\bf{dilation}} of a CP-semigroup is a triplet $(\cA, \alpha, p)$, where $\cA$ is a C*-algebra, $p \in \cA$ is a projection such that $\cB = p\cA p$, and $\alpha = \{\alpha_s\}_{s\in\cS}$ is an E-semigroup on $\cA$, such that 
\bes
\Theta_s(b) = p \alpha_s(b) p 
\ees
for all $b \in \cB$ and $s \in \cS$. 
A {\bf{strong dilation}} is a dilation in which the stronger condition 
\bes
\Theta_s(pap) = p \alpha_s(a) p 
\ees
holds for all $a \in \cA$ and $s \in \cS$. 
It is a fact, not hard to show, that if $\Theta$ is a Markov semigroup then every dilation is strong. 
Examples show that this is not true for general CP-semigroups. 
Sometimes, to lighten the terminology a bit, we just say that $\alpha$ is a (strong) dilation. 

\begin{remark}
It is worth pausing to emphasize that the dilation defined above is entirely different from Stinespring's dilation: the Stinespring dilations of $\Theta_s$ and $\Theta_{s'}$ cannot be composed. 
The reader should also be aware that there are other notions of dilations, for example in which the ``small" algebra $\cB$ is embedded as a {\em unital} subalgebra of the ``large" algebra $\cA$ (see, e.g., \cite{Gae15} or \cite{vED19} and the references therein), or where additional restrictions are imposed (see, e.g. \cite{Kum85}, and the papers that cite it). 
\end{remark}

The most important is the one-parameter case, where $\cS = \bN$ or $\bR_+$. 
The following result was proved first by Bhat in the case $\cB = B(\cH)$ \cite{Bhat96} (slightly later SeLegue gave a different proof in the case $\cB = B(\cH)$ \cite{SeL97}), then it was proved by Bhat and Skeide for general unital C*-algebras \cite{BS00}, and then by Muhly and Solel in the case of unital semigroups on von Neumann algebras \cite{MS02} (slightly later this case was also proved by Arveson \cite[Chapter 8]{ArvBook}).

\begin{theorem}\label{thm:Bhat}
Every CP-semigroup $\Theta = \{\Theta_s\}_{s \in \cS}$ on $\cB$ over $\cS = \bN$ or $\cS = \bR_+$ has a strong dilation $(\cA, \alpha, p)$. 

Moreover, if $\Theta$ is unital then $\alpha$ can be chosen unital; if $\cB$ is a von Neumann algebra and $\Theta$ is normal, then $\cA$ can be taken to be a von Neumann algebra and $\alpha$ normal; and if further $\cS = \bR_+$ and $\Theta$ is point weak-$*$ continuous, in the sense that $t \mapsto \rho( \Theta_t(b))$ is continuous for all $\rho \in \cB_*$, then $\alpha$ can also be chosen to be point weak-$*$ continuous. 
\end{theorem}
\begin{proof}
Let us illustrate the proof for the case where $\cB = B(\cH)$, $\cS = \bN$, $\theta$ is a normal contractive CP map on $B(\cH)$, and $\Theta_n = \theta^n$ for all $n$. 
Then it is well known that $\theta$ must have the form $\theta(b) = \sum_i T_i b T_i^*$ for a row contraction $T = (T_i)$ (see, e.g., \cite[Theorem 1]{KrausBook}). 
By Theorem \ref{thm:rowiso_dil}, $T$ has a row isometric coextension $V = (V_i)$ on a Hilbert space $\cK$. 
Letting $\cA = B(\cK)$, $p = P_\cH$ and 
\[
\alpha(a) = \sum_i V_i a V_i^*
\] 
we obtain a strong dilation (as the reader will easily verify). 
\end{proof}

The above proof suggests that there might be strong connections between operator dilation theory and the dilation theory of CP-semigroups. 
This is true, but there are some subtleties. 
Consider an elementary CP-semigroup \eqref{eq:proto} acting on $\cB = B(\cH)$, where $T$ is a semigroup of contractions on $\cH$. 
If $V = \{V_s\}_{s \in \cS}$ is an isometric dilation of $T$, then $\alpha(a) = V_s a V_s^*$ is a dilation of $\Theta$. 
If $V$ is a coextension, then $\alpha$ is a strong dilation. 
Thus, we know that we can find (strong) dilations for elementary semigroups when the semigroup is such that isometric dilations (coextensions) exist for every contractive semigroup; for example, when $\cS = \bN$, $\bN^2$, $\bR_+$, $\bR_+^2$. 
Moreover, we expect that we won't always be able to dilate CP-semigroups over monoids for which isometric dilations don't always exist (for example $\cS = \bN^3$). 
This analogy based intuition is almost correct, and usually helpful. 

\begin{remark}
As in classical dilation theory, there is also a notion of {\em minimal dilation}. 
However, it turns out that there are several reasonable notions of minimality. 
In the setting of normal continuous semigroups on von Neumann algebras, the most natural notions of minimality turn out to be equivalent in the one-parameter case, but in the multi-parameter case they are not equivalent.  
See \cite[Chapter 8]{ArvBook} and \cite[Section 21]{ShalitSkeideBig} for more on this subject. 
\end{remark}

Theorem \ref{thm:Bhat} has the following interesting interpretation. 
A one-param\-eter CP-semigroup models the time evolution in an open quantum dynamical system, and a one-parameter automorphism semigroup models time evolution in a closed one. 
In many cases, E-semigroups can be extended to automorphism semigroups, and so Theorem \ref{thm:Bhat} can be interpreted as saying that every open quantum dynamical system can be embedded in a closed (reversible) one (this interpretation was the theoretical motivation for the first dilation theorems, see \cite{DaviesBook,EL76}). 

\subsection{Main approaches and results}\label{subsec:results} 

There are two general approaches by which strong dilations of CP-semigroups can be constructed.  

\subsubsection{The Muhly-Solel approach}
One approach, due to Muhly and Solel \cite{MS02}, seeks to represent a CP-semigroup in a form similar to \eqref{eq:proto}, and then to import ideas from classical dilation theory. 
To give little more detail, if $\Theta = \{\Theta_s\}_{s\in\cS}$ is a CP-semigroup on a von Neumann algebra $\cB\subseteq B(\cH)$, then one tries to find a product system $E^{\odot} = \{E_s\}_{s\in\cS}$ of $\cB'$-correspondences over $\cS$ (see Section \ref{subsec:correspondences}) and a completely contractive covariant representation $T = \{T_s\}_{s\in\cS}$ such that 
\be\label{eq:CP_rep}
\Theta_s(b) = \widetilde{T}_s ({\bf id}_{E_s} \odot b) \widetilde{T}_s^*
\ee
for every $s \in \cS$ and $b \in \cB$. 
Here, $\widetilde{T}_s: E_s \odot \cH \to \cH$ is given by $\widetilde{T}_s ( x \odot h) = T_s(x) h$. 
This form is reminiscent of \eqref{eq:proto}, and it is begging to try to dilate $\Theta$ by constructing an isometric dilation $V$ for $T$, and then defining 
\[
\alpha_s(a) = \widetilde{V}_s ({\bf id}_{E_s} \odot a) \widetilde{V}_s^*
\] 
for $a \in \cA := V_0(\cB')'$. 
This is a direct generalization of the apporach to dilating elementary semigroups discussed in the paragraph following the proof of Theorem \ref{thm:Bhat}, and in fact it also generalizes the proof we gave for that theorem. 
This approach was used successfully to construct and analyze dilations in the discrete and continuous one-parameter cases (Muhly and Solel, \cite{MS02,MS07}), in the discrete two-parameter case (Solel, \cite{Sol06}; this case was solved earlier by Bhat for $\cB = B(\cH)$ \cite{Bhat98}), and in the ``strongly commuting" two-parameter case (Shalit, \cite{ShalitE0Dil,ShalitWhat,ShalitEDil}). 

However, it turns out that finding a product system and representation giving back $\Theta$ as in \eqref{eq:CP_rep} is not always possible, and one needs a new notion to proceed. 
A {\bf{subproduct system}} is a family $\cE^{\varogreaterthan} = \{\cE_s\}_{s \in\cS}$ of C*-correspondences such that, roughly, $\cE_{s+t} \subseteq \cE_s \odot \cE_t$ (up to certain identification that iterate associatively). 
Following earlier works \cite{ArvBook,MS02}, it was shown that for every CP-semigroup there is a {\em subproduct system} and a representation, called the {\bf{Arveson-Stinespring subproduct system and representation}}, satisfying \eqref{eq:CP_rep} (Shalit and Solel, \cite{ShalitSolel}). 
Subproduct systems have appeared implicitly in the theory in several places, and in \cite{ShalitSolel} they were finally formally introduced (at the same time, subproduct systems of Hilbert spaces were introduced in Bhat and Mukherjee's paper \cite{BM10}). 

The approach to dilations introduced in \cite{ShalitSolel} consists of two parts: first, embed the Arveson-Stinespring subproduct system associated with a CP-semigroup into a product system, and then dilate the representation to an isometric dilation.  
This approach was used to find necessary and sufficient conditions for the existence of dilations. 
In particular, it was used to prove that a Markov semigroup has a (certain kind of) minimal dilation if and only if the Arveson-Stinespring subproduct system can be embedded into a product system. 
Moreover, the framework was used to show that there exist CP-semigroups over $\bN^3$ that have no {\em minimal} strong dilations, as was suggested from experience with classical dilation theory. 
Vernik later used these methods to prove an analogue of Opela's theorem (see \ref{subsec:semi}) for completely positive maps commuting according to a graph \cite{Ver16}. 

The reader is referred to \cite{ShalitSolel} for more details. 
The main drawback of that approach is that it works only for CP-semigroups of normal maps on von Neumann algebras. 

\subsubsection{The Bhat-Skeide approach}
The second main approach to dilations of CP-semigroups is due to Bhat and Skeide \cite{BS00}. 
It has several advantages, one of which is that it works for semigroups on unital C*-algebras (rather than von Neumann algebras). 
The Bhat-Skeide approach is based on a fundamental and useful representation theorem for CP maps called Paschke's {\em GNS representation} \cite{Pas73}, which I will now describe. 

For every CP map $\phi : \cA \to \cB$ between two unital C*-algebras, there exists a C*-correspondence $E$ from $\cA$ to $\cB$ (see Section \ref{subsec:correspondences}) and a vector $\xi \in E$ such that $\phi(a) = \langle \xi , a \xi \rangle$ for all $a \in \cA$. 
The existence of such a representation follows from a construction: one defines $E$ to be the completion of the algebraic tensor product $\cA \otimes \cB$ with respect to the $\cB$-valued inner product 
\[
\langle a \otimes b, a' \otimes b' \rangle = b^* \phi(a^*a') b', 
\]
equipped with the natural left and right actions. 
Letting $\xi = 1_\cA \otimes 1_\cB$, it is immediate that $\phi(a) = \langle \xi , a \xi \rangle$ for all $a \in \cA$. 
Moreover, $\xi$ is {\bf cyclic}, in the sense that it generates $E$ as a C*-correspondence. 
The pair $(E,  \xi)$ is referred to as the {\bf GNS representation} of $\phi$. 
The GNS representation is unique in the sense that whenever $F$ is a C*-correspondence from $\cA$ to $\cB$ and $\eta \in F$ is a cyclic vector such that $\phi(a) = \langle \eta , a \eta \rangle$ for all $a \in \cA$, then there is an isomorphism of C*-correspondences from $E$ onto $F$ that maps $\xi$ to $\eta$. 

In the Bhat-Skeide approach to dilations, the idea is to find a product system $F^\odot = \{F_s\}_{s \in \cS}$ of $\cB$-correspondences and a unit, i.e., a family $\xi^\odot = \{\xi_s\in F_s\}_{s \in \cS}$ satisfying $\xi_{s+t} = \xi_s \odot \xi_t$, such that $\Theta$ is recovered as 
\be\label{eq:GNS}
\Theta_s(b) = \langle \xi_s , b \xi_s \rangle
\ee
for all $s \in \cS$ and all $b \in \cB$. 
If $\Theta$ is a Markov semigroup, the dilation is obtained via a direct limit construction. 
For non Markov semigroups, a dilation can be obtained via a unitalization procedure. 
In \cite{BS00}, dilations were constructed this way in the continuous and discrete one-parameter cases. 
This strategy bypasses product system representations, but, interestingly, it can also be used to prove the existence of an isometric dilation for any completely contractive covariant representation of a one-parameter product system \cite{Ske08}. 

Again, it turns out that constructing such a product system is not always possible. 
However, if one lets $(\cF_s, \xi_s)$ be the GNS representation of the CP map $\Theta_s$, then it is not hard to see that $\cF^\varogreaterthan = \{\cF_s\}_{s\in\cS}$ is a subproduct system (called the {\bf{GNS subproduct system}}) and $\xi^\odot = \{\xi_s\}$ is a unit. 

The above observation was used by Shalit and Skeide to study the existence of dilations of CP-semigroups in a very general setting \cite{ShalitSkeideBig}. 
If one can embed the GNS subproduct system into a product system, then one has \eqref{eq:GNS}, and can invoke the Bhat-Skeide approach to obtain a dilation. 
The paper \cite{ShalitSkeideBig} develops this framework to give a unified treatment of dilation theory of CP-semigroups over a large class of monoids $\cS$, including noncommutative ones. 
One of the main results in \cite{ShalitSkeideBig}, is the following, which generalizes a result obtained earlier in \cite{ShalitSolel}. 

\begin{theorem}
A Markov semigroup over an Ore monoid admits a full strict dilation if and only if its GNS subproduct
system embeds into a product system. 
\end{theorem}

This theorem essentially enables to recover almost all the other known dilation theorems and counter examples. 
It is used in \cite{ShalitSkeideBig} to show that every Markov semigroup over $\cS = \bN^2$ on a von Neumann algebra has a unital dilation, and also that for certain multi-parameter semigroups (the so called {\em quantized convolution semigroups}) there is always a dilation. 
The theorem is also used in the converse direction, to construct a large class of examples that have no dilation whatsoever. 

In the setting of normal semigroups on von Neumann algebras, the Bhat-Skeide and Muhly-Solel approaches to dilations are connected to each other by a functor called {\em commutant}; see \cite[Appendix A(iv)]{ShalitSkeideBig} for details. 

\subsubsection{New phenomena}
We noted above some similarities between the theory of isometric dilations of contractions, and the dilation theory of CP-semigroups. 
In particular, in both theories, there always exists a dilation when the semigroup is parameterized by $\cS = \bN$, $\bN^2$, $\bR_+$, and the results support the possibility that this is true for $\cS = \bR_+^2$ as well. 
Moreover, in both settings, there exist semigroups over $\cS = \bN^3$ for which there is no dilation. 

But there are also some surprises. 
By Corollary \ref{cor:suff_reg}, if a tuple of commuting contractions $T = (T_1, \ldots, T_d)$ satisfies $\sum_{i=1}^d \|T_i\|^2 \leq 1$, then $T$ has a regular unitary dilation. 
Therefore, one might think that if a commuting $d$-tuple of CP maps $\Theta_1, \ldots, \Theta_d$ are such that $\sum_{i=1}^d \|\Theta_i\|$ is sufficiently small, then this tuple has a dilation. 
This is false, at least in a certain sense (that is, if one requires a strong and {\em minimal} dilation); see \cite[Section 5.3]{ShalitSolel}. 

Moreover, by Corollary \ref{cor:suff_reg}, a tuple of commuting isometries always has a unitary dilation, and it follows that every tuple of commuting coisometries has an isometric (in fact, unitary) coextension. 
In \eqref{eq:proto} coisometries correspond to unital maps. 
Hence, one might expect that commuting unital CP maps always have dilations. 
Again, this is false \cite[Section 18]{ShalitSkeideBig} (see also \cite{ShSk11}). 
The reason for the failure of these expectations is that not every subproduct system over $\bN^d$ (when $d \geq 3$) can be embedded into a product system (in fact, there are subproduct systems that cannot even be embedded in a {\em superproduct system}). 
However, the product system of an elementary CP-semigroup is a trivial product system, so the obstruction to embeddability does not arise in that case.

\subsection{Applications of dilations of CP-semigroups}\label{subsec:appl_CP}

Besides the interesting interpretation of dilations given at the end of Section \ref{subsec:CPsemigroups}, dilations have some deep applications in noncommutative dynamics.  
In this section, we will use the term CP-semigroup to mean a one-parameter semigroup $\Theta = \{\Theta_s\}_{s \in \cS}$ (where $\cS = \bN$ or $\cS = \bR_+$) of normal CP maps acting on a von Neumann algebra $\cB$. 
In case of continuous time (i.e., $\cS = \bR_+$), we will also assume that $\Theta$ is {\em point weak-$*$ continuous}, in the sense that $t \mapsto \rho(\Theta_t(b))$ is continuous for every $\rho \in \cB_*$. 
We will use the same convention for E- or E$_0$-semigroups. 

\subsubsection{The noncommutative Poisson boundary}
Let $\Theta$ be a normal UCP map on a von Neumann algebra $\cB$. 
Then one can show that the fixed point set $\{b \in \cB : \Theta(b) = b\}$ is an operator system, and moreover that it is the image of a completely positive projection $E:  \cB \to \cB$. 
Hence, the {\em  Choi-Effros product} $x \circ y = E(xy)$ turns the fixed point set into a von Neumann algebra $H^\infty(\cB,\Theta)$, called the {\bf{noncommutative Poisson boundary}} of $\Theta$. 
The projection $E$ and the concrete structure on $H^\infty(\cB,\Theta)$ are hard to get a grip with. 

Arveson observed that if $(\cA,\alpha,p)$ is the minimal dilation of $\Theta$, and if $\cA^\alpha$ is the fixed point algebra of $\alpha$, then the compression $a \mapsto pap$ is a unital, completely positive order isomorphism between $\cA^\alpha$ and $H^\infty(\cB,\Theta)$. 
Hence $\cA^\alpha$ is a concrete realization of the noncommutative Poisson boundary. 
See the survey \cite{Izu12} for details (it has been observed that this result holds true also for dilations of abelian CP-semigroups \cite{Pru12}). 

\subsubsection{Continuity of CP-semigroups} 
Recall that one-parameter CP-semigroups on a von Neumann algebra $\cB \subseteq B(\cH)$ are assumed to be point weak-$*$ continuous. 
Since CP-semigroups are bounded, this condition is equivalent to {\em point weak-operator} continuity, i.e., that $t \mapsto \langle \Theta_t(b) g, h \rangle$ is continuous for all $b \in \cB$ and all $g,h \in \cH$. 
Another natural kind of continuity to consider is {\em point strong-operator continuity}, which means that $t \mapsto \Theta_t(b)h$ is continuous in norm for all $b \in \cB$ and $h \in \cH$. 
For brevity, below we shall say a semigroup is {\bf weakly continuous} if it is point weak-operator continuous, and {\bf strongly continuous} if it is point strong-operator continuous. 

Strong continuity is in some ways easier to work with and hence it is desirable, but it is natural to use the weak-$*$ topology, because it is independent of the representation of the von Neumann algebra. 
Happily, it turns out that weak (and hence point weak-$*$) continuity implies strong continuity. 

One possible approach to prove the above statement is via dilation theory. 
First, one notices that the implication is easy for E-semigroups. 
Indeed, if $\alpha$ is an E-semigroup on $\cA \subseteq B(\cK)$, then 
\[
\|\alpha_t(a) k - \alpha_s(a)k\|^2 = \langle \alpha_t(a^*a)k, k\rangle + \langle \alpha_s(a^*a)k, k\rangle - 2 \re \langle \alpha_t(a) k, \alpha_s(a) k \rangle. 
\]
Assuming that $s$ tends to $t$, the expression on the right hand side tends to zero if $\alpha$ is weakly continuous. 
Now, if $\Theta$ is a weakly continuous CP-semigroup, then its dilation $\alpha$ given by Theorem \ref{thm:Bhat} is also weakly continuous. 
By the above argument, $\alpha$ is strongly continuous, and this continuity is obviously inherited by $\Theta_t(\cdot) = p\alpha_t (\cdot)p$. 
Hence, by dilation theory, weak continuity implies strong continuity. 

The above argument is a half cheat, because, for a long time, the known proofs that started from assuming weak continuity and ended with a weakly continuous dilation, actually assumed implicitly, somewhere along the way, that CP-semigroups are strongly continuous \cite{ArvBook,BS00,MS02}. 
This gap was pointed out and fixed by Markiewicz and Shalit \cite{MarkiewiczShalit}, who proved directly that a weakly continuous CP-semigroup is strongly continuous. 
Later, Skeide proved that the minimal dilation of a weakly continuous semigroup of CP maps is strongly continuous, independently of \cite{MarkiewiczShalit}, thereby recovering the result ``weakly continuous $\Rightarrow$ strongly continuous" with a proof that truly goes through the construction of a dilation; see \cite[Appendix A.2]{Ske16}. 

\subsubsection{Existence of E$_0$-semigroups} 
As we have seen above, dilations can be used to study CP-semigroups. 
We will now see an example, where dilations are used in the theory of E$_0$-semigroups. 

The fundamental classification theory E$_0$-semigroups on $\cB = B(\cH)$ was developed Arveson, Powers, and others about two decades ago; see the monograph \cite{ArvBook} for the theory, and in particular for the results stated below (the classification theory of E$_0$-semigroups on arbitrary C* or von Neumann algebras is due to Skeide; see \cite{Ske16}). 
For such E$_0$-semigroups there exists a crude grouping into {\em type I, type II, and type III} semigroups. 
However, it is not at all obvious that there exist any E$_0$-semigroups of every type. 
Given a semigroup of isometries on a Hilbert space $H$, one may use second quantization to construct E$_0$-semigroups on the symmetric and anti-symmetric Fock spaces over $H$, called the CCR and CAR flows, respectively.  
CAR and CCR flows are classified in term of their {\em index}. 
These E$_0$-semigroups are of type I, and, conversely, every type I E$_0$-semigroup is cocycle conjugate to a CCR flow, which is, in turn, conjugate to a CAR flow.  

It is much more difficult to construct an E$_0$-semigroup that is not type I. 
How does one construct a non-trivial E$_0$-semigroup? 
Theorem \ref{thm:Bhat} provides a possible way: construct a Markov semigroup, and then take its minimal dilation. 
This procedure has been applied successfully to provide examples of non type I E$_0$-semigroups, even with prescribed {\em index}; see \cite{ArvBook,Izu12}.

\part{Recent results in the dilation theory of noncommuting operators}

\section{Matrix convexity and dilations} 

In recent years, dilation theory has found a new role in operator theory, through the framework of matrix convexity. 
In this section I will quickly introduce matrix convex sets in general, special examples, minimal and maximal matrix convex over a given convex set, and the connection to dilation theory. 
Then I will survey the connection to the UCP interpolation problem, some dilation results, and finally an application to spectrahedral inclusion problems.

\subsection{Matrix convex sets} 

Fix $d \in \bN$. 
The ``noncommutative universe" $\bM^d$ is the set of all $d$-tuples of $n \times n$ matrices, of all sizes $n$, that is, 
\[
\bM^d = \bigsqcup_{n=1}^\infty M_n^d .
\]
Sometimes it is useful to restrict attention to the subset $\bM_{\text{sa}}^d$, which consists of all tuples of selfadjoint matrices. 
We will refer to a subset $\cS \subseteq \bM^d$ as a {\bf{noncommutative (nc) set}}, and we will denote by $\cS_n$ or $\cS(n)$ the {\bf{$n$th level of $\cS$}}, by which we mean $\cS_n = \cS(n) := \cS \cap M_n^d$. 
Let us endow tuples with the row norm $\|A\| := \|\sum_i A_i A_i^*\|$; this induces a metric on $B(\cH)^d$ for every $d$, and in particular on $M_n^d$ for every $n$. 
We will say that a nc set $\cS$ is {\bf{closed}} if $\cS_n$ is closed in $M_n^d$ for all $n$. 
We will say that $\cS$ is {\bf{bounded}} if there exists some $C>0$ such that $\|X\|\leq C$ for all $X \in \cS$. 

For a tuple $X = (X_1, \ldots, X_n) \in M_n^d$ and a linear map $\phi : M_n \to M_k$, we write $\phi(X) = (\phi(X_1), \ldots, \phi(X_d)) \in M_k^d$. 
In particular, if $A$ and $B$ are $n \times k$ matrices, then we write $A^* X B = (A^* X_1 B, \ldots, A^*X_d B)$. 
Another operation that we can preform on tuples is the {\bf{direct sum}}, that is, if $X \in M_m^d$ and $Y \in M_n^d$, then we let $X \oplus Y = (X_1 \oplus Y_1, \ldots, X_d \oplus Y_d) \in M_{m+n}^d$. 

A {\bf{matrix convex set}} $\cS$ is a nc set $\cS = \sqcup_{n=1}^\infty \cS_n$ which is invariant under direct sums and the application of UCP maps: 
\[
X \in \cS_m , Y \in \cS_n \Longrightarrow X \oplus Y \in \cS_{m+n}
\]
and 
\[
X \in \cS_n \textup{   and   } \phi \in \UCP(M_n,M_k) \Longrightarrow \phi(X) \in \cS_k .
\]
It is not hard to check that a nc set $\cS \subseteq \bM^d$ is matrix convex if and only if it is closed under {\bf{matrix convex combinations}} in the following sense: whenever $X^{(j)} \in \cS_{n_j}$ and $V_j \in M_{n_j,n}$ for $j = 1, \ldots, k$ are such that $\sum_{j=1}^k V_j^* V_j =I_n$, then $\sum V_j^* X^{(j)} V_j \in \cS_n$. 

\begin{remark}
The above notion of matrix convexity is due to Effros and Winkler \cite{EW97}. 
Other variants appeared before and after. 
A very general take on matrix convexity that I will not discuss here has recently been initiated by Davidson and Kennedy \cite{DK+}. 
I will follow a more pedestrian point of view, in the spirit of \cite{DDSSa} (note: the arxiv version \cite{DDSSa} is a corrected version of the published version \cite{DDSS17}. The latter contains several incorrect statements in Section 6, which result from a missing hypothesis; the problem and its solution are explained in \cite{DDSSa}). 
We refer to the first four chapters of \cite{DDSSa} for explanations and/or references to the some of the facts that will be mentioned below without proof. 
The papers \cite{EHKM18,Passer,PS19} make a connection between the geometry of matrix convex sets, in particular various kinds of extreme points, and dilation theory. 
For a comprehensive and up-to-date account of matrix convex sets the reader can consult \cite{Kriel}. 
\end{remark}

\begin{example}
Let $A \in B(\cH)^d$. 
The {\bf{matrix range}} of $A$ is the nc set $\cW(A) = \sqcup_{n=1}^\infty \cW_n(A)$ given by 
\[
\cW_n(A) = \left\{\phi(A) : \phi : B(\cH) \to M_n \textup{ is UCP} \right\}. 
\]
The matrix range is a closed and bounded matrix convex set. 
Conversely, every closed and bounded matrix convex set is the matrix range of some operator tuple. 
\end{example}
If $d = 1$, then the first level $\cW_1(A)$ of the matrix range of an operator $A$ coincides with the closure of the {\bf{numerical range}} $W(A)$
\[
W(A) = \left\{\langle Ah, h \rangle : \|h\| = 1\right\}. 
\]
We note, however, that for $d \geq 2$, the first level $\cW_1(A)$ does not, in general, coincide with the closure of what is sometimes referred to as the {\bf{joint numerical range}} of a tuple \cite{LP00}. 

Matrix ranges of single operators were introduced by Arveson \cite{Arv72}, and have been picked up again rather recently.
The matrix range of an operator tuple $A$ is a complete invariant of the operator system generated by $A$, and --- as we shall see below --- it is useful when considering interpolation problems for UCP maps.
Moreover, in the case of a {\em fully compressed} tuple $A$ of compact operators or normal operators, the matrix range determines $A$ up to unitary equivalence \cite{PS19}. 
The importance of matrix ranges has led to the investigation of random matrix ranges, see \cite{GS+}.

\begin{example}
Let $A \in B(\cH)^d$. 
The {\bf{free spectrahedron}} determined by $A$ is the nc set $\cD_A = \sqcup_{n=1}^\infty \cD_A$ given by 
\[
\cD_A(n) = \left\{X \in M_n^d : \re \sum_{j=1}^d X _j\otimes A_j \leq I \right\}. 
\]
A free spectrahedron is always a closed matrix convex set, that contains the origin in its interior. 
Conversely, every closed matrix convex set with $0$ in its interior is a free spectrahedron. 
In some contexts it is more natural to work with just selfadjoint matrices. 
For $A \in B(\cH)^d_{\text{sa}}$ one defines 
\[
\cD^{\text{sa}}_A = \left\{X \in \bM^d_{\text{sa}} : \sum_{j=1}^d X _j\otimes A_j \leq I \right\}. 
\]
\end{example}

The first level $\cD_A(1)$ is called a {\bf{spectrahedron}}. 
Most authors use the word {\em spectrahedron} to describe only sets of the form $\cD_A(1)$ where $A$ is a tuple of {\em matrices}; and likewise for the term {\em free spectrahedron}. 
This distinction is important for applications of the theory, since spectrahedra determined by tuples of matrices form a class of reasonably tractable convex sets that arise in applications, and not every convex set with $0$ in its interior can be represented as $\cD_A(1)$ for a tuple $A$ acting on a finite dimensional space. 

For a matrix convex set $\cS \subseteq \bM^d$ we define its {\bf{polar dual}} to be 
\[
\cS^\circ = \left\{ X \in \bM^d : \re \left(\sum X_j \otimes A_j\right) \leq I \,\,\textup{ for all } A \in \cS\right\}. 
\]
If $\cS \subseteq \bM_{\text{sa}}^d$, then it is more convenient to use the following variant
\[
\cS^\bullet = \left\{ X \in \bM^d_{\text{sa}} : \sum X_j \otimes A_j \leq I \,\,\textup{ for all } A \in \cS\right\}. 
\]
By the Effros-Winkler Hahn-Banach type separation theorem \cite{EW97}, $\cS^{\circ \circ} = \cS$ whenever $\cS$ is a matrix convex set containing $0$ (if $ 0 \notin \cS$, then $\cS^{\circ \circ}$ is equal to the {\em matrix convex hull} of $\cS$ and $0$). 
It is not hard to see that $\cD_A = \cW(A)^\circ$, and that when $0 \in \cW(A)$, we also have $\cW(A) = \cD_A^\circ$. 

\begin{example}\label{expl:cones}
Another natural and important way in which matrix convex sets arise, is as positivity cones in operator systems. 
In \cite{FNT17} it was observed that a finite dimensional abstract operator system $\cM$ (see \cite[Chapter 13]{PauBook}) generated by $d$ linearly independent elements $A_1, \ldots, A_d \in \cM$, corresponds to a matrix convex set $\cC \subseteq \bM^d_{\text{sa}}$ where every $\cC_n$ is the cone in $(M_n^d)_{\text{sa}}$ consisting of the matrix tuples $X = (X_1, \ldots, X_d)$ such that $\sum X_j \otimes A_j$ is positive in $M_n(\cM)$. 
Such matrix convex sets can be described by a slight modification of the notion of free spectrahedron: 
\[
\cC = \left\{ X \in \bM^d_{\text{sa}} : \sum X_j \otimes A_j \geq 0\right\}. 
\]
\end{example}

\subsection{The UCP interpolation problem} 

Suppose we are given two $d$-tuples of operators $A = (A_1, \ldots, A_d) \in B(\cH)^d$ and $B = (B_1, \ldots, B_d) \in B(\cK)^d$. 
A very natural question to ask is whether there exists a completely positive map $\phi : B(\cH) \to B(\cK)$ such that $\phi(A_i) = B_i$. 
This is the {\em CP interpolation problem}. 
In the realm of operator algebras it is sometimes more useful to ask about the existence of a UCP map that interpolates between the operators, and in quantum information theory it makes sense to ask whether there exists a completely positive trace preserving (CPTP) interpolating map. 
A model result is the following. 

\begin{theorem}\label{thm:UCPinterpolation}
Let $A \in \cB(H)^d$ and $B\in \cB(K)^d$ be $d$-tuples of operators.
\begin{enumerate}
\item There exists a UCP map $\phi : B(\cH) \to B(\cK)$ such that $\phi(A_i) = B_i$ for all $i=1, \ldots, d$ if and only if $\cW(B) \subseteq \cW(A)$.
\item There exists a unital completely isometric map $\phi : B(\cH) \to B(\cK)$ such that $\phi(A_i) = B_i$ for all $i=1, \ldots, d$ if and only if $\cW(B) = \cW(A)$.
\end{enumerate}
\end{theorem}

This result was obtained by Davidson, Dor-On, Shalit and Solel in \cite[Theorem 5.1]{DDSSa}. 
An earlier result was obtained by Helton, Klep and McCullough in the case where $\cH$ and $\cK$ are finite dimensional, and the condition $\cW(B) \subseteq \cW(A)$ is replaced by the {\em dual} condition $\cD_A \subseteq \cD_B$, under the blanket assumption that $\cD_A$ is bounded \cite{HKM13} (see a somewhat different approach in \cite{AG15a}). 
Later, Zalar showed that the condition $\cD_A \subseteq \cD_B$ is equivalent to the existence of an interpolating UCP map without the assumption that $\cD_A$ is bounded, and also in the case of operators on an infinite dimensional space \cite{Zalar}. 
Variants of the above theorem were of interest to mathematical physicists for some time, see the references in the above papers. 

From Theorem \ref{thm:UCPinterpolation} one can deduce also necessary and sufficient conditions for the existence of contractive CP (CCP) or completely contractive (CC) maps sending one family of operators to another, as well as approximate versions (see \cite[Section 5]{DDSSa}). 
The theorem also leads to more effective conditions under additional assumptions, for example when dealing with normal tuples (recall, that $A = (A_1, \ldots, A_d)$ is said to be normal if $A_i$ is normal and $A_i A_j = A_j A_i$ for all $i,j$). 

\begin{corollary}\label{cor:normal_interpolation}
Let $A \in \cB(H)^d$ and $B\in \cB(K)^d$ be two normal $d$-tuples of operators.
Then there exists a UCP map $\phi : B(\cH) \to B(\cK)$ such that 
\[
\phi(A_i) = B_i \,\, , \,\, \textup{ for all } \, i=1, \ldots, d, 
\]
if and only if 
\[
\sigma(B) \subseteq \operatorname{conv} \sigma(A) .
\]
\end{corollary}
This result was first obtained by Li and Poon \cite{LP11}, in the special case where $A$ and $B$ each consist of commuting selfadjoint matrices. 
It was later recovered in \cite{DDSSa}, in the above generality, as a consequence of Theorem \ref{cor:normal_interpolation} together with the fact that for a normal tuple $N$, the matrix range $\cW(N)$ is the {\em minimal matrix convex set that contains the joint spectrum $\sigma(N)$ in its first level} (see \cite[Corollary 4.4]{DDSSa}). 
The next section is dedicated to explaining what are the minimal and maximal matrix convex sets over a convex set, and how these notions are related to dilation theory.

\subsection{Minimal and maximal matrix convex sets}

Every level $\cS_n$ of a matrix convex set $\cS$ is a convex subset of $M_n^d$. 
In particular, the first level $\cS_1$ is a convex subset of $\bC^d$. 
Conversely, given a convex set $K \subseteq \bC^d$ (or $K \subseteq \bR^d$), we may ask whether there exists a matrix convex set $\cS \subseteq \bM^d$ (or $\cS \subseteq \bM^d_{\text{sa}}$) such that $\cS_1 = K$. 
The next question to ask is, to what extent does the first level $\cS_1 = K$ determine the matrix convex set $\cS$? 

In order to approach the above questions, and also as part of a general effort to understand inclusions between matrix convex sets (motivated by results as Theorem \ref{thm:UCPinterpolation}), notions of minimal and maximal matrix convex sets have been introduced by various authors \cite{DDSS17,FNT17,HKM16}. 
These are very closely related (via Example \ref{expl:cones}) to the notion of minimal and maximal operator systems that was introduced earlier \cite{PTT11}. 

For brevity, we shall work in the selfadjoint setting. 
Let $K \subseteq \bR^d$ be a convex set. 
By the Hahn-Banach theorem, $K$ can be expressed as the intersection of a family of half spaces: 
\[
K = \{ x \in \bR^d : f_i(x) \leq c_i \,\,\textup{ for all } i \in \cI\}
\]
where $\{f_i\}_{i \in \cI}$ is a family of linear functionals and $\{c_i\}_{i \in \cI}$ is a family of scalars. 
Writing $f_i(x) = \sum_j a^i_j x_j$, we define 
\[
\Wmax{n}(K) = \{X \in (M_n^d)_{\text{sa}} : \sum_j a^i_j X_j \leq c_i I_n \,\,\textup{ for all } i \in \cI\}
\]
and $\Wmax{}(K) = \sqcup_{n=1}^\infty \Wmax{n}(K)$. 
In other words, $\Wmax{}(K)$ is the nc set determined by the linear inequalities that determine $K$. 
It is clear that $\Wmax{}(K)$ is matrix convex, and a moment's thought reveals that it contains every matrix convex set that has $K$ as its first level. 

That settles the question, of whether or not there exists a matrix convex set with first level equal to $K$. 
It follows, that there has to exist a minimal matrix convex set that has $K$ as its first level --- simply intersect over all such matrix convex sets. 
There is a useful description of this minimal matrix convex set. 
We define 
\be\label{eq:Wmin}
\Wmin{}(K) = \left\{X \in \bM^d_{\text{sa}} : \exists \textup{ normal } T \textup { with } \sigma(T) \subseteq K \textup{ s.t. } X \prec T \right\}.
\ee
$\Wmin{}(K)$ is clearly invariant under direct sums. 
To see that it is invariant also under the application of UCP maps, one may use Stinespring's theorem as follows. 
If $X \prec T$, $T$ is normal, and $\phi$ is UCP, then the map $T \mapsto X \mapsto \phi(X)$ is UCP. 
By Stinespring's theorem there is a $*$-representation $\pi$ such that $\phi(X) \prec \pi(T)$, and $\pi(T)$ is a normal tuple with $\sigma(\pi(T)) \subseteq \sigma(T)$ (alternatively, one may use the dilation guaranteed by Theorem \ref{thm:Bhat}). 

We see that the set defined in \eqref{eq:Wmin} is matrix convex. 
On the other hand, any matrix convex set containing $K$ in the first level must contain all unitary conjugates of tuples formed from direct sums of points in $K$, as well as their compressions, therefore the minimal matrix convex set over $K$ contains all $X$ that have a normal dilation $T$ {\em acting on a finite dimensional space} such that $\sigma(T) \subseteq K$. 
But for $X \in \bM^d_{\text{sa}}$, the existence of a normal dilation $X \prec T$ with $\sigma(T) \subseteq K$ implies the existence of a normal dilation {\em acting on a finite dimensional space} (see \cite[Theorem 7.1]{DDSS17}), thus the nc set $\Wmin{}(K)$ that we defined above is indeed the minimal matrix convex set over $K$. 

\begin{example}
Let $\ol{\bD}$ be the closed unit disc in $\bC$. 
Let us compute $\Wmin{}(\ol{\bD})$ and $\Wmax{}(\ol{\bD})$. 
We can consider $\ol{\bD}$ as a subset of $\bR^2$, and pass to the selfadjoint setting (and back) by identifying $T = \re T + i \im T \in \bM^1$ with the selfadjoint tuple $(\re T, \im R) \in \bM^2_{\text{sa}}$. 
The minimal matrix convex set is just 
\[
\Wmin{}(\ol{\bD}) = \{X \in \bM^1 : \|X\| \leq 1 \}, 
\]
because by Theorem \ref{eq:power_dil}, every contraction has a unitary dilation. 
Since the set of real linear inequalities determining the disc is 
\[
\ol{\bD} = \{z \in \bC : \re \left(e^{i\theta} z\right) \leq 1 \,\,\textup{ for all } \theta \in \bR\},
\]
it follows that
\[
\Wmax{}(\ol{\bD}) = \{X : \re \left( e^{i\theta} X\right) \leq I \,\,\textup{ for all } \theta \in \bR\}, 
\]
which equals the set of all matrices with numerical range contained in the disc. 
\end{example}

Given a convex set $K \subseteq \bC^d$, Passer, Shalit and Solel introduced a constant $\theta(K)$ that quantifies the difference between the minimal and maximal matrix convex sets over $K$ \cite[Section 3]{PSS18}.  
For two convex sets $K, L$, we define 
\[
\theta(K, L) = \inf\{ C : \Wmax{}(K) \subseteq C\Wmin{}(L) \}, 
\]
and $\theta(K) = \theta(K,K)$. 
Note that $C \Wmin{}(L) = \Wmin{}(C L)$. 

\begin{remark}
In the theory of operator spaces, there are the notions of minimal and maximal operator spaces over a normed space $V$, and there is a constant $\alpha(V)$ that quantifies the difference between the minimal and maximal operator space structures \cite{Pau92} (see also \cite[Chapter 14]{PauBook} and \cite[Chapter 3]{PisBook}). 
These notions are analogous to the above notions of minimal and maximal matrix convex sets, but one should not confuse them. 
\end{remark}

By the characterization of the minimal and maximal matrix convex sets, the inclusion $\Wmax{}(K) \subseteq \Wmin{}(L)$ is a very general kind of dilation result: it means that {\em every $d$-tuple $X$ satisfying the linear inequalities defining $K$, has a normal dilation $X \prec N$ such that $\sigma(N) \subseteq L$.} 
Let us now review a few results obtained regarding this dilation problem. 

\begin{theorem}[Theorem 6.9, \cite{PSS18}]
For $p \in [1,\infty]$, let $\ol{\bB}_{p,d}$ denote the unit ball in $\bR^d$ with respect to the $\ell^p$ norm, and let $\ol{\bB}_{p,d}(\bC)$ denote the unit ball in $\bC^d$ with respect to the $\ell^p$ norm. 
Then 
\[
\theta(\ol{\bB}_{p,d}) = d^{1-|1/2-1/p|}
\] 
and 
\[
\theta(\ol{\bB}_{p,d}(\bC)) = 2d^{1-|1/2-1/p|}.
\] 
\end{theorem}
See \cite{PSS18} for many other (sharp) inclusions $\Wmax{}(K) \subseteq \Wmin{}(L)$. 
Interestingly, the fact that $\theta(\ol{\bB}_{1,d}) = \sqrt{d}$ has implications in quantum information theory --- it allows to find a quantitative measure of how much {\em noise} one needs to add to a $d$-tuple of quantum effects to guarantee that they become {\em jointly measurable}; see \cite{BN18}. 

The case $d = 2$ in the above theorem was first obtained in \cite[Section 14]{HKMS19} and \cite[Section 7]{DDSSa} using other methods. 
It also follows from the following result. 
\begin{theorem}[Theorem 5.8, \cite{FNT17}]\label{thm:sym_dil}
Let $K \subseteq \bR^d$ be a symmetric convex set, i.e.  $K = -K$. 
Then 
\[
\Wmax{}(K) \subseteq d \Wmin{}(K). 
\]
\end{theorem}
The above result was originally proved by Fritz, Netzer and Thom \cite{FNT17} for cones {\em with a symmetric base}; to pass between the language of convex bodies and that of cones, one may use the gadget developed in \cite[Section 7]{PSS18}. 
In \cite[Theorem 4.5]{PSS18} it was observed that Theorem \ref{thm:sym_dil} is also a consequence of the methods of \cite[Section 7]{DDSSa} together with some classical results in convex geometry. 

Already in \cite[Lemma 3.1]{EW97} it was observed that there is only one matrix convex $\cS$ with $\cS_1 = [a,b] \subset \bR$, namely the {\bf{matrix interval}} given by $\cS_n = \{X \in (M_n)_{\text{sa}} : aI_n \leq X \leq b I_n \}$. 
Said differently, $\Wmax{}([a,b]) = \Wmin{}([a,b])$. 
It is natural to ask whether there exists any other convex body (i.e., a compact convex set) $K$ with the property that $\Wmin{}(K) = \Wmax{}(K)$. 

\begin{theorem}
Let $K \subseteq \bR^d$ be a convex body. 
Then $\Wmax{}(K) = \Wmin{}(K)$ if and only if $K$ is a simplex, that is, if $K$ is the convex hull of a set of affinely independent points. 
In fact, $\Wmax{2}(K) = \Wmin{2}(K)$ already implies that $K$ is a simplex. 
\end{theorem}
The result that the equality $\Wmax{}(K) = \Wmin{}(K)$ is equivalent to $K$ being a simplex was first obtained by Fritz, Netzer and Thom \cite[Corollary 5.3]{FNT17} for polyhedral cones. 
In \cite[Theorem 4.1]{PSS18} it was proved for general convex bodies, and it was also shown that one does not need to check equality $\Wmax{n}(K) = \Wmin{n}(K)$ for all $n$ in order to deduce that $K$ is a simplex --- it suffices to check this for some $n \geq 2^{d-1}$. 
For {\em simplex pointed} convex bodies, it was shown that $\Wmax{2}(K) = \Wmin{2}(K)$ already implies that $K$ is a simplex \cite[Theorem 8.8]{PSS18}. 
Huber and Netzer later obtained this for all polyhedral cones \cite{HN+}, and finally Aubrun, Lami, Palazuelos and Plavala proved the result for all cones \cite[Corollary 2]{ALPP+}.

\begin{remark}
The minimal matrix convex set over a ``commutative" convex set $K \subseteq \bR^d$ can be considered as the {\em matrix convex hull} of $K$. 
There are some variations on this theme. 
Helton, Klep and McCullough studied the matrix convex hull of {\em free semialgebraic sets} \cite{HKM16}. 
Instead of $\Wmin{}(K)$ and $\Wmax{}(K)$, which are the minimal and maximal matrix convex sets with prescribed first level, one can also discuss the minimal and maximal matrix convex sets with a prescribed $k$th level (see \cite{Kriel}, or \cite{Xhabli1, Xhabli2} for the version of this notion in the framework of operator systems). 
In the recent paper \cite{PasPau+}, Passer and Paulsen define, given a matrix convex set $\cS$, the minimal and maximal matrix convex sets $\cW^{\text{min-}k}(\cS)$ and $\cW^{\text{max-}k}(\cS)$ such that $\cW^{\text{min-}k}_k(\cS) = \cW^{\text{max-}k}_k(\cS) = \cS_k$, and they utilize quantitative measures of discrepancy between $\cW^{\text{min-}k}(\cS)$, $\cW^{\text{max-}k}(\cS)$ and $\cS$ to glean information on the operator system corresponding to $\cS$; unfortunately, these results are beyond the scope of this survey. 
The paper \cite{PasPau+} also ties together some of the earlier work in this direction, so it is a good place to start if one is interested in this problem. 
\end{remark}

\subsection{Further dilation results} 

There are many other interesting dilation results in \cite{DDSS17,DOPhD,FNT17,HKMS19,Passer,PSS18}. 
In this section I will review a few more.

\begin{problem}\label{prob:dilconst}
Fix $d \in \bN$.
What is the smallest constant $C_d$ such that for every $d$-tuple of contractions $A$, there exists a $d$-tuple of commuting normal operators $B$, such that $A \prec B$ holds with $\|B_i\| \leq C_d$ for all $i$?
\end{problem}

First, we note that the sharp dilation constant $\theta(\ol{\bB}_{\infty,d}) = \sqrt{d}$ obtained in Theorem \ref{thm:sym_dil} implies the following result, which is a solution to Problem \ref{prob:dilconst} in the selfadjoint setting. 

\begin{theorem}[Theorem 6.7, \cite{PSS18}]\label{thm:SADil}
For every $d$-tuple $A = (A_1, \ldots, A_d)$ of selfadjoint contractions, there exists a $d$-tuple of commuting selfadjoints $N = (N_1, \ldots, N_d)$ with $\|N_i\| \leq \sqrt{d}$ for $i=1,\ldots, d$, such that $A \prec N$. 
Moreover, $\sqrt{d}$ is the optimal constant for selfadjoints. 
\end{theorem}
It is interesting to note that one of the proofs of the above theorem goes through a concrete construction of the dilation. 
The nonselfadjoint version of Problem \ref{prob:dilconst} is more difficult, and it does not correspond to an inclusion problem of some $\Wmax{}$ in some $\Wmin{}$. 
The best general result in the nonselfadjoint case is the following theorem obtained by Passer. 
\begin{theorem}[Theorem 4.4, \cite{Passer}]
For every $d$-tuple $A = (A_1, \ldots, A_d)$ of contractions, there exists a $d$-tuple of commuting normal operators $N = (N_1, \ldots, N_d)$ with $\|N_i\| \leq \sqrt{2d}$ for $i=1,\ldots, d$, such that $A \prec N$. 
\end{theorem}
Thus
\[
\sqrt{d} \leq C_d \leq \sqrt{2d}. 
\]
In the next section we will improve the lower bound in the case $d=2$. 

Helton, Klep, McCullough and Schweighofer obtained a remarkable result, which is analogous to Theorem \ref{thm:SADil}, but in which the dilation constant is {\em independent of the number of operators $d$} \cite{HKMS19}. 
Following Ben-Tal and Nemirovski \cite{BTN02}, Helton et al. defined a constant $\vartheta(n)$ as follows: 
\[
\frac{1}{\vartheta(n)} = \min \left\{\int_{\partial \bB_n} \left|\sum_{i=1}^n a_i x_i^2 \right| d\mu(x) \,\, \colon \, \sum_{i=1}^n |a_i| = 1 \right\} 
\]
where $\mu$ is the uniform probability measure on the unit sphere $\partial \bB_n \subset \bR^n$. 

\begin{theorem}[Theorem 1.1, \cite{HKMS19}]\label{thm:HKMS}
Fix $n$ and and a real $n$-dimensional Hilbert space $\cH$. 
Let $\cF\subseteq B(\cH)_{\text{sa}}$ be a family of selfadjoint contractions. 
Then there exists a real Hilbert space $\cK$, an isometry $V : \cH \to \cK$, and a commuting family $\cC$ in the unit ball of $B(\cK)_{\text{sa}}$ such that for every contraction $A \in \cF$, there exists $N \in \cC$ such that 
\[
\frac{1}{\vartheta(n)} A = V^* N V  . 
\]
Moreover, $\vartheta(n)$ is the smallest constant such that the above holds for all finite sets of contractive selfadjoints $\cF \subseteq B(\cH)_{\text{sa}}$. 
\end{theorem}
Note the difference from Theorem \ref{thm:SADil}: the dimension of matrices is fixed at $n \times n$, but the number of matrices being simultaneously dilated is {\bf not} fixed. 
In other words, the constant $\vartheta(n)$ depends only on the size of the matrices being dilated (in fact, it is shown that $n$ can be replaced with the maximal rank of the matrices being dilated). 
It is also shown that  
\[
\vartheta(n) \sim \frac{\sqrt{\pi n\,\,}}{2} .
\] 
In the next subsection I will explain the motivation for obtaining this result.

\subsection{An application: matricial relaxation of spectrahedral inclusion problems} 

Any dilation result, such as Theorem \ref{thm:SADil} or Theorem \ref{thm:HKMS}, leads to a von Neumann type inequality. 
For example, if $A$ is a $d$-tuple of selfadjoint contractions, then by Theorem \ref{thm:SADil}, for every matrix valued polynomial $p$ of degree at most one, we have the following inequality: 
\[
\|p(A)\| \leq \sup\left\{\|p(z)\| : z \in [-\sqrt{d}, \sqrt{d}]^d \right\}. 
\]
This result is by no means trivial, but it is the kind of application of dilation theory that we have already seen above several times. 

We will now see a deep application of Helton, Klep, McCullough and Schweighofer's theorem (Theorem \ref{thm:HKMS}) that is of a different nature from the applications that we have seen hitherto, and is the main motivation for the extraordinary paper \cite{HKMS19}. 
The application builds on earlier work of Ben-Tal and Nemirovski \cite{BTN02} in control theory and optimization, related to what is sometimes called {\em the matrix cube problem}. 
I will give a brief account; the reader who seeks a deeper understanding should start with the introductions of \cite{BTN02} and \cite{HKMS19}. 

In the analysis of a linear controlled dynamical system (as in \cite{BTN02}), one is led to the problem of deciding whether the cube $[-1,1]^d$ is contained in the spectrahedron $\cD^{\text{sa}}_A(1)$, for a given a $d$-tuple of selfadjoint $n \times n$ matrices $A_1, \ldots, A_d$; this is called {\em the matrix cube problem}. 
More generally, given another $d$-tuple of selfadjoint matrices $B_1, \ldots, B_d$, it is of practical interest to solve the spectrahedral inclusion problem, that is, to be able to decide whether 
\[
\cD^{\text{sa}}_B(1) \subseteq \cD^{\text{sa}}_A(1). 
\]
Note that the matrix cube problem is a special case of the spectrahedral inclusion problem, since $[-1,1]^d = \cD^{\text{sa}}_C(1)$ for the $d$-tuple of $2d \times 2d$ diagonal matrices $C_1 = \diag(1, -1, 0,\ldots, 0)$, $C_2 = \diag(0, 0, 1, -1, 0, \ldots, 0)$, \ldots , $C_d = (0,\ldots, 0, 1, -1)$. 
The free spectrahedron $\cD^{\text{sa}}_C$ determined by $C$ is nothing but the nc set consisting of all $d$-tuples of selfadjoint contractions. 

The problem of deciding whether one spectrahedron is contained in another is a {\em hard} problem. 
In fact, deciding whether or not $[-1,1]^d \subseteq \cD^{\text{sa}}_A(1)$ has been shown to be NP hard (note that the naive solution of checking whether all the vertices of the cube are in $\cD^{\text{sa}}_A(1)$ requires one to test the positive semidefiniteness of $2^d$ matrices). 
However, Ben-Tal and Nemirovski introduced a tractable {\em relaxation} of this problem \cite{BTN02}. 
In \cite{HKM13}, Helton, Klep and McCullough showed that the relaxation from \cite{BTN02} is equivalent to the {\em free relaxation} $\cD^{\text{sa}}_C \subseteq \cD^{\text{sa}}_A$, and the subsequent work in \cite{HKMS19} gives a full understanding of this relaxation, including sharp estimates of the error bound. 

Let's take a step back. 
Fix two $d$-tuples of selfadjoint matrices $A$ and $B$. 
We mentioned that the problem of determining whether $\cD^{\text{sa}}_B(1) \subseteq \cD^{\text{sa}}_A(1)$ is hard. 
In \cite{HKM13}, it was observed that the {\em free relaxation}, that is, the problem $\cD^{\text{sa}}_B \subseteq \cD^{\text{sa}}_A$ is tractable. 
Indeed, as explained after Theorem \ref{thm:UCPinterpolation}, the inclusion $\cD^{\text{sa}}_B \subseteq \cD^{\text{sa}}_A$ is equivalent to the UCP interpolation problem, that is, to the existence of a UCP map sending $B_i$ to $A_i$ for all $i = 1, \ldots, d$ \cite[Theorem 3.5]{HKM13}. 
Now, the UCP interpolation problem can be shown to be equivalent to the solution of a certain {\em semidefinite program} \cite[Section 4]{HKM13}. 
In practice, there are numerical software packages that can solve such problems efficiently. 

So we see that instead of solving the matrix cube problem $[-1,1]^d \subseteq \cD^{\text{sa}}_A(1)$, one can solve the free relaxation $\cD^{\text{sa}}_C \subseteq \cD^{\text{sa}}_A$. 
Now, the whole point of the sharp results in \cite{HKMS19} is that they give a tight estimate of how well the tractable free relaxation approximates the hard matrix cube problem. 
To explain this, we need the following lemma. 

\begin{lemma}
Suppose that $A$ is a $d$-tuple of selfadjoint $n\times n$ matrices. 
Then, 
\[
[-1,1]^d \subseteq \cD^{\textup{sa}}_A(1) \Rightarrow \cD^{\textup{sa}}_C \subseteq \vartheta(n) \cD^{\textup{sa}}_A.
\]
\end{lemma}
\begin{proof}
Suppose that $[-1,1]^d \subseteq \cD^{\text{sa}}_A(1)$. 
If $X \in \cD^{\text{sa}}_C(n)$, then by Theorem \ref{thm:HKMS}, $X \prec  \vartheta(n) N$, where $N$ is a normal tuples and $\sigma(N) \subseteq [-1,1]^d \subseteq \cD^{\text{sa}}_A(1)$.  
So 
\[
\sum X_j \otimes A_j \prec \vartheta(n) \sum N_j \otimes A_j \leq \vartheta(n) I , 
\]
where the last inequality follows easily by the spectral theorem and the assumption $[-1,1]^d \subseteq \cD^{\text{sa}}_A(1)$. 
\end{proof}
Finally, we can now understand how to give an approximate solution to the matrix cube problem. 
Simply, one tests whether $\cD^{\text{sa}}_C \subseteq \vartheta(n) \cD^{\text{sa}}_A$, which is a tractable problem. 
If the inclusion holds, then it holds at every level and in particular $[-1,1]^d \subseteq \vartheta(n) \cD^{\text{sa}}_A(1)$. 
If not, then, using the lemma, we conclude that $[-1,1]^d \subsetneq \cD^{\text{sa}}_A(1)$. 
Thus, we are able to determine the containment of $[-1,1]^d$ in $\cD^{\text{sa}}_A(1)$, up to a multiplicative error of $\vartheta(n)$, which is known to high precision, and independent of $d$.

\section{Dilation of \texorpdfstring{$q$}{q}-commuting unitaries}\label{sec:q}

This section is dedicated to presenting the results Gerhold and Shalit from \cite{GS19}, on dilations of $q$-commuting unitaries.  

Let $\theta\in \mathbb R$ and write $q=e^{i\theta}$. 
If $u$ and $v$ are two unitaries that satisfy $vu = quv$, then we say that $u$ and $v$ are {\bf{$q$-commuting}}. 
We denote by $\cA_\theta$ the universal C*-algebra generated by a pair of $q$-commuting unitaries, and we call $\cA_\theta$ a {\bf{rational/irrational rotation C*-algebra}} if $\frac{\theta}{2\pi}$ is rational/irrational respectively.
We shall write $u_\theta, v_\theta$ for the generators of $\cA_\theta$.
The rotation C*-algebras have been of widespread interest ever since they were introduced by Rieffel \cite{Rie81}. 
A good reference for this subject is Boca's book \cite{BocBook}.

In an attempt to make some progress in our understanding of the general constant $C_d$ from Problem \ref{prob:dilconst}, Malte Gerhold and I studied a certain refinement of that problem which is of independent interest.
Instead of dilating arbitrary tuples of contractions, we considered the task of dilating pairs of unitaries $u,v$ that satisfy the $q$-commutation relation $vu = quv$, and studied the dependence of the dilation constant on the parameter $q$.
In the context of Problem \ref{prob:dilconst}, it is worth noting that, by a result of Buske and Peters \cite{BP98} (see also \cite{KM19}), every pair of $q$-commuting contractions has a $q$-commuting unitary power dilation; therefore, this work has implications to all pairs of $q$-commuting operators. 
Surprisingly, our dilation results also have implications for the continuity of the norm and the spectrum of the almost Mathieu operator from mathematical physics (this application will be discussed in the final section). 

For every $\theta \in \bR$ we define the optimal dilation constant
\[
c_\theta:=\inf\{c > 1\mid (u_\theta, v_\theta) \prec c (U,V) \textrm{ where } U , V \textrm{ are commuting unitaries}\}.
\]
We note that the infimum is actually a minimum, and that it is equal to the infimum of the constants $c$ that satisfy: for every $q$-commuting pair of unitaries $U,V$ there exists a commuting normal dilation $M,N$ such that $\|M\|,\|N\|\leq c$ (see \cite[Proposition 2.3]{GS19}). 
Thus, $c_\theta$ is a lower bound for the constant $C_2$ from Problem \ref{prob:dilconst}.

\subsection{Continuity of the dilation scale}

\begin{theorem}[Theorem 3.2, \cite{GS19}]\label{thm:qqtag}
Let $\theta, \theta' \in \mathbb R$, set $q=e^{i\theta}, q'=e^{i\theta'}$, and put $c = e^{\frac{1}{4}|\theta - \theta'|}$.
Then for any pair of
$q$-commuting unitaries $U,V$ there exists a pair of $q'$-commuting unitaries $U',
V'$ such that $cU', cV'$ dilates $U, V$.
\end{theorem}

\begin{proof} 
The proof makes use of the Weyl operators on symmetric Fock space (see \cite[Section 20]{Par12}). 
For a Hilbert space $H$ let $H^{\otimes_s k}$ be the $k$-fold symmetric tensor product of $H$, and let
\[
\Gamma(H):=\bigoplus_{k=0}^{\infty} H^{\otimes_s k}
\]
be the symmetric Fock space over $H$.
The {\bf exponential vectors} 
\[
e(x):=\sum_{k=0}^{\infty}\frac{1}{\sqrt{k!}} x^{\otimes k} \quad , \quad  x\in H, 
\]
form a linearly independent and total subset of $\Gamma(H)$.
For $z\in H$ we define the {\bf Weyl unitary} $W(z)\in B(\Gamma(H))$ which is determined by
\[
W(z) e(x)=e(z+x) \exp\left(-\frac{\|z\|^2}{2} - \langle x,z\rangle \right)
\]
for all exponential vectors $e(x)$.

Consider Hilbert spaces $H\subset K$ with $p$ the projection onto $H$,
and the symmetric Fock spaces $\Gamma(H) \subset \Gamma(K)$ with $P$ the
projection onto $\Gamma(H)$.
We write $p^\perp $ for the projection onto the orthogonal complement $H^\perp $.
Note that for exponential vectors we have
$Pe(x)=e(px)$.
For every $y,z\in K$, the Weyl unitaries $W(y), W(z)$ satisfy:
\begin{enumerate}
\item $W(z)$ and $W(y)$ commute up to the phase factor $e^{2i \im \langle y,z\rangle}$.
\item $PW(z)\big|_{\Gamma(H)}= e^{-\frac{\|p^\perp  z\|^2}{2}}W(pz)$, so it is a scalar multiple of a unitary on $\Gamma(H)$.
\item $PW(z)\big|_{\Gamma(H)}$ and $PW(y)\big|_{\Gamma(H)}$ commute up to a phase factor $e^{2i \im \langle py,z\rangle}$.
\end{enumerate}
In \cite{GS19} it is shown that, assuming without loss that $\theta > \theta'$, things can be arranged so that there are two linearly independent vectors $z,y$ so that $pz$ and $py$ are linearly independent, and such that
\begin{enumerate}
\item $p^\perp y=-ip^\perp z$,
\item $\theta'=2\im \langle y,z\rangle$,
\item $\theta= 2\im \langle py,z\rangle$.
\end{enumerate}
Then we get $q'$-commutation of
$W(z)$ and $W(y)$, $q$-commutation of the operators $PW(z)\big|_{\Gamma(H)}$ and $PW(y)\big|_{\Gamma(H)}$,  and
\[
\theta-\theta' = -2\im \langle p^\perp  y, z\rangle = 2\|p^\perp z\|^2 = 2\|p^\perp y\|^2,
\]
so
\[
\left\|PW(z)\big|_{\Gamma(H)}\right\|=\left\|PW(y)\big|_{\Gamma(H)}\right\|=e^{-\frac{\|p^\perp  y\|^2}{2}}=e^{-\frac{|\theta-\theta'|}{4}} .
\]
Now if we put
\[
U= e^{\frac{|\theta-\theta'|}{4}}PW(z)\big|_{\Gamma(H)} \,\, , \,\, V= e^{\frac{|\theta-\theta'|}{4}} PW(y)\big|_{\Gamma(H)},
\]
and
\[
U'=W(z) \,\, , \,\, V'=W(y)
\]
then we get the statement for this particular $q$-commuting pair $U,V$.
Since the Weyl unitaries give rise to a universal representation of $\cA_\theta$, the general result follows (see \cite[Proposition 2.3]{GS19}). 
\end{proof}

From the above result we obtained continuity of the dilation scale. 
\begin{corollary}[Corollary 3.4, \cite{GS19}]\label{cor:ctheta}
The optimal dilation scale $c_\theta$
depends Lipschitz continuously on $\theta$. 
More precisely, for all $\theta,\theta'\in\mathbb{R}$ we have 
\[
\left|c_\theta-c_{\theta'}\right| \leq 0.39\left|\theta-\theta'\right| .
\]
\end{corollary}

\subsection{The optimal dilation scale} 

The main result of \cite{GS19} is the following theorem. 

\begin{theorem}[Theorems 6.3 and 6.4, \cite{GS19}]\label{thm:optdil_gen}
Let $\theta,\theta' \in \bR$, $q = e^{i\theta}$, $q' = e^{i\theta'}$, and put $\gamma = \theta'-\theta$. 
The smallest constant $c_{\theta,\theta'}$ such that every pair of $q$-commuting unitaries can be dilated to $c_{\theta,\theta'}$ times a pair of $q'$-commuting unitaries is given by 
\[
c_{\theta,\theta'} = \frac{4}{\|u_\gamma + u_\gamma^* + v_\gamma + v_\gamma^*\|} .
\]
In particular, for every $\theta \in \bR$,
\[
c_\theta = \frac{4}{\|u_\theta+u_\theta^*+v_\theta+v_\theta^*\|}.
\]
\end{theorem} 
\begin{proof}
Since it is a nice construction that we have not yet seen, let us show just that the value of $c_{\theta,\theta'}$ is no bigger than $\frac{4}{\|u_\gamma + u_\gamma^* + v_\gamma + v_\gamma^*\|}$; for the optimality of the dilation constant we refer the reader to \cite{GS19} (the formula for $c_\theta = c_{\theta,0}$ follows, since it is not hard to see that $c_{\theta} = c_{-\theta}$). 

Represent $C^*(U,V)$ concretely on a Hilbert space $\cH$.
Let $u_\gamma, v_\gamma$ be the universal generators of $\cA_\gamma$ and put $h_\gamma:=u_\gamma+u_\gamma^*+v_\gamma+v_\gamma^*$.
We claim that there exists a state $\varphi$ on $\cA_\gamma$ such that $|\varphi(u_\gamma)| = |\varphi(v_\gamma)| = \frac{\|h_\gamma\|}{4}$ (for the existence of such a state, see \cite{GS19}).
Assuming the existence of such a state, we define
\[
U' = U \otimes \frac{\pi(u_\gamma)}{\varphi(u_\gamma)} \quad , \quad V' = V \otimes \frac{\pi(v_\gamma)}{\varphi(v_\gamma)}.
\]
on $\cK = \cH \otimes \cL$, where $\pi\colon A_\gamma\to B(\cL)$ is the GNS representation of $\varphi$. 
These are $q'$-commuting scalar multiples of unitaries, and they have norm $\frac{4}{\|h_\gamma\|}$. 
By construction, there exists a unit vector $x \in \cL$ such that $\varphi(a) = \langle  \pi(a) x , x \rangle$ for all $a \in A_\gamma$.
Consider the isometry $W \colon \cH \to \cH \otimes \cL$ defined by
\[
W h = h \otimes x \quad , \quad h \in \cH.
\]
Then 
\[
W^* U' W = \frac{1}{\varphi(u_\gamma)} \langle \pi(u_\gamma)  x, x\rangle U = U
\] 
and 
\[
W^*V' W = \frac{1}{\varphi(v_\gamma)} \langle \pi(v_\gamma) x, x\rangle V = V, 
\] 
and the proof of the existence of a dilation is complete.
\end{proof}

The operator $h_\theta = u_\theta + u^*_\theta + v_\theta + v_\theta^*$ is called the {\em almost Mathieu operator}, and it has been intensively studied by mathematical physicists, before and especially after Hofstadter's influential paper \cite{Hof76} (we will return to it in the next section). 
However, the precise behaviour of the norm $\|h_\theta\|$ as a function of $\theta$ is still not completely understood. 
We believe that the most detailed analysis is contained in the paper \cite{BZ05}. 

In \cite[Section 7]{GS19} we obtained numerical values for $c_\theta = 4/\|h_\theta\|$ for various $\theta$.
We calculated by hand $c_{\frac{4}{5}\pi} \approx 1.5279$,
allowing us to push the lower bound $C_2 \geq 1.41...$ to $C_2 \geq 1.52$.
We also made some numerical computations, which lead to an improved estimate $C_2 \geq \max_\theta c_\theta \geq 1.5437$. 
The latter value is an approximation of the constant $c_{\theta_s}$ attained at the {\em silver mean} $\theta_s=\frac{2\pi}{\gamma_s}=2\pi(\sqrt{2}-1)$ (where $\gamma_s=\sqrt{2}+1$ is the {\em silver ratio})  which we conjecture to be the angle where the maximum is attained. 
However, we do not expect that the maximal value of $c_\theta$ will give a tight lower approximation for $C_2$. 
Determining the value of $C_2$ remains an open problem.

\subsection{An application: continuity of the spectrum of almost Mathieu operators}

The almost Mathieu operator $h_\theta = u_\theta + u^*_\theta + v_\theta + v^*_\theta$, which appears in the formula $c_\theta = \frac{4}{\|h_\theta\|}$, arises as the Hamiltonian in a certain mathematical model describing an electron in a lattice under the influence of a magnetic field; see Hofstadter \cite{Hof76}. 
This operator has been keeping mathematicians and physicists busy for more than a generation. 
Hofstadter's paper included a picture that depicts the spectrum (computed numerically) of $h_\theta$ for various values of $\theta$, famously known as the {\em Hofstadter butterfly} (please go ahead and google it). 
From observing the Hofstadter butterfly, one is led to making several conjectures. 

First and foremost, it appears that the spectrum of $h_\theta$ varies continuously with $\theta$; since $\theta$ is a physical parameter of the system studied, and the spectrum is supposed to describe possible energy levels, any other possibility is unreasonable. 
There are other natural conjectures to make, suggested just by looking at the picture. 
The most famous one is perhaps what Barry Simon dubbed as the {\em Ten Martini Problem}, which asks whether the spectrum is a Cantor set for irrational angles. 
This problem was settled (in greater generality) by Avila and Jitomirskaya (see \cite{AJ09} for the conclusive work as well as for references to earlier work).  

The continuity of the spectrum $\sigma(h_\theta)$ is a delicate problem that attracted a lot of attention. 
For example, in \cite{CEY} Choi, Elliott, and Yui showed that the spectrum $\sigma(h_\theta)$ of $h_\theta$ depends H{\"o}lder continuously (in the Hausdorff metric) on $\theta$, with H{\"o}lder exponent $1/3$.
This was soon improved by Avron, Mouche, and Simon to H{\"o}lder continuity with exponent $1/2$ \cite{AMS90}.
The $1/2$-H{\"o}lder continuity of the spectrum also follows from a result of Haagerup and R{\o}rdam, who showed that there exist $1/2$-H{\"o}lder {\em norm} continuous paths $\theta \mapsto u_\theta \in B(\cH)$, $\theta \mapsto v_\theta \in B(\cH)$ \cite[Corollary 5.5]{HR95}.

As an application of our dilation techniques, we are able to recover the best possible continuity result regarding the spectrum of the operator $h_\theta$. 
This result is not new, but our proof is new and simple, and I believe that it is a beautiful and exciting application of dilation theory with which to close this survey. 
The following theorem also implies that the rotation C*-algebras form a continuous field of C*-algebras, a result due to Elliott \cite{Ell82}. 
Our dilation methods can also be used to recover the result of Bellisard \cite{Bel94}, that the norm of $h_\theta$ is a Lipschitz continuous function of $\theta$. 

\begin{theorem}
Let $p$ be a selfadjoint $*$-polynomial in two noncommuting variables.
Then the spectrum $\sigma(p(u_\theta,v_\theta))$ of $p(u_\theta,v_\theta)$ is $\frac{1}{2}$-H{\"o}lder continuous in $\theta$ with respect to the Hausdorff distance for compact subsets of $\mathbb R$.
\end{theorem}
\begin{proof} 
Let us present the idea of the proof for the most important case 
\[
p(u_\theta,v_\theta) = h_\theta = u_\theta+ u_\theta^* + v_\theta + v_\theta^* ,
\]
without going into the details of H{\"o}lder continuity. 
The idea is that, due to Theorem \ref{thm:qqtag},  when $\theta \approx \theta'$ we have the dilation $(u_\theta,v_\theta) \prec (cu_{\theta'}, cv_{\theta'})$ with $c  = e^{\frac{1}{4}|\theta - \theta'|} \approx 1$.  
Thus, 
\[
cu_{\theta'} = \begin{pmatrix} u_\theta & x \\ y & z  \end{pmatrix} 
\] 
and so $x$ and $y$ must be small, to be precise, 
\[
\|x\|,\|y\| \leq \sqrt{c^2-1} \approx 0 .
\]
A similar estimate holds for the off diagonal block of $cv_{\theta'}$ which dilates $v_\theta$. 
By a basic lemma in operator theory, for any selfadjoint operators $a$ and $b$, the Hausdorff distance between their spectra is bounded as follows:
\[
d(\sigma(a) ,\sigma(b)) \leq \|a-b\|.
\] 
We have $h_\theta = u_\theta + u_\theta^* + v_\theta + v_\theta^*$, and so 
\[
ch_{\theta'} = \begin{pmatrix} h_\theta & * \\ * & *  \end{pmatrix}  \approx \begin{pmatrix} h_\theta & 0 \\ 0 & *  \end{pmatrix} ,
\] 
because the off diagonal blocks have small norm, and therefore
\[
\sigma(h_\theta) \subseteq \sigma\left( \begin{pmatrix} h_\theta & 0 \\ 0 & *  \end{pmatrix} \right)  \approx  \sigma(c h_{\theta'})  \approx \sigma(h_{\theta'}). 
\]
In the same way one shows that $\sigma(h_{\theta'})$ is approximately contained in $\sigma(h_\theta)$, and therefore the Hausdorff distance between the spectra is small. 
\end{proof} 

It is interesting to note that the above proof generalizes very easily to higher dimensional noncommutative tori. 
Determining the precise dilation scales for higher dimensional noncommutative tori remains an open problem.

%

\end{document}